\documentclass{amsart}
\usepackage{amsmath, amssymb, amsthm, mathrsfs}
\usepackage[all]{xy}

\theoremstyle{plain} 
 \newtheorem{thm}{Theorem}[section]
 \newtheorem{lem}[thm]{Lemma}
 \newtheorem{cor}[thm]{Corollary}
 \newtheorem{prop}[thm]{Proposition}
 \newtheorem{claim}[thm]{Claim}
\theoremstyle{definition}
  \newtheorem{defn}[thm]{Definition}

  \newtheorem{notation}[thm]{Notation}

\theoremstyle{remark}
  \newtheorem{rem}[thm]{Remark}

\newcommand{\comm}{{\rm Comm}}
\newcommand{\stab}{{\rm Stab}}
\newcommand{\aut}{{\rm Aut}}
\newcommand{\ad}{{\rm Ad}}
\newcommand{\isom}{{\rm Isom}}
\newcommand{\bs}{{\rm BS}}
\newcommand{\qn}{{\rm QN}}
\newcommand{\m}{\mathfrak{m}}
\newcommand{\n}{\mathfrak{n}}
\newcommand{\D}{\mathfrak{D}}
\newcommand{\I}{\mathfrak{I}}
\newcommand{\R}{\mathbb{R}}
\newcommand{\Rm}{\mathbb{R}_+^{\times}}
\newcommand{\cal}{\mathcal}

\newcommand{\ci}[2]{\cite[#1]{#2}}
\renewcommand{\c}{\curvearrowright}

\begin{document}

\title[Invariants of orbit equivalence relations]{Invariants of orbit equivalence relations and Baumslag-Solitar groups}
\author{Yoshikata Kida}
\address{Department of Mathematics, Kyoto University, 606-8502 Kyoto, Japan}
\email{kida@math.kyoto-u.ac.jp}
\date{June 3, 2013}

\begin{abstract}
To an ergodic, essentially free and measure-preserving action of a non-amenable Baumslag-Solitar group on a standard probability space, a flow is associated.
The isomorphism class of the flow is shown to be an invariant of such actions of Baumslag-Solitar groups under weak orbit equivalence.
Results on groups which are measure equivalent to Baumslag-Solitar groups are also provided.
\end{abstract}

\maketitle


\section{Introduction}\label{sec-int}

The study of probability-measure-preserving actions of discrete countable groups through the associated orbit equivalence relations has been making progress since the seminal works \cite{cfw}, \cite{dye1}, \cite{dye2} and \cite{ow} on amenable groups.
We recommend the reader to consult \cite{furman-survey}, \cite{gab}, \cite{popa-survey}, \cite{shalom-survey} and \cite{vaes-survey} for recent development and related topics.
In this paper, we focus on actions of non-amenable Baumslag-Solitar groups, and introduce their invariants under orbit equivalence.
For two non-zero integers $p$, $q$, the {\it Baumslag-Solitar group} $\bs(p, q)$ is defined as the group with the presentation
\[\bs(p, q)=\langle \, a, t\mid ta^pt^{-1}=a^q\,\rangle.\]
After the appearance in \cite{bs} as a simple example of a non-Hopfian finitely presented group for certain $p$ and $q$, this group attracts attention in combinatorial group theory.
The isomorphism problem for these groups is solved by Moldavanski\u{\i} \cite{mold}, who shows that $\bs(p, q)$ and $\bs(r, s)$ are isomorphic if and only if there exists an integer $\varepsilon \in \{ \pm 1\}$ such that either $(p, q)=(\varepsilon r, \varepsilon s)$ or $(p, q)=(\varepsilon s, \varepsilon r)$.
The group $\bs(p, q)$ is amenable if and only if either $|p|=1$ or $|q|=1$.
We therefore assume $2\leq |p|\leq |q|$ throughout the paper.

A discrete group is assumed countable unless otherwise mentioned.
We mean by an {\it f.f.m.p}.\ action of a discrete group $G$ a measure-preserving Borel action of $G$ on a standard finite measure space $(X, \mu)$ such that the stabilizer of a.e.\ $x\in X$ in $G$ is trivial, where ``f.f.m.p." stands for ``essentially free and finite-measure-preserving".
The following equivalence relations among such actions are of our interest.

\begin{defn}
Let $G\c (X, \mu)$ and $H\c (Y, \nu)$ be ergodic f.f.m.p.\ actions of discrete groups.
We say that these actions are {\it orbit equivalent (OE)} if there exists a Borel isomorphism $f$ from a conull Borel subset of $X$ onto a conull Borel subset of $Y$ such that $f_*\mu$ and $\nu$ are equivalent and we have $f(Gx)=Hf(x)$ for a.e.\ $x\in X$.

More generally, we say that the actions $G\c (X, \mu)$ and $H\c (Y, \nu)$ are {\it weakly orbit equivalent (WOE)} if there exists a Borel isomorphism $f$ from a Borel subset $A$ of $X$ with $\mu(A)>0$ onto a Borel subset $B$ of $Y$ with $\nu(B)>0$ such that $f_*(\mu|_A)$ and $\nu|_B$ are equivalent and we have $f(Gx\cap A)=Hf(x)\cap B$ for a.e.\ $x\in A$.
\end{defn}

We put $\Gamma =\bs(p, q)$ with $2\leq |p|\leq |q|$ and define a homomorphism $\m \colon \Gamma \to \Rm$, called the {\it modular homomorphism} of $\Gamma$, by setting $\m(a)=1$ and $\m(t)=|q/p|$.
Let $\Gamma \c (X, \mu)$ be an ergodic f.f.m.p.\ action.
We now introduce the flow associated with this action, which is shown to be an invariant under orbit equivalence.
Let $\pi \colon (X, \mu)\to (Z, \xi)$ be the ergodic decomposition for the action of $\ker \m$ on $(X, \mu)$.
We have the canonical ergodic measure-preserving action of $\m(\Gamma)$ on $(Z, \xi)$.
The {\it flow associated with} the action $\Gamma \c (X, \mu)$ is defined as the action of $\R$ induced from the action of $\log \circ \m(\Gamma)$ on $(Z, \xi)$ through the isomorphism $\log \colon \Rm \to \R$.
We refer to Remark \ref{rem-type} for a detailed description of this flow.

\begin{thm}\label{thm-flow}
We set $\Gamma =\bs(p, q)$ and $\Lambda =\bs(r, s)$ with $2\leq |p|\leq |q|$ and $2\leq |r|\leq |s|$.
Let $\Gamma \c (X, \mu)$ and $\Lambda \c (Y, \nu)$ be ergodic f.f.m.p.\ actions.
If these two actions are WOE, then the flows associated with them are isomorphic.
\end{thm}

For a locally compact group $G$, two measure-preserving actions $G\c (Z, \xi)$ and $G\c (W, \omega)$ on measure spaces are called {\it isomorphic} if there exists a measurable isomorphism $f$ from a conull measurable subset of $Z$ onto a conull measurable subset of $W$ such that $f_*\xi$ and $\omega$ are equivalent and for any $g\in G$, we have $f(gz)=gf(z)$ for a.e.\ $z\in Z$.
In Corollary \ref{cor-conj}, we obtain a refinement of Theorem \ref{thm-flow} under an additional assumption.

Theorem \ref{thm-flow} is strongly inspired by the theory of the flow associated with an ergodic transformation of type ${\rm III}$, summarized in \cite{ho}.
This flow is defined as the Mackey range of the Radon-Nikodym cocycle for the transformation.
It is notable that the isomorphism class of the associated flow is a complete invariant of ergodic single transformations of type ${\rm III}$ under orbit equivalence.
This is due to Hamachi-Oka-Osikawa \cite{hoo} and Krieger \cite{krieger}, \cite{krieger-flow}.

Theorem \ref{thm-flow} is indeed a consequence of a rigidity result on the composition of the cocycle associated with the WOE and the homomorphism $\m$, which is formulated in terms of measure equivalence in Theorem \ref{thm-r}.
In Section \ref{sec-var}, we prove a variant of Furman's theorem on construction of a representation of a group which is measure equivalent to a given group.
Combining these results, we obtain the following:

\begin{thm}\label{thm-z}
We set $\Gamma =\bs(p, q)$ with $2\leq |p|<|q|$.
Suppose that an ergodic f.f.m.p.\ action $\Gamma \c (X, \mu)$ is WOE to a weakly mixing f.f.m.p.\ action $\Delta \c (Y, \nu)$ of a discrete group $\Delta$.
Then there exists a homomorphism from $\Delta$ onto $\mathbb{Z}$.
\end{thm}

Putting $\Gamma =\bs(p, q)$, we mean by an {\it elliptic} subgroup of $\Gamma$ a subgroup contained in the group generated by $a$ or in its conjugate in $\Gamma$.
Under the assumption that any non-trivial elliptic subgroup of $\Gamma$ acts ergodically, the parameters $p$, $q$ are shown to be invariant under WOE, while the associated flow is isomorphic to the action of $\R$ on $\R /(\log |q/p|)\mathbb{Z}$ by addition, and remembers at most the modulus $|q/p|$.

\begin{thm}\label{thm-woe}
We set $\Gamma =\bs(p, q)$ and $\Lambda =\bs(r, s)$ with $2\leq |p|<|q|$ and $2\leq |r|<|s|$.
Suppose that $q$ is not a multiple of $p$ and that $s$ is not a multiple of $r$.
Let $\Gamma \c (X, \mu)$ and $\Lambda \c (Y, \nu)$ be f.f.m.p.\ actions such that any non-trivial elliptic subgroup of $\Gamma$ acts ergodically and so does any non-trivial elliptic subgroup of $\Lambda$.
If the two actions $\Gamma \c (X, \mu)$ and $\Lambda \c (Y, \nu)$ are WOE, then there exists an integer $\varepsilon \in \{ \pm 1\}$ with $(p, q)=(\varepsilon r, \varepsilon s)$, that is, $\Gamma$ and $\Lambda$ are isomorphic.
\end{thm}

The assumption on the actions of non-trivial elliptic subgroups of $\Gamma$ and $\Lambda$ can be relaxed.
We refer to Theorem \ref{thm-mer} for a more general assertion.
In Section \ref{sec-oe}, when $\Gamma =\bs(p, q)$ with $2\leq |p|\leq |q|$, we construct two f.f.m.p.\ actions of $\Gamma$ which are WOE, but not conjugate.
The restrictions of the two actions to any non-trivial elliptic subgroup of $\Gamma$ are shown to be ergodic.
It is therefore impossible to obtain conjugacy of the two actions under the assumption in Theorem \ref{thm-woe}. 
As for Bernoulli shifts of Baumslag-Solitar groups, Popa's cocycle superrigidity theorem \cite{popa-gap} shows orbit equivalence rigidity of them (see Remark \ref{rem-ber}).

We now turn our attention to measure equivalence (ME).
This is an equivalence relation between discrete groups, introduced by Gromov \cite{gromov-as-inv} (see Definition \ref{defn-me}).
It is known that two discrete groups $G$ and $H$ are ME if and only if there exists an ergodic f.f.m.p.\ action of $G$ which is WOE to an ergodic f.f.m.p.\ action of $H$.

\begin{thm}\label{thm-bs-hyp}
We set $\Gamma =\bs(p, q)$ with $2\leq |p|<|q|$.
Let $\Delta$ be a discrete group having an infinite amenable normal subgroup $N$.
Suppose that the quotient $\Delta /N$ is a non-elementarily word-hyperbolic group.
Then $\Gamma$ and $\Delta$ are not ME.
\end{thm}

For any integers $r$, $s$ with $2\leq |r|=|s|$, the group $\bs(r, s)$ contains a finite index subgroup isomorphic to the direct product of $\mathbb{Z}$ with a non-abelian free group of finite rank.
It follows from Theorem \ref{thm-bs-hyp} that the group $\bs(p, q)$ with $2\leq |p|<|q|$ is not ME to $\bs(r, s)$.
The basic question asking whether the groups $\bs(p, q)$ with different parameters $p$, $q$ are ME or not remains unsolved.

This paper is organized as follows.
In Section \ref{sec-bs}, we review basic properties of Baumslag-Solitar groups and the Bass-Serre trees associated with them.
In Section \ref{sec-dmg}, starting with terminology of discrete measured groupoids, we introduce several notions related to them, index, local index, normality, quasi-normality and quotient.
In Section \ref{sec-var}, we review the aforementioned Furman's theorem and prove its variant.
In Section \ref{sec-ell}, given a finite-measure-preserving action of a Baumslag-Solitar group, we present an algebraic and sufficient condition for a subgroupoid of the associated groupoid to be elliptic.
Elliptic subgroupoids are shown to be preserved under an isomorphism between the groupoids associated with actions of Baumslag-Solitar groups.
In Section \ref{sec-mod}, we introduce the modular cocycle of Radon-Nikodym type and the local-index cocycle.
They are defined for the pair of a discrete measured groupoid and its quasi-normal subgroupoid.
If the pair comes from an action of a Baumslag-Solitar group and its restriction to an elliptic subgroup, then those two cocycles are related to the modular homomorphism.
In Section \ref{sec-mackey}, Theorems \ref{thm-flow}, \ref{thm-z} and \ref{thm-bs-hyp} are proved by using those cocycles. 
In Section \ref{sec-red}, Theorem \ref{thm-woe} is proved.
In Section \ref{sec-oe}, we discuss f.f.m.p.\ actions of a Baumslag-Solitar group which are WOE, but not conjugate.

In Appendix \ref{app}, we observe basic properties of an ergodic probability-measure-preserving action of a Baumslag-Solitar group, focusing on ergodicity of the action of an elliptic subgroup.
Those properties are used to relax ergodicity assumptions in theorems.
In Appendix \ref{sec-exotic}, we show that the groupoid associated with a certain action of a Baumslag-Solitar group has an amenable normal subgroupoid of infinite type.
We also show that the quotient groupoid is of type ${\rm III}$, and clarify relationship with the associated flow.


\section{Baumslag-Solitar groups}\label{sec-bs}

Contents of this section are discussed from a more general viewpoint in \cite[Section 2]{levitt}.
For a group $G$ and an element $g$ of $G$, let $\langle g \rangle$ denote the cyclic subgroup of $G$ generated by $g$.
Let $p$ and $q$ be integers with $2\leq |p|\leq |q|$, and set
\begin{equation*}\tag{$\star$}\label{pre}
\Gamma =\bs(p, q)=\langle\, a, t\mid ta^pt^{-1}=a^q\,\rangle.
\end{equation*}
The group $\Gamma$ is the HNN extension of the infinite cyclic group $\langle a\rangle$ relative to the isomorphism from $\langle a^p\rangle$ onto $\langle a^q\rangle$ sending $a^p$ to $a^q$.
Let $T=T_{\Gamma}$ denote the Bass-Serre tree associated with this HNN extension.
The set of vertices of $T$, denoted by $V(T)$, is defined to be $\Gamma /\langle a\rangle$.
The set of edges of $T$, denoted by $E(T)$, is defined to be $\Gamma/ \langle a^q\rangle$, and for each $\gamma \in \Gamma$, the edge corresponding to the coset $\gamma \langle a^q\rangle$ joins the two vertices corresponding to the cosets $\gamma \langle a\rangle$ and $\gamma t\langle a\rangle$.
We orient this edge so that the origin is the vertex corresponding to $\gamma \langle a\rangle$.
The group $\Gamma$ then acts on $T$ by orientation-preserving simplicial automorphisms.
The action of $\Gamma$ on $V(T)$ and that on $E(T)$ are both transitive.
Let $\aut(T)$ denote the group of orientation-preserving simplicial automorphisms of $T$ equipped with the standard Borel structure induced by the pointwise convergence topology.
Unless otherwise stated, we mean by the {\it Bass-Serre tree associated with} $\Gamma$ the oriented simplicial tree defined above for a fixed presentation of $\Gamma$ of the form (\ref{pre}).
We refer to \cite{serre} for the Bass-Serre theory.

Let $E_+$ denote the set of oriented edges of $T$ with the orientation defined above.
Let $E_-$ denote the set of oriented edges of $T$ consisting of the inverses of edges in $E_+$.
We set $\sigma(e)=1$ for each $e\in E_+$, and set $\sigma(f)=-1$ for each $f\in E_-$.
For a simplex $s$ of $T$, let $\Gamma_s$ denote the stabilizer of $s$ in $\Gamma$.
We pick two distinct vertices $v_0, v\in V(T)$.
Let $e_1,\ldots, e_n$ be the shortest sequence of oriented edges of $T$ such that the origin of $e_1$ is $v_0$; the terminal of $e_n$ is $v$; and for any $i=1,\ldots, n-1$, the terminal of $e_i$ and the origin of $e_{i+1}$ are equal.
We define $M$ as the maximal number in the set $\{ 0\} \cup \{\, \sum_{i=1}^k\sigma(e_i)\mid k=1,\ldots, n\, \}$, and define $m$ as the minimal number in it.
The equality
\[[\Gamma_{v_0}: \Gamma_{v_0}\cap \Gamma_v]=d_0|p_0|^m|q_0|^M\]
then holds, where we denote by $d_0>0$ the greatest common divisor of $p$ and $q$, and set $p_0=p/d_0$ and $q_0=q/d_0$.

For a group $G$ and a subgroup $H$ of $G$, we set
\[\comm_G(H)=\{\, g\in G\mid [H: gHg^{-1}\cap H]<\infty,\ [gHg^{-1}:gHg^{-1}\cap H]<\infty\,\}\]
and call it the {\it (relative) commensurator} of $H$ in $G$, which is a subgroup of $G$.
We say that an element $\gamma$ of $\Gamma$ is {\it elliptic} if it fixes a vertex of $T$.
This condition is equivalent to the equality $\comm_{\Gamma}(\langle \gamma \rangle)=\Gamma$ (see \cite[Lemma 2.1]{levitt}).
It follows that ellipticity of elements of $\Gamma$ is independent of the presentation (\ref{pre}).
We say that a subgroup $E$ of $\Gamma$ is {\it elliptic} if it fixes a vertex of $T$.
This condition is equivalent to that there exists an elliptic element of $\Gamma$ generating $E$.
It follows that ellipticity of subgroups of $\Gamma$ is also independent of the presentation (\ref{pre}).

We define the {\it modular homomorphism} $\m \colon \Gamma \to \Rm$ by sending $a$ to $1$ and $t$ to $|q/p|$.
This is indeed independent of the presentation (\ref{pre}) as shown below.
Any two non-trivial elliptic subgroups $E$, $F$ of $\Gamma$ are commensurable, that is, the intersection $E\cap F$ is of finite index in both $E$ and $F$.
Let $E$ be a non-trivial elliptic subgroup of $\Gamma$, and pick a non-neutral element $x$ of $E$.
For each $\gamma \in \Gamma$, since $\gamma E\gamma^{-1}$ and $E$ are commensurable, there exist non-zero integers $n$, $m$ with $\gamma x^n\gamma^{-1}=x^m$.
We define a map $f\colon \Gamma \to \Rm$ by $f(\gamma)=|m/n|$.
This is a well-defined homomorphism, and is independent of the choice of $E$ and $x$.
Since we have $f(a)=1$ and $f(t)=|q/p|$ by definition, the equality $\m =f$ holds.
The modular homomorphism $\m$ is thus defined independently of the presentation (\ref{pre}).


\section{Discrete measured groupoids}\label{sec-dmg}

\subsection{Terminology}

We recommend the reader to consult \cite{kechris} and \cite[Chapter XIII, Section 3]{take3} for basic knowledge of standard Borel spaces and discrete measured groupoids, respectively.
Unless otherwise stated, all sets and maps that appear in this paper are assumed to be Borel, and relations among Borel sets and maps are understood to hold up to sets of measure zero.

We refer to a standard Borel space with a finite positive measure as a {\it standard finite measure space}.
When the measure is a probability one, we refer to it as a {\it standard probability space}.
Let $(X, \mu)$ be a standard finite measure space.
Let $\cal{G}$ be a discrete measured groupoid on $(X, \mu)$.
Let $r, s\colon \cal{G}\rightarrow X$ denote the range and source maps, respectively.
For a Borel subset $A\subset X$ of positive measure, we define
\[(\cal{G})_A=\{\, g\in \cal{G}\mid r(g), s(g)\in A\, \}\]
and call it the restriction of $\cal{G}$ to $A$.
We denote by $\cal{G}A$ the saturation
\[\cal{G}A=\{\, r(g)\in X\mid g\in \cal{G},\ s(g)\in A\, \},\]
which is a Borel subset of $X$.
We say that $\cal{G}$ is {\it finite} if for a.e.\ $x\in X$, the set $r^{-1}(x)$ consists of at most finitely many points.
We say that $\cal{G}$ is of {\it infinite type} if for a.e.\ $x\in X$, the set $r^{-1}(x)$ consists of infinitely many points.

Let $\Gamma$ be a discrete group.
We say that a Borel action of $\Gamma$ on $(X, \mu)$ is {\it non-singular} if the action preserves the class of the measure $\mu$.
For a non-singular action $\Gamma \c (X, \mu)$, the associated discrete measured groupoid is denoted by $\Gamma \ltimes (X, \mu)$.
If $\mu$ is not specified, it is also denoted by $\Gamma \ltimes X$.
The range and source maps of $\Gamma \ltimes (X, \mu)$ are defined by $r(\gamma, x)=\gamma x$ and $s(\gamma, x)=x$, respectively, for $\gamma \in \Gamma$ and $x\in X$.
The product operation is defined by $(\gamma_1, \gamma_2x)(\gamma_2, x)=(\gamma_1\gamma_2, x)$ for $\gamma_1, \gamma_2\in \Gamma$ and $x\in X$.

Let $\Gamma \c (X, \mu)$ and $\Lambda \c (Y, \nu)$ be ergodic f.f.m.p.\ actions of discrete groups on standard finite measure spaces, and let $\cal{G}$ and $\cal{H}$ be the associated groupoids, respectively.
The two actions are WOE if and only if there are Borel subsets $A\subset X$ and $B\subset Y$ of positive measure such that $(\cal{G})_A$ and $(\cal{H})_B$ are isomorphic.

We recall the ergodic decomposition for a discrete measured groupoid.

\begin{thm}[\ci{Theorem 6.1}{hahn}]
Let $(X, \mu)$ be a standard finite measure space.
Let $\cal{G}$ be a groupoid on $X$ admitting the structure of a discrete measured groupoid on $(X, \mu)$.
Then there exist a standard finite measure space $(Z, \xi)$ and a Borel map $\pi \colon X\to Z$ such that $\pi_*\mu=\xi$; and for a.e.\ $z\in Z$, $\cal{G}$ admits the structure of an ergodic discrete measured groupoid on $(X, \mu_z)$, where $\mu_z$ is the probability measure on $X$ obtained through the disintegration $\mu =\int_Z\mu_zd\xi(z)$ with respect to $\pi$.

Moreover, if a standard finite measure space $(W, \omega)$ and a Borel map $\theta\colon X\to W$ satisfy these properties in place of $(Z, \xi)$ and $\pi$, respectively, then there exists a Borel isomorphism $f$ from a conull Borel subset of $Z$ onto a conull Borel subset of $W$ such that $\theta(x)=f\circ \pi(x)$ for a.e.\ $x\in X$. 
\end{thm}

We call the map $\pi \colon (X, \mu)\to (Z, \xi)$ the {\it ergodic decomposition} for $\cal{G}$.
For $z\in Z$, we put $X_z=\pi^{-1}(z)$ and denote by $\cal{G}_z$ the ergodic discrete measured groupoid $\cal{G}$ on $(X, \mu_z)$.
Since $\mu_z$ is supported on $X_z$, we often identify $\cal{G}_z$ with a discrete measured groupoid on $(X_z, \mu_z)$.

We introduce an invariant map for an action of a groupoid.
If the unit space of the groupoid consists of a single point, then it is a fixed point of the action of the group.

\begin{defn}\label{defn-inv-map}
Let $\cal{G}$ be a discrete measured groupoid on a standard finite measure space $(X, \mu)$, $\Gamma$ a discrete group, and $S$ a standard Borel space.
Suppose that we have a Borel action of $\Gamma$ on $S$ and a Borel homomorphism $\rho \colon \cal{G}\rightarrow \Gamma$.
A Borel map $\varphi \colon X\rightarrow S$ is said to be {\it $(\cal{G}, \rho)$-invariant} if we have the equality
\[\rho(g)\varphi(s(g))=\varphi(r(g))\]
for a.e. $g\in \cal{G}$.

More generally, if $A$ is a Borel subset of $X$ with positive measure and if a Borel map $\varphi \colon A\rightarrow S$ satisfies the above equality for a.e.\ $g\in (\cal{G})_A$, then we also say that $\varphi$ is {\it $(\cal{G}, \rho)$-invariant}.
\end{defn}

The following extendability of invariant maps is in general use.

\begin{lem}[\ci{Lemma 2.3}{kida-ama}]\label{lem-inv-ext}
Let $\cal{G}$ be a discrete measured groupoid on a standard finite measure space $(X, \mu)$.
Suppose that we have a Borel action of a discrete group $\Gamma$ on a standard Borel space $S$ and a Borel homomorphism $\rho \colon \cal{G}\rightarrow \Gamma$.
If $A$ is a Borel subset of $X$ with positive measure and if $\varphi \colon A\rightarrow S$ is a $(\cal{G}, \rho)$-invariant Borel map, then $\varphi$ extends to a $(\cal{G}, \rho)$-invariant Borel map from $\cal{G}A$ into $S$. 
\end{lem}

We refer to \cite{ar-book} for amenable groupoids.
A consequence of amenability in terms of the fixed point property is the following:

\begin{prop}[\ci{Theorem 4.2.7}{ar-book}]\label{prop-ame-basic}
Let $\mathcal{G}$ be a discrete measured groupoid on a standard finite measure space $(X, \mu)$, $\Gamma$ a discrete group, and $\rho \colon \mathcal{G}\rightarrow \Gamma$ a Borel homomorphism. 
Suppose that $\Gamma$ acts on a compact Polish space $K$ continuously. 
Let $M(K)$ denote the space of probability measures on $K$, on which $\Gamma$ naturally acts. 
If $\mathcal{G}$ is amenable, then there exists a $(\cal{G}, \rho)$-invariant Borel map from $X$ into $M(K)$.
\end{prop}


\subsection{Index for groupoids}

The index of a subrelation in a discrete measured equivalence relation is introduced by Feldman, Sutherland and Zimmer \cite{fsz}.
Their definition is directly generalized to that for discrete measured groupoids.
We define index for groupoids and present its basic properties.

Let $(X, \mu)$ be a standard finite measure space and $\cal{G}$ a discrete measured groupoid on $(X, \mu)$.
Let $r, s\colon \cal{G}\to X$ denote the range and source maps of $\cal{G}$, respectively.
Let $\cal{H}$ be a subgroupoid of $\cal{G}$.
For each $x\in X$, we define an equivalence relation on $s^{-1}(x)$ so that two elements $g, h\in s^{-1}(x)$ are equivalent if and only if $gh^{-1}\in \cal{H}$.
For a subset $E$ of $s^{-1}(x)$, let $E/\cal{H}$ denote the set of equivalence classes with respect to this equivalence relation on $E$.
We define the {\it index} of $\cal{H}$ in $\cal{G}$ at $x$, denoted by $[\cal{G}: \cal{H}]_x$, as the cardinality $|s^{-1}(x)/\cal{H}|$.

\begin{lem}\label{lem-index}
In the above notation, we put $I(x)=[\cal{G}:\cal{H}]_x$ for $x\in X$. 
Then the following assertions hold:
\begin{enumerate}
\item The function $I\colon X\to \mathbb{Z}_{>0}\cup \{ \infty \}$ is Borel, and we have $I(s(g))=I(r(g))$ for any $g\in \cal{G}$.
In particular, if $\cal{G}$ is ergodic, then $I$ is essentially constant.
\item For any Borel subset $A$ of $X$ with positive measure and any $x\in A$, we have the inequality $[(\cal{G})_A: (\cal{H})_A]_x\leq I(x)$.
If $\cal{H}$ is ergodic, then the equality holds.
\item If we have a non-singular action of a discrete group $\Gamma$ on $(X, \mu)$ and a subgroup $\Lambda$ of $\Gamma$ such that $\cal{G}=\Gamma \ltimes (X, \mu)$ and $\cal{H}=\Lambda \ltimes (X, \mu)$, then $I(x)=[\Gamma: \Lambda]$ for any $x\in X$.
\end{enumerate}
\end{lem}

\begin{proof}
We prove assertion (i), following the proof of \cite[Lemma 1.1 (a)]{fsz}.
Taking Borel sections of the source map $s$, we find a countable set $N$ and a Borel map $\eta_n\colon D_n\to \cal{G}$ indexed by $n\in N$, where $D_n$ is a Borel subset of $X$ with positive measure, such that for any $n\in N$ and any $x\in D_n$, we have $s\circ \eta_n(x)=x$; and for any $g\in \cal{G}$, there exists a unique $n\in N$ with $s(g)\in D_n$ and $\eta_n(s(g))=g$.
Similarly, we find a countable set $M$ and a Borel map $\zeta_m\colon D_m\to \cal{H}$ indexed by $m\in M$, where $D_m$ is a Borel subset of $X$ with positive measure, such that for any $m\in M$ and any $x\in D_m$, we have $s\circ \zeta_m(x)=x$; and for any $g\in \cal{H}$, there exists a unique $m\in M$ with $s(g)\in D_m$ and $\zeta_m(s(g))=g$.
For $n_1,\ldots, n_k\in N$ and $m\in M$, we define a Borel subset $E(n_1,\ldots, n_k; m)$ of $X$ as the set of all points $x$ of $X$ such that for any $i=1,\ldots, k$, we have $x\in D_{n_i}$; and for any distinct $i, j=1,\ldots, k$ with $r\circ \eta_{n_i}(x)\in D_m$, the two elements $\zeta_m(r\circ \eta_{n_i}(x))\eta_{n_i}(x)$ and $\eta_{n_j}(x)$ are distinct.
For any $k\in \mathbb{Z}_{>0}$, the equality
\[\{\, x\in X\mid I(x)\geq k\,\}=\bigcup_{(n_1,\ldots, n_k)\in N^k}\bigcap_{m\in M}E(n_1,\ldots, n_k; m)\]
holds.
The function $I$ is thus Borel.

For any $g\in \cal{G}$, putting $x=s(g)$ and $y=r(g)$, we have the bijection from $s^{-1}(x)$ onto $s^{-1}(y)$ sending each element $h$ of $s^{-1}(x)$ to $hg^{-1}$.
It induces a bijection from $s^{-1}(x)/\cal{H}$ onto $s^{-1}(y)/\cal{H}$.
The equality $I(x)=I(y)$ follows.
Assertion (i) is proved.

Let $A$ be a Borel subset of $X$ with positive measure.
Let $s_A\colon (\cal{G})_A\to A$ denote the source map of $(\cal{G})_A$.
The inequality in assertion (ii) holds because for any $x\in A$, the inclusion of $s_A^{-1}(x)$ into $s^{-1}(x)$ induces an injective map from $s_A^{-1}(x)/(\cal{H})_A$ into $s^{-1}(x)/\cal{H}$.
Assume that $\cal{H}$ is ergodic.
For $x\in X$, let $e_x\in \cal{G}$ denote the unit element at $x$.
Ergodicity of $\cal{H}$ implies that there exists a Borel map $f\colon X\to \cal{H}$ with $f(x)=e_x$ for any $x\in A$ and with $s\circ f(y)=y$ and $r\circ f(y)\in A$ for any $y\in X$.
For any $x\in A$ and any $g\in s^{-1}(x)$, the two elements $g$ and $f(r(g))g$ lie in the same class in $s^{-1}(x)/\cal{H}$, and we have $f(r(g))g\in (\cal{G})_A$.
The map from $s_A^{-1}(x)/(\cal{H})_A$ into $s^{-1}(x)/\cal{H}$ is therefore surjective, and the equality $[(\cal{G})_A: (\cal{H})_A]_x=I(x)$ follows.
Assertion (ii) is proved.

Assertion (iii) follows by definition.
\end{proof}

\begin{lem}\label{lem-index-formula}
Let $\cal{G}$ be a discrete measured groupoid on a standard finite measure space $(X, \mu)$.
Let $r, s\colon \cal{G}\to X$ denote the range and source maps of $\cal{G}$, respectively.
Let $\cal{H}$, $\cal{K}$ and $\cal{L}$ be subgroupoids of $\cal{G}$.
Then the following assertions hold:
\begin{enumerate}
\item For any $x\in X$, we have $[\cal{G}:\cal{H}\cap \cal{K}]_x\leq [\cal{G}:\cal{H}]_x[\cal{G}:\cal{K}]_x$.
\item If $\cal{H}<\cal{K}$, then for any $x\in X$, we have $[\cal{K}\cap \cal{L}: \cal{H}\cap \cal{L}]_x\leq [\cal{K}: \cal{H}]_x$.
\item If $\cal{H}<\cal{K}$, then for any $x\in X$, we have the equality
\[[\cal{G}: \cal{H}]_x=\sum_{g\in E}[\cal{K}:\cal{H}]_{r(g)},\]
where $E$ is a set of representatives of all classes in $s^{-1}(x)/\cal{K}$. 
Moreover, if any real-valued Borel function on $X$ that is $\cal{K}$-invariant is $\cal{G}$-invariant, then for a.e.\ $x\in X$, we have the equality
\[[\cal{G}:\cal{H}]_x=[\cal{G}:\cal{K}]_x[\cal{K}:\cal{H}]_x.\]
\end{enumerate}
\end{lem}

\begin{proof}
For $z\in X$, we set $\cal{G}_z=s^{-1}(z)$.
We prove assertion (i).
Fix $x\in X$.
Let
\[p\colon \cal{G}_x/(\cal{H}\cap \cal{K})\to \cal{G}_x/\cal{H},\quad q\colon \cal{G}_x/(\cal{H}\cap \cal{K})\to \cal{G}_x/\cal{K}\]
be the natural maps.
For any $g, h\in \cal{G}_x$, if $p(g)=p(h)$ and $q(g)=q(h)$, then $gh^{-1}\in \cal{H}\cap \cal{K}$.
It follows that the map from the set $\cal{G}_x/(\cal{H}\cap \cal{K})$ into the product $\cal{G}_x/\cal{H}\times \cal{G}_x/\cal{K}$ defined by $p$ and $q$ is injective.
Assertion (i) is proved.

Assertion (ii) holds because for any $x\in X$, the inclusion of $\cal{G}_x\cap \cal{K}\cap \cal{L}$ into $\cal{G}_x\cap \cal{K}$ induces an injective map from $(\cal{G}_x\cap \cal{K}\cap \cal{L})/(\cal{H}\cap \cal{L})$ into $(\cal{G}_x\cap \cal{K})/\cal{H}$.

We prove assertion (iii).
For any $z\in X$, we have the natural map $q_z\colon \cal{G}_z\to \cal{G}_z/\cal{K}$.
Fix $x\in X$.
Let $E$ be a set of representatives of all classes in $\cal{G}_x/\cal{K}$.
Pick $g\in E$ and put $y=r(g)$.
Let $e_y\in \cal{G}$ denote the unit element at $y$.
The map from $q_x^{-1}(q_x(g))$ onto $q_y^{-1}(q_y(e_y))$ sending each element $h$ of $q_x^{-1}(q_x(g))$ to $hg^{-1}$ is bijective.
This map induces a bijection from $q_x^{-1}(q_x(g))/\cal{H}$ onto $q_y^{-1}(q_y(e_y))/\cal{H}$.
We thus have the equality $|q_x^{-1}(q_x(g))/\cal{H}|=|q_y^{-1}(q_y(e_y))/\cal{H}|=[\cal{K}:\cal{H}]_y$.
The equality
\[[\cal{G}: \cal{H}]_x=|\cal{G}_x/\cal{H}|=\sum_{g\in E}|q_x^{-1}(q_x(g))/\cal{H}|=\sum_{g\in E}[\cal{K}:\cal{H}]_{r(g)}\]
is obtained.
The former assertion in assertion (iii) is proved.

The function on $X$ assigning $[\cal{K}: \cal{H}]_z$ to each $z\in X$ is $\cal{K}$-invariant by Lemma \ref{lem-index} (i).
If the ergodic decompositions for $\cal{G}$ and $\cal{K}$ are the same map, then this function is also $\cal{G}$-invariant.
For a.e.\ $x\in X$ and any $g\in E$, the equality $[\cal{K}: \cal{H}]_{r(g)}=[\cal{K}: \cal{H}]_x$ then holds.
The latter assertion in assertion (iii) is proved.
\end{proof}

The following lemma is obtained from the definition of index.

\begin{lem}\label{lem-index-fi}
Let $\cal{G}$ be a discrete measured groupoid on a standard finite measure space $(X, \mu)$.
Let $\Gamma$ be a discrete group, and let $\Lambda$ be a subgroup of $\Gamma$.
If $\rho \colon \cal{G}\to \Gamma$ is a Borel homomorphism, then $\rho^{-1}(\Lambda)$ is a subgroupoid of $\cal{G}$, and for any $x\in X$, we have $[\cal{G}:\rho^{-1}(\Lambda)]_x\leq [\Gamma :\Lambda]$.
\end{lem}

Let $\Gamma$ be a discrete group, $\Lambda$ a subgroup of $\Gamma$ of finite index, and $S$ a space on which $\Gamma$ acts.
Choose a family $\{ \gamma_1,\ldots, \gamma_N\}$ of representatives for all right cosets of $\Lambda$ in $\Gamma$ with $N=[\Gamma: \Lambda]$.
For any fixed point $x_0$ in $S$ for the action of $\Lambda$, the subset $\{ \gamma_1^{-1}x_0,\ldots, \gamma_N^{-1}x_0\}$ of $S$ is an orbit for the action of $\Gamma$ and is thus fixed by $\Gamma$.
The following lemma is an analogue for groupoids.

\begin{lem}\label{lem-finite-index}
In the notation in the second paragraph of this subsection, let $\Gamma$ be a discrete group, $\rho \colon \cal{G}\to \Gamma$ a Borel homomorphism, and $S$ a standard Borel space on which $\Gamma$ acts by Borel automorphisms.
We denote by $\cal{F}(S)$ the Borel space of all non-empty finite subsets of $S$, on which $\Gamma$ naturally acts.
Suppose that the function $I(x)=[\cal{G}:\cal{H}]_x$ on $X$ is essentially constant, and its essential value, denoted by $N$, is finite.
Let $\phi_1,\ldots, \phi_N$ be Borel maps from $X$ into $\cal{G}$ such that for a.e.\ $x\in X$, $\{ \phi_1(x),\ldots, \phi_N(x)\}$ is a set of representatives of all classes in $s^{-1}(x)/\cal{H}$.
Then for any $(\cal{H}, \rho)$-invariant Borel map $\psi \colon X\to S$, the Borel map $\Psi \colon X\to \cal{F}(S)$ defined by
\[\Psi(x)=\{\, \rho(\phi_1(x))^{-1}\psi(r\circ \phi_1(x)),\ldots, \rho(\phi_N(x))^{-1}\psi(r\circ \phi_N(x))\,\}\]
for $x\in X$ is $(\cal{G}, \rho)$-invariant.
\end{lem}

\begin{proof}
We denote by $\mathfrak{S}(N)$ the symmetric group on the set $\{ 1,\ldots, N\}$.
For a.e.\ $g\in \cal{G}$ with $x=s(g)$ and $y=r(g)$, the set $\{ \phi_1(x)g^{-1},\ldots, \phi_N(x)g^{-1}\}$ is then a set of representatives of all classes in $s^{-1}(y)/\cal{H}$.
There thus exists a Borel homomorphism $\alpha \colon \cal{G}\to \mathfrak{S}(N)$ such that for a.e.\ $g\in \cal{G}$ with $x=s(g)$ and $y=r(g)$ and for any $i\in \{ 1,\ldots, N\}$, we have $\phi_i(x)g^{-1}\phi_{\alpha(g)(i)}(y)^{-1}\in \cal{H}$.
Putting $j=\alpha(g)(i)$, we have
\begin{align*}
\rho(g)\rho(\phi_i(x))^{-1}\psi(r\circ \phi_i(x))&=\rho(\phi_j(y))^{-1}\rho(\phi_j(y)g\phi_i(x)^{-1})\psi(r\circ \phi_i(x))\\
&=\rho(\phi_j(y))^{-1}\psi(r\circ \phi_j(y)),
\end{align*}
where the last equality holds because $\psi$ is $(\cal{H}, \rho)$-invariant.
The lemma follows.
\end{proof}

In Lemma \ref{lem-finite-index}, we note that for a.e.\ $x\in X$, exactly one of $\phi_1(x),\ldots, \phi_N(x)$, say $\phi_1(x)$, belongs to $\cal{H}$, and we then have $\rho(\phi_1(x))^{-1}\psi(r\circ \phi_1(x))=\psi(x)$ because $\psi$ is $(\cal{H}, \rho)$-invariant.
It follows that $\Psi(x)$ contains $\psi(x)$ for a.e.\ $x\in X$.

\begin{lem}\label{lem-erg-dec}
In the notation in the second paragraph of this subsection, suppose that $\cal{G}$ is measure-preserving.
We also suppose that the function $I(x)=[\cal{G}: \cal{H}]_x$ on $X$ is essentially constant, and its essential value, denoted by $N$, is finite.
Let $\pi \colon (X, \mu)\to (Z, \xi)$ and $\theta \colon (X, \mu)\to (W, \omega)$ be the ergodic decompositions for $\cal{G}$ and $\cal{H}$, respectively.
Let $\sigma \colon (W, \omega)\to (Z, \xi)$ be the canonical Borel map such that $\pi =\sigma \circ \theta$, i.e., the following diagram commutes:
\[\xymatrix{
& \ar[dl]_{\theta} (X, \mu) \ar[dr]^{\pi} & \\
(W, \omega) \ar[rr]^{\sigma} & & (Z, \xi) \\  
}\]
Let $\omega =\int_{Z}\omega_z d\xi(z)$ be the disintegration with respect to $\sigma$.
Then for a.e.\ $z\in Z$, the set $\sigma^{-1}(z)$ consists of at most $N$ points up to $\omega_z$-null sets.
\end{lem}

\begin{proof}
To prove the lemma, it is enough to show that for a.e.\ $z\in Z$, there exists no Borel subset of $\sigma^{-1}(z)$ whose measure with respect to $\omega_z$ is positive and less than $1/N$.
We assume the contrary, and will deduce a contradiction.

Thanks to general description of a measurable decomposition of a measure space in \cite[\S 4, No.1]{rohlin}, there exist a countable set $M$, a Borel partition $Z=\bigsqcup_{m\in M}Z_m$ and a standard finite measure space $(E_m, \eta_m)$ indexed by each $m\in M$ satisfying the following:
For any $m\in M$, there exists a Borel isomorphism $f_m$ from a conull Borel subset of $E_m\times Z_m$ onto a conull Borel subset of $\sigma^{-1}(Z_m)$ such that for a.e.\ $y\in E_m$ and a.e.\ $z\in Z_m$, the equality $\sigma \circ f_m(y, z)=z$ holds; and the measure $(f_m)_*(\eta_m\times \xi_m)$ and the restriction of $\omega$ to $\sigma^{-1}(Z_m)$ are equivalent, where $\xi_m$ is the restriction of $\xi$ to $Z_m$.
This application of the result in \cite{rohlin} owes to the proof of \cite[Proposition 2.21]{fmw}.
Our assumption in the last paragraph implies that there exists a Borel subset $A$ of $Z$ with $\xi(A)>0$ and a Borel subset $B$ of $\sigma^{-1}(A)$ with $0<\omega_z(B\cap \sigma^{-1}(z))<1/N$ for a.e.\ $z\in A$.
We choose Borel maps $\phi_1,\ldots, \phi_N$ from $X$ into $\cal{G}$ such that for a.e.\ $x\in X$, $\{ \phi_1(x),\ldots, \phi_N(x)\}$ is a set of representatives of all classes in $s^{-1}(x)/\cal{H}$.
We put
\[Y=\theta^{-1}(B)\quad \textrm{and}\quad D=\bigcup_{i=1}^N (r\circ \phi_i)^{-1}(Y).\]
We claim that $D$ is $\cal{G}$-invariant.
For a.e.\ $g\in \cal{G}$ with $x=s(g)\in D$ and $y=r(g)$, there exists $i$ with $r\circ \phi_i(x)\in Y$, and there exists $j$ with $\phi_j(y)g\phi_i(x)^{-1}\in \cal{H}$.
Since $Y$ is $\cal{H}$-invariant, we have $r\circ \phi_j(y)\in Y$.
We thus have $y\in D$.
The claim is proved.
The claim and the inclusion $B\subset \sigma^{-1}(A)$ imply the inclusion $D\subset \pi^{-1}(A)$.

We show the converse inclusion.
Let $\mu =\int_Z\mu_z d\xi(z)$ be the disintegration with respect to $\pi$.
We have $\omega =\theta_*\mu =\int_Z\theta_*\mu_z d\xi(z)$.
By uniqueness of disintegration, we have $\omega_z=\theta_*\mu_z$ for a.e.\ $z\in Z$.
The condition $\omega_z(B\cap \sigma^{-1}(z))>0$ for a.e.\ $z\in A$ implies that $\mu_z(Y\cap \pi^{-1}(z))>0$ for a.e.\ $z\in A$.
Since $D$ is a $\cal{G}$-invariant Borel subset of $X$ containing $Y$, we have $\pi^{-1}(A)\subset D$.

We thus obtained the equality $D=\pi^{-1}(A)$.
By the definition of $D$, we have
\[\mu(D)\leq N\mu(Y)=N\omega(B)=N\int_A\omega_z(B\cap \sigma^{-1}(z))\, d\xi(z)<\xi(A)=\mu(D).\]
This is a contradiction.
\end{proof}

In Lemma \ref{lem-erg-dec}, thanks to its conclusion, replacing $W$ by its conull Borel subset, we can assume that for a.e.\ $z\in Z$, any point of $\sigma^{-1}(z)$ has positive measure with respect to $\omega_z$, and $\sigma^{-1}(z)$ consists of at most $N$ points.

We end this subsection with the following lemma on disintegration of a measure, which will be applied in the setting of Lemma \ref{lem-erg-dec}.

\begin{lem}\label{lem-rest}
Let $(X, \mu)$, $(Z, \xi)$ and $(W, \omega)$ be standard finite measure spaces.
Let $\pi \colon X\to Z$, $\theta \colon X\to W$ and $\sigma \colon W\to Z$ be Borel maps satisfying the equalities $\pi_*\mu=\xi$, $\theta_*\mu=\omega$ and $\pi =\sigma \circ \theta$.
Let $\mu =\int_Z\mu_zd\xi(z)$ and $\mu =\int_W\nu_wd\omega(w)$ be the disintegrations with respect to $\pi$ and $\theta$, respectively.
Suppose that for a.e.\ $z\in Z$, the set $\sigma^{-1}(z)$ is countable.
Then for a.e.\ $w\in W$, we have the equality
\[\mu_{\sigma(w)}|_{\theta^{-1}(w)}=\mu_{\sigma(w)}(\theta^{-1}(w))\nu_w.\]
\end{lem}

\begin{proof}
For $z\in Z$, we put $X_z=\pi^{-1}(z)$ and $W_z=\sigma^{-1}(z)$.
Similarly, for $w\in W$, we put $X_w=\theta^{-1}(w)$.
For a.e.\ $z\in Z$, we have the partition $X_z=\bigsqcup_{w\in W_z}X_w$ into countably many Borel subsets.
For any Borel subset $B$ of $W$, we have
\begin{align*}
\omega(B)&=\mu(\theta^{-1}(B))=\int_Z\mu_z(\theta^{-1}(B)\cap X_z)\, d\xi(z)=\int_Z\sum_{w\in B\cap W_z}\mu_z(X_w)\, d\xi(z).
\end{align*}
For any Borel subset $A$ of $X$, we thus have
\begin{align*}
\mu(A)&=\int_Z\mu_z(A\cap X_z)\, d\xi(z)=\int_Z\sum_{w\in W_z}\mu_z(A\cap X_w)\, d\xi(z)\\
&=\int_W\frac{\mu_{\sigma(w)}(A\cap X_w)}{\mu_{\sigma(w)}(X_w)}\, d\omega(w).
\end{align*}
By uniqueness of disintegration, we obtain the desired equality.
\end{proof}


\subsection{Local index}

Throughout this subsection, we fix a standard finite measure space $(X, \mu)$, a discrete measured groupoid $\cal{G}$ on $(X, \mu)$ and a subgroupoid $\cal{H}$ of $\cal{G}$.
As proved in Lemma \ref{lem-index} (ii), if $A$ is a Borel subset of $X$ with positive measure, the inequality $[(\cal{G})_A:(\cal{H})_A]_x\leq [\cal{G}:\cal{H}]_x$ holds for any $x\in A$ though the equality does not hold in general.
Under a certain assumption, for $x\in X$, we define the local index of $\cal{H}$ in $\cal{G}$ at $x$, denoted by $[[\cal{G}: \cal{H}]]_x$, which necessarily satisfies the equality $[[(\cal{G})_A:(\cal{H})_A]]_x=[[\cal{G}:\cal{H}]]_x$ for any Borel subset $A$ of $X$ with positive measure and a.e.\ $x\in A$.
This local index can be seen as an anlogue of the local index for subfactors introduced by Jones in \cite[\S 2.2]{jones}.

Let $\pi \colon (X, \mu)\to (Z, \xi)$ and $\theta \colon (X, \mu)\to (W, \omega)$ be the ergodic decompositions for $\cal{G}$ and $\cal{H}$, respectively.
Let $\sigma \colon (W, \omega)\to (Z, \xi)$ be the canonical Borel map such that $\pi =\sigma \circ \theta$, i.e., the following diagram commutes:
\[\xymatrix{
& \ar[dl]_{\theta} (X, \mu) \ar[dr]^{\pi} & \\
(W, \omega) \ar[rr]^{\sigma} & & (Z, \xi) \\  
}\]
We suppose that for any $z\in Z$, the set $\sigma^{-1}(z)$ is countable.
Under this assumption, we define the local index of $\cal{H}$ in $\cal{G}$ at $x\in X$.

Choose a countable set $N$ and a Borel partition $W=\bigsqcup_{n\in N}W_n$ such that for any $n\in N$, we have $\omega(W_n)>0$, and the map $\sigma$ is injective on $W_n$.
Set $Y_n=\theta^{-1}(W_n)$ for $n\in N$.
We define a Borel function $J\colon X\to \mathbb{Z}_{>0}\cup \{ \infty \}$ so that for $n\in N$ and $x\in Y_n$, we have $J(x)=[(\cal{G})_{Y_n}: (\cal{H})_{Y_n}]_x$.
The following lemma implies that this function does not depend on the choice of the partition $W=\bigsqcup_{n\in N}W_n$.

\begin{lem}\label{lem-li-wd}
In the above notation, fix $n\in N$.
If $A$ is a Borel subset of $Y_n$ with $\mu(A)>0$, then for a.e.\ $x\in A$, we have $[(\cal{G})_A:(\cal{H})_A]_x=[(\cal{G})_{Y_n}:(\cal{H})_{Y_n}]_x$.
\end{lem}

\begin{proof}
We put $Y=Y_n$.
Let $s_Y\colon (\cal{G})_Y\to Y$ and $s_A\colon (\cal{G})_A\to A$ be the source maps.
For any $x\in A$, the inclusion of $s_A^{-1}(x)$ into $s_Y^{-1}(x)$ induces the injective map $\imath$ from $s_A^{-1}(x)/\cal{H}$ into $s_Y^{-1}(x)/\cal{H}$.
To prove the lemma, it suffices to show that $\imath$ is surjective for a.e.\ $x\in A$.
Since $\sigma$ is injective on $W_n$, the equality $(\cal{G})_YA=\cal{H}A$ holds.
For a.e.\ $x\in A$ and any $g\in s_Y^{-1}(x)$, there thus exists $h\in \cal{H}$ with $s(h)\in A$ and $r(h)=r(g)$.
The product $h^{-1}g$ belongs to $s_A^{-1}(x)$, and we have $g(h^{-1}g)^{-1}=h\in \cal{H}$.
The map $\imath$ is therefore surjective. 
\end{proof}

For $x\in X$, we denote $J(x)$ by $[[\cal{G}:\cal{H}]]_x$, and call this number the {\it local index} of $\cal{H}$ in $\cal{G}$ at $x$.
We present basic properties of local index.

\begin{lem}\label{lem-li-prod}
In the notation in the second paragraph of this subsection, let $\cal{K}$ be a subgroupoid of $\cal{G}$ with $\cal{H}<\cal{K}$.
Then for a.e.\ $x\in X$, the local indices $[[\cal{G}:\cal{K}]]_x$ and $[[\cal{K}:\cal{H}]]_x$ at $x$ are well-defined, and we have the equality
\[[[\cal{G}:\cal{H}]]_x=[[\cal{G}:\cal{K}]]_x[[\cal{K}:\cal{H}]]_x.\]
\end{lem}

\begin{proof}
Let $\Phi \colon (X, \mu)\to (V, \upsilon)$ be the ergodic decomposition for $\cal{K}$.
We have the canonical Borel map from $(W, \omega)$ into $(V, \upsilon)$ and that from $(V, \upsilon)$ into $(Z, \xi)$ because we have $\cal{H}<\cal{K}<\cal{G}$.
Since the composition of these two maps is equal to $\sigma$, the inverse image of any point under these two maps is countable.
The former assertion follows.
The desired equality follows from Lemma \ref{lem-index-formula} (iii) and Lemma \ref{lem-li-wd}.
\end{proof}

\begin{lem}\label{lem-li-res}
In the notation in the second paragraph of this subsection, if $A$ is a Borel subset of $X$ with $\mu(A)>0$, then for a.e.\ $x\in A$, we have the equality
\[[[(\cal{G})_A:(\cal{H})_A]]_x=[[\cal{G}:\cal{H}]]_x.\]
\end{lem}

\begin{proof}
Let $\pi_1\colon (A, \mu|_A)\to (Z_1, \xi_1)$ and $\theta_1\colon (A, \mu|_A)\to (W_1, \omega_1)$ be the ergodic decompositions for $(\cal{G})_A$ and $(\cal{H})_A$, respectively.
We have the canonical Borel map $\sigma_1\colon (W_1, \omega_1)\to (Z_1, \xi_1)$ such that $\pi_1=\sigma_1\circ \theta_1$, i.e., the following diagram commutes:
\[\xymatrix{
& \ar[dl]_{\theta_1} (A, \mu|_A) \ar[dr]^{\pi_1} & \\
(W_1, \omega_1) \ar[rr]^{\sigma_1} & & (Z_1, \xi_1) \\
}\]
We also have the canonical injective Borel map from $Z_1$ into $Z$, and identify $Z_1$ with the image, which is a Borel subset of $Z$.
In the same manner, we naturally identify $W_1$ with a Borel subset of $W$.
The map $\sigma_1$ is the restriction of $\sigma$ under these identifications.

Let $W_2$ be a Borel subset of $W_1$ such that $\omega_1(W_2)>0$, and $\sigma_1$ is injective on $W_2$.
Put $Y_1=\theta_1^{-1}(W_2)$.
By definition, we have $[[(\cal{G})_A:(\cal{H})_A]]_x=[(\cal{G})_{Y_1}:(\cal{H})_{Y_1}]_x$ for a.e.\ $x\in Y_1$.
We set $Y=\cal{H}Y_1=\theta^{-1}(W_2)$.
Since $\sigma$ is injective on $W_2$, we have the equality $[[\cal{G}:\cal{H}]]_x=[(\cal{G})_Y:(\cal{H})_Y]_x$ for a.e.\ $x\in Y$.
By Lemma \ref{lem-li-wd}, the desired equality holds for a.e.\ $x\in Y_1$.
\end{proof}

\begin{lem}\label{lem-li-group}
Let $\Gamma$ be a discrete group.
Let $\Lambda$ be a finite index, normal subgroup of $\Gamma$.
Suppose that we have a non-singular action of $\Gamma$ on a standard finite measure space $(X, \mu)$.
We set $\cal{G}=\Gamma \ltimes X$ and $\cal{H}=\Lambda \ltimes X$, and define the maps $\pi$, $\theta$ and $\sigma$ as in the second paragraph of this subsection.
Let $\omega =\int_Z\omega_zd\xi(z)$ be the disintegration with respect to $\sigma$.
We assume that for a.e.\ $z\in Z$, any point of $\sigma^{-1}(z)$ has positive measure with respect to $\omega_z$.
Then for a.e.\ $x\in X$, we have the equality
\[[[\cal{G}:\cal{H}]]_x=\frac{[\Gamma :\Lambda]}{|\sigma^{-1}(\pi(x))|}.\]
\end{lem}

\begin{proof}
We have the canonical non-singular action of $\Gamma$ on $(W, \omega)$ because $\Lambda$ is normal in $\Gamma$.
For $w\in W$, let $\Gamma_w$ denote the stabilizer of $w$ in $\Gamma$, which contains $\Lambda$.
The set $\sigma^{-1}(\sigma(w))$ is then identified with $\Gamma /\Gamma_w$.
Let $W_1$ be a Borel subset of $W$ such that $\omega(W_1)>0$; the map $\sigma$ is injective on $W_1$; and $\Gamma_w=\Gamma_{w'}$ for any $w, w'\in W_1$.
Set $Y=\theta^{-1}(W_1)$ and fix $w_0\in W_1$.
The set $Y$ is $\Gamma_{w_0}$-invariant, and the equality $(\cal{G})_Y=\Gamma_{w_0}\ltimes Y$ holds.
For a.e.\ $x\in Y$, we then have the equality
\[[[\cal{G}:\cal{H}]]_x=[(\cal{G})_Y:(\cal{H})_Y]_x=[\Gamma_{w_0}\ltimes Y: \Lambda \ltimes Y]_x=[\Gamma_{w_0}:\Lambda].\]
The lemma therefore follows.
\end{proof}


\subsection{Quasi-normal subgroupoids}

Normal subgroupoids are studied in \cite{kida-exama} and \cite{sauer-thom} as a generalization of normal subrelations introduced by Feldman, Sutherland and Zimmer \cite{fsz}.
In this subsection, we recall normal subgroupoids, and introduce quasi-normal subgroupoids, which are a slight generalization of normal ones.
We refer to \cite{aoi} and \cite{ay} for related works in the framework of von Neumann algebras.

Let $\cal{G}$ be a discrete measured groupoid on a standard finite measure space $(X, \mu)$.
Let $r, s\colon \cal{G}\to X$ denote the range and source maps of $\cal{G}$, respectively.
We define $[[\cal{G}]]$ as the set of all Borel maps $\phi \colon D_{\phi}\to \cal{G}$ from a Borel subset $D_{\phi}$ of $X$ into $\cal{G}$ such that $s\circ \phi(x)=x$ for any $x\in D_{\phi}$; and the map $r\circ \phi \colon D_{\phi}\to X$ is injective.
For any $\phi \in [[\cal{G}]]$, we set $R_{\phi}=r\circ \phi(D_{\phi})$ and define a Borel map $U_{\phi}\colon (\cal{G})_{D_{\phi}}\to (\cal{G})_{R_{\phi}}$ by the formula $U_{\phi}(g)=\phi(r(g))g\phi(s(g))^{-1}$ for $g\in (\cal{G})_{D_{\phi}}$.
The map $U_{\phi}$ is an isomorphism of discrete measured groupoids.

Let $\cal{S}$ be a subgroupoid of $\cal{G}$.
For $\phi \in [[\cal{G}]]$, we set $\cal{S}^{\phi}=U_{\phi}((\cal{S})_{D_{\phi}})$.
We define two subsets ${\rm N}_{\cal{G}}(\cal{S})$, $\qn_{\cal{G}}(\cal{S})$ of $[[\cal{G}]]$ by
\begin{align*}
{\rm N}_{\cal{G}}(\cal{S})&=\{\, \phi \in [[\cal{G}]]\mid \cal{S}^{\phi}=(\cal{S})_{R_{\phi}}\,\},\\
\qn_{\cal{G}}(\cal{S})&=\{\, \phi \in [[\cal{G}]]\mid [(\cal{S})_{R_{\phi}}: (\cal{S})_{R_{\phi}}\cap \cal{S}^{\phi}]_x<\infty,\ [\cal{S}^{\phi}: (\cal{S})_{R_{\phi}}\cap \cal{S}^{\phi}]_x<\infty\\ 
& \hspace{85mm} \textrm{for a.e.\ }x\in R_{\phi}\,\}.
\end{align*}

\begin{defn}\label{defn-qn}
Let $\cal{G}$ be a discrete measured groupoid on a standard finite measure space $(X, \mu)$ with $s\colon \cal{G}\to X$ the source map.
Let $\cal{S}$ be a subgroupoid of $\cal{G}$.
\begin{enumerate}
\item We say that $\cal{S}$ is {\it normal} in $\cal{G}$ if there exists a countable family $\{ \phi_n\}_n$ of elements of ${\rm N}_{\cal{G}}(\cal{S})$ such that for a.e.\ $g\in \cal{G}$, there exists $n$ with $s(g)\in D_{\phi_n}$ and $\phi_n(s(g))g^{-1}\in \cal{S}$.
\item We say that $\cal{S}$ is {\it quasi-normal} in $\cal{G}$ if the same condition as in (i) holds after replacing ${\rm N}_{\cal{G}}(\cal{S})$ with $\qn_{\cal{G}}(\cal{S})$.
\end{enumerate}
\end{defn}

While the above definition of normal subgroupoids is slightly different from those in \cite{kida-exama} and \cite{sauer-thom}, it can be checked that they are equivalent.

We introduce composition and inverse of elements of $[[\cal{G}]]$. 
Pick $\phi, \psi \in [[\cal{G}]]$.
We define $\eta \in [[\cal{G}]]$ as follows.
Set $D_{\eta}=(r\circ \phi)^{-1}(D_{\psi}\cap R_{\phi})$ and define a Borel map $\eta \colon D_{\eta}\to \cal{G}$ by $\eta(x)=\psi(r\circ \phi(x))\phi(x)$ for $x\in D_{\eta}$.
The map $\eta$ then belongs to $[[\cal{G}]]$.
The isomorphism $U_{\eta}$ is equal to the restriction of $U_{\psi}\circ U_{\phi}$ to $(\cal{G})_{D_{\eta}}$.
We denote the map $\eta$ by $\psi \bullet \phi$ and call it the {\it composition} of $\phi$ and $\psi$.

For $\phi \in [[\cal{G}]]$, we define $\zeta \in [[\cal{G}]]$ as follows.
Set $D_{\zeta}=R_{\phi}$ and define a Borel map $\zeta \colon D_{\zeta}\to \cal{G}$ by $\zeta(x)=\phi((r\circ \phi)^{-1}(x))^{-1}$ for $x\in D_{\zeta}$.
The map $\zeta$ then belongs to $[[\cal{G}]]$.
We have the equality $U_{\zeta}=U_{\phi}^{-1}$.
Let us call the map $\zeta$ the {\it inverse} of $\phi$.

\begin{lem}\label{lem-qn}
Let $\cal{S}$ be a subgroupoid of $\cal{G}$.
Then the following assertions hold.
\begin{enumerate}
\item For any $\phi \in \qn_{\cal{G}}(\cal{S})$ and any Borel subset $A$ of $D_{\phi}$, the restriction of $\phi$ to $A$, denoted by $\phi|_A$, belongs to $\qn_{\cal{G}}(\cal{S})$.
\item For any $\phi, \psi \in \qn_{\cal{G}}(\cal{S})$, we have $\psi \bullet \phi \in \qn_{\cal{G}}(\cal{S})$.
\item For any $\phi \in \qn_{\cal{G}}(\cal{S})$, the inverse of $\phi$ also belongs to $\qn_{\cal{G}}(\cal{S})$.
\end{enumerate}
\end{lem}

\begin{proof}
Let $r, s\colon \cal{G}\to X$ be the range and source maps of $\cal{G}$, respectively. 
Assertion (i) follows from Lemma \ref{lem-index} (ii).
To prove assertion (iii), we pick $\phi \in \qn_{\cal{G}}(\cal{S})$.
Let $\theta$ be the inverse of $\phi$.
For a.e.\ $x\in R_{\theta}$, we have
\[[(\cal{S})_{R_{\theta}}: (\cal{S})_{R_{\theta}}\cap \cal{S}^{\theta}]_x=[(\cal{S})_{D_{\phi}}: (\cal{S})_{D_{\phi}}\cap \cal{S}^{\theta}]_x=[\cal{S}^{\phi}: \cal{S}^{\phi}\cap (\cal{S})_{R_{\phi}}]_{r\circ \phi(x)},\]
where the second equality is obtained by applying $U_{\phi}$.
The right hand side is finite because $\phi$ belongs to $\qn_{\cal{G}}(\cal{S})$.
In a similar way, we can show that $[\cal{S}^{\theta}: (\cal{S})_{R_{\theta}}\cap \cal{S}^{\theta}]_x$ is finite for a.e.\ $x\in R_{\theta}$.
Assertion (iii) is proved.

To prove assertion (ii), we pick $\phi, \psi \in \qn_{\cal{G}}(\cal{S})$ and put $\eta =\psi \bullet \phi$.
Let $\zeta$ be the inverse of $\psi$.
For any $x\in R_{\eta}$, let $E$ be a set of representatives of all classes in $(s^{-1}(x)\cap (\cal{S})_{R_{\eta}})/((\cal{S})_{R_{\eta}}\cap (\cal{S}^{\psi})_{R_{\eta}})$.
We have
\begin{align*}
&[(\cal{S})_{R_{\eta}}: (\cal{S})_{R_{\eta}}\cap \cal{S}^{\eta}]_x\leq [(\cal{S})_{R_{\eta}}: (\cal{S})_{R_{\eta}}\cap (\cal{S}^{\psi})_{R_{\eta}}\cap \cal{S}^{\eta}]_x\\
= \, & \sum_{g\in E}[(\cal{S})_{R_{\eta}}\cap (\cal{S}^{\psi})_{R_{\eta}}: (\cal{S})_{R_{\eta}}\cap (\cal{S}^{\psi})_{R_{\eta}}\cap \cal{S}^{\eta}]_{r(g)},
\end{align*}
where the last equality holds by Lemma \ref{lem-index-formula} (iii).
For a.e.\ $x\in R_{\eta}$, the set $E$ is finite because $\psi$ belongs to $\qn_{\cal{G}}(\cal{S})$.
We put $C=r\circ \zeta(R_{\eta})$.
For any $g\in E$, putting $y=r(g)$, we have
\begin{align*}
&[(\cal{S})_{R_{\eta}}\cap (\cal{S}^{\psi})_{R_{\eta}}: (\cal{S})_{R_{\eta}}\cap (\cal{S}^{\psi})_{R_{\eta}}\cap \cal{S}^{\eta}]_y\\
= \, & [(\cal{S}^{\zeta})_C\cap (\cal{S})_C: (\cal{S}^{\zeta})_C\cap (\cal{S})_C\cap (\cal{S}^{\phi})_C]_{r\circ \zeta(y)}\\
\leq \, & [(\cal{S})_C: (\cal{S})_C\cap (\cal{S}^{\phi})_C]_{r\circ \zeta(y)}\leq [(\cal{S})_{R_{\phi}}: (\cal{S})_{R_{\phi}}\cap \cal{S}^{\phi}]_{r\circ \zeta(y)},
\end{align*}
where the first equality is obtained by applying $U_{\zeta}$.
The first and second inequalities follow from Lemma \ref{lem-index-formula} (ii) and Lemma \ref{lem-index} (ii), respectively.
The right hand side is finite because $\phi$ belongs to $\qn_{\cal{G}}(\cal{S})$.
We therefore have $[(\cal{S})_{R_{\eta}}: (\cal{S})_{R_{\eta}}\cap \cal{S}^{\eta}]_x<\infty$ for a.e.\ $x\in R_{\eta}$.

Let $\xi$ denote the inverse of $\eta$, which is the composition of the inverses of $\psi$ and of $\phi$.
For a.e.\ $x\in R_{\eta}$, applying $U_{\xi}$, we obtain the equality
\[[\cal{S}^{\eta}: (\cal{S})_{R_{\eta}}\cap \cal{S}^{\eta}]_x=[(\cal{S})_{R_{\xi}}: \cal{S}^{\xi}\cap (\cal{S})_{R_{\xi}}]_{r\circ \xi(x)}.\]
The right hand side is finite thanks to assertion (iii) and the argument in the last paragraph.
Assertion (ii) is proved.
\end{proof}

In the notation in Definition \ref{defn-qn}, $\cal{S}$ is normal in $\cal{G}$ if and only if there exists a countable family $\{ \phi_n\}_n$ of elements of ${\rm N}_{\cal{G}}(\cal{S})$ such that $\cal{G}=\bigcup_{n}\phi_n(D_{\phi_n})$ up to null sets.
This is because $[[\cal{S}]]$ is contained in ${\rm N}_{\cal{G}}(\cal{S})$ and the composition of two elements of ${\rm N}_{\cal{G}}(\cal{S})$ also belongs to ${\rm N}_{\cal{G}}(\cal{S})$.
A similar property holds for quasi-normality if ${\rm N}_{\cal{G}}(\cal{S})$ is replaced by $\qn_{\cal{G}}(\cal{S})$.

Let $\Gamma$ be a discrete group and $\Lambda$ a subgroup of $\Gamma$.
The set $\qn_{\Gamma}(\Lambda)$ introduced right before Definition \ref{defn-qn} is naturally identified with the subgroup $\comm_{\Gamma}(\Lambda)$ of $\Gamma$ introduced in Section \ref{sec-bs}.
It follows that $\Lambda$ is quasi-normal in $\Gamma$ in the sense of Definition \ref{defn-qn} if and only if the equality $\comm_{\Gamma}(\Lambda)=\Gamma$ holds.

\begin{lem}\label{lem-qn-group}
Let $\Gamma$ be a discrete group and $\Lambda$ a subgroup of $\Gamma$.
Let $\Gamma \c (X, \mu)$ be a non-singular action.
If $\Lambda$ is quasi-normal in $\Gamma$, then $\Lambda \ltimes (X, \mu)$ is quasi-normal in $\Gamma \ltimes (X, \mu)$.
\end{lem}

\begin{proof}
Put $\cal{G}=\Gamma \ltimes (X, \mu)$ and $\cal{H}=\Lambda \ltimes (X, \mu)$.
For any $\gamma \in \Gamma$, the map from $X$ into $\cal{G}$ sending each element $x$ of $X$ to $(\gamma, x)$ belongs to $\qn_{\cal{G}}(\cal{H})$.
The lemma follows.
\end{proof}

\begin{lem}\label{lem-qn-res}
Let $\cal{G}$ be a discrete measured groupoid on a standard finite measure space $(X, \mu)$.
Let $\cal{S}$ be a quasi-normal subgroupoid of $\cal{G}$.
Then for any Borel subset $A$ of $X$ with positive measure, $(\cal{S})_A$ is quasi-normal in $(\cal{G})_A$.
\end{lem}

\begin{proof}
Let $\{ \phi_n\}_{n\in N}$ be a countable family of elements of $\qn_{\cal{G}}(\cal{S})$ such that for a.e.\ $g\in \cal{G}$, there exists $n\in N$ with $s(g)\in D_{\phi_n}$ and $\phi_n(s(g))g^{-1}\in \cal{S}$.
Choose a Borel map $\psi \colon \cal{S}A\to \cal{S}$ such that $\psi(x)=e_x$ for any $x\in A$, where $e_x$ is the unit element at $x$; and $s\circ \psi(x)=x$ and $r\circ \psi(x)\in A$ for any $x\in \cal{S}A\setminus A$.
For each $n\in N$, we set
\[D_n=\{\, x\in A\cap D_{\phi_n}\mid r\circ \phi_n(x)\in \cal{S}A\,\}\]
and define a Borel map $\psi_n\colon D_n\to (\cal{G})_A$ by $\psi_n(x)=\psi(r\circ \phi_n(x))\phi_n(x)$ for $x\in D_n$.
Taking a countable Borel partition of $D_n$ and restricting $\psi_n$ to each piece, we obtain a countable family $\{ \eta_m\}_{m\in M}$ of elements of $\qn_{(\cal{G})_A}((\cal{S})_A)$ such that for any $n\in N$ and a.e.\ $x\in D_n$, there exists $m\in M$ with $x\in D_{\eta_m}$ and $\eta_m(x)=\psi_n(x)$.

For a.e.\ $g\in (\cal{G})_A$, there exists $n\in N$ with $s(g)\in D_{\phi_n}$ and $\phi_n(s(g))g^{-1}\in \cal{S}$.
Putting $x=s(g)$, we have $r\circ \phi_n(x)\in \cal{S}A$.
We can thus find $m\in M$ with $x\in D_{\eta_m}$ and $\psi_n(x)=\eta_m(x)$.
The equality $\eta_m(x)g^{-1}=\psi(r\circ \phi_n(x))\phi_n(x)g^{-1}$ holds, and this element belongs to $(\cal{S})_A$.
It follows that $(\cal{S})_A$ is quasi-normal in $(\cal{G})_A$.
\end{proof}

\begin{lem}\label{lem-qn-finite}
Let $\cal{G}$ be a discrete measured groupoid on a standard finite measure space $(X, \mu)$.
Let $\cal{S}$ and $\cal{T}$ be subgroupoids of $\cal{G}$ such that $\cal{S}<\cal{T}$ and $[\cal{T}:\cal{S}]_x<\infty$ for a.e.\ $x\in X$.
Then we have the equality $\qn_{\cal{G}}(\cal{S})=\qn_{\cal{G}}(\cal{T})$.
In particular, $\cal{S}$ is quasi-normal in $\cal{G}$ if and only if $\cal{T}$ is quasi-normal in $\cal{G}$.
\end{lem}

\begin{proof}
We denote by $r, s\colon \cal{G}\to X$ the range and source maps of $\cal{G}$, respectively.
Pick $\phi \in \qn_{\cal{G}}(\cal{S})$.
For any $x\in R_{\phi}$, let $E$ be a set of representatives of all classes in $(s^{-1}(x)\cap (\cal{T})_{R_{\phi}})/(\cal{S})_{R_{\phi}}$.
We have
\[[(\cal{T})_{R_{\phi}}: (\cal{T})_{R_{\phi}}\cap \cal{T}^{\phi}]_x\leq [(\cal{T})_{R_{\phi}}: (\cal{S})_{R_{\phi}}\cap \cal{S}^{\phi}]_x= \sum_{g\in E}[(\cal{S})_{R_{\phi}}: (\cal{S})_{R_{\phi}}\cap \cal{S}^{\phi}]_{r(g)},\]
where the last equality holds by Lemma \ref{lem-index-formula} (iii).
For a.e.\ $x\in R_{\phi}$, the right hand side is finite because $E$ is finite and $\phi$ belongs to $\qn_{\cal{G}}(\cal{S})$.
As in the last part in the proof of Lemma \ref{lem-qn} (ii), we can conclude that $[\cal{T}^{\phi}: (\cal{T})_{R_{\phi}}\cap \cal{T}^{\phi}]_x$ is also finite for a.e.\ $x\in R_{\phi}$.
It follows that $\phi$ belongs to $\qn_{\cal{G}}(\cal{T})$.

Pick $\psi \in \qn_{\cal{G}}(\cal{T})$.
For any $x\in R_{\psi}$, let $F$ be a set of representatives of all classes in $(s^{-1}(x)\cap (\cal{T})_{R_{\psi}})/((\cal{T})_{R_{\psi}}\cap \cal{T}^{\psi})$.
We have
\begin{align*}
& [(\cal{S})_{R_{\psi}}: (\cal{S})_{R_{\psi}}\cap \cal{S}^{\psi}]_x\leq [(\cal{T})_{R_{\psi}}: (\cal{S})_{R_{\psi}}\cap \cal{S}^{\psi}]_x\\
= \, & \sum_{g\in F}[(\cal{T})_{R_{\psi}}\cap \cal{T}^{\psi}: (\cal{S})_{R_{\psi}}\cap \cal{S}^{\psi}]_{r(g)}.
\end{align*}
For a.e.\ $x\in R_{\psi}$, the set $F$ is finite because $\psi$ belongs to $\qn_{\cal{G}}(\cal{T})$.
For any $g\in F$, putting $y=r(g)$, we have
\begin{align*}
& [(\cal{T})_{R_{\psi}}\cap \cal{T}^{\psi}: (\cal{S})_{R_{\psi}}\cap \cal{S}^{\psi}]_y\\
\leq \, & [(\cal{T})_{R_{\psi}}\cap \cal{T}^{\psi}: (\cal{S})_{R_{\psi}}\cap \cal{T}^{\psi}]_y[(\cal{T})_{R_{\psi}}\cap \cal{T}^{\psi}: \cal{S}^{\psi}\cap (\cal{T})_{R_{\psi}}]_y\\
\leq \, & [(\cal{T})_{R_{\psi}}: (\cal{S})_{R_{\psi}}]_y[\cal{T}^{\psi}: \cal{S}^{\psi}]_y, 
\end{align*}
where the first and second inequalities hold by Lemma \ref{lem-index-formula} (i) and (ii), respectively.
The right hand side is finite for a.e.\ $x\in R_{\psi}$.
We see that $[(\cal{S})_{R_{\psi}}: (\cal{S})_{R_{\psi}}\cap \cal{S}^{\psi}]_x$ is finite for a.e.\ $x\in R_{\psi}$.
Similarly, we can conclude that $[\cal{S}^{\psi}: (\cal{S})_{R_{\psi}}\cap \cal{S}^{\psi}]_x$ is finite for a.e.\ $x\in R_{\psi}$.
It follows that $\psi$ belongs to $\qn_{\cal{G}}(\cal{S})$.
\end{proof}


\subsection{Quotient}

Let $(X, \mu)$ be a standard finite measure space.
Let $\cal{G}$ be a discrete measured groupoid on $(X, \mu)$.
Given a normal subgroupoid $\cal{S}$ of $\cal{G}$, we can construct a discrete measured groupoid $\cal{Q}$ on a standard finite measure space $(Z, \xi)$ and a Borel homomorphism $\theta \colon \cal{G}\to \cal{Q}$ satisfying the following three conditions:
\begin{enumerate}
\item[(a)] The equality $\ker \theta =\cal{S}$ holds.
\item[(b)] For a.e.\ $h\in \cal{Q}$ and $x\in X$ such that $\theta(x)$ is equal to the source of $h$, there exists $g\in \cal{G}$ with $s(g)=x$ and $\theta(g)=h$, where the map from $X$ into $Z$ induced by $\theta$ is denoted by the same symbol $\theta$.
\item[(c)] If $\cal{Q}'$ is a discrete measured groupoid on a standard finite measure space $(Z', \xi')$ and if $\theta'\colon \cal{G}\to \cal{Q}'$ is a Borel homomorphism with $\cal{S}<\ker \theta'$, then there exists a Borel homomorphism $\tau \colon \cal{Q}\to \cal{Q}'$ with $\tau \circ \theta =\theta'$.
\end{enumerate}
The groupoid $\cal{Q}$ is called the {\it quotient} of $\cal{G}$ by $\cal{S}$ and denoted by $\cal{G}/\cal{S}$.
We refer to the proof of \cite[Theorem 2.2]{fsz} for the construction of $\cal{Q}$ and $\theta$ (see also \cite[Section 3]{sauer-thom}, where the quotient by a strongly normal subgroupoid is discussed).
Although in \cite{fsz}, the quotient is constructed in the case where $\cal{G}$ is principal, it is also valid in the general case.
In the construction, the map from $(X, \mu)$ into $(Z, \xi)$ induced by $\theta$ is defined as the ergodic decomposition for $\cal{S}$.
In particular, if $\cal{S}$ is ergodic, then $\cal{Q}$ is a discrete group.
The following lemma is deduced from the construction of the quotient.

\begin{lem}\label{lem-qu}
Let $\cal{G}$ be a discrete measured groupoid on a standard finite measure space $(X, \mu)$.
Let $\cal{H}$ be a normal subgroupoid of $\cal{G}$.
Then the following assertions hold:
\begin{enumerate}
\item Let $A$ be a Borel subset of $X$ with $X=\cal{H}A$.
Then the inclusion of $(\cal{G})_A$ into $\cal{G}$ induces an isomorphism from $(\cal{G})_A/(\cal{H})_A$ onto $\cal{G}/\cal{H}$.
\item Suppose that we have a non-singular action of a discrete group $\Gamma$ on $(X, \mu)$ and a normal subgroup $\Lambda$ of $\Gamma$ such that $\cal{G}=\Gamma \ltimes (X, \mu)$ and $\cal{H}=\Lambda \ltimes (X, \mu)$.
If the action of $\Lambda$ on $(X, \mu)$ is ergodic, then the projection from $\cal{G}$ onto $\Gamma$ induces an isomorphism from $\cal{G}/\cal{H}$ onto $\Gamma/\Lambda$.
\end{enumerate}
\end{lem}


\section{A variant of Furman's theorem}\label{sec-var}

Theorem \ref{thm-furman} below is used to construct a representation of a group ME to a given group.
It is proved by Furman \cite{furman-mer} in the framework of higher rank lattices, and plays a significant role to deduce ME and OE rigidity results in \cite{bfs}, \cite{furman-mer}, \cite{furman-oer}, \cite{kida-oer} and \cite{ms}, etc.
We provide a variant of this theorem to get a representation of a group ME to a Baumslag-Solitar group, into $\R$.
We first review measure equivalence and introduce terminology.

\begin{defn}[\ci{0.5.E}{gromov-as-inv}]\label{defn-me}
Two discrete groups $\Gamma$ and $\Lambda$ are said to be {\it measure equivalent (ME)} if we have a standard Borel space $(\Sigma, m)$ with a $\sigma$-finite positive measure and a measure-preserving action of $\Gamma \times \Lambda$ on $(\Sigma, m)$ such that there exist Borel subsets $X, Y\subset \Sigma$ satisfying $m(X)<\infty$, $m(Y)<\infty$ and the equality
\[\Sigma =\bigsqcup_{\gamma \in \Gamma}(\gamma, e)Y=\bigsqcup_{\lambda \in \Lambda}(e, \lambda)X\]
up to $m$-null sets.
The space $(\Sigma, m)$ equipped with the action of $\Gamma \times \Lambda$ is then called a {\it $(\Gamma, \Lambda)$-coupling}.
\end{defn}

ME is an equivalence relation between discrete groups (see \cite[Section 2]{furman-mer}).
It is known that two discrete groups $\Gamma$ and $\Lambda$ are ME if and only if there exists an ergodic f.f.m.p.\ action of $\Gamma$ which is WOE to an ergodic f.f.m.p.\ action of $\Lambda$, as discussed in \cite[Section 3]{furman-oer}.

Let $\Gamma$ and $\Lambda$ be discrete groups, and let $G$ be a standard Borel group.
Given homomorphisms $\pi \colon \Gamma \to G$ and $\rho \colon \Lambda \to G$, we denote by $(G, \pi, \rho)$ the Borel space $G$ equipped with the action of $\Gamma \times \Lambda$ on it defined by
\[(\gamma, \lambda)g=\pi(\gamma)g\rho(\lambda)^{-1},\quad g\in G,\ \gamma \in \Gamma,\ \lambda \in \Lambda.\]
Let $\Sigma$ be a $(\Gamma, \Lambda)$-coupling, and let $\Phi \colon \Sigma \to S$ be a Borel map into a standard Borel space $S$ on which $\Gamma \times \Lambda$ acts.
We say that $\Phi$ is {\it almost $(\Gamma \times \Lambda)$-equivariant} if we have the equality
\[\Phi((\gamma, \lambda)x)=(\gamma, \lambda)\Phi(x),\quad \forall \gamma \in \Gamma,\ \forall \lambda \in \Lambda,\ \textrm{a.e.}\ x\in \Sigma.\] 
For a $(\Gamma, \Lambda)$-coupling $\Sigma$, we denote by $\Sigma \times_{\Lambda}\Sigma$ the quotient space of $\Sigma \times \Sigma$ by the diagonal action of $\Lambda$, which is a $(\Gamma, \Gamma)$-coupling.
We refer to \cite[Theorem 2.5]{bfs} for the proof of the following:

\begin{thm}\label{thm-furman}
Let $\Gamma$ be a discrete group, $G$ a standard Borel group and $\pi \colon \Gamma \to G$ a homomorphism.
Let $\Lambda$ be a discrete group, and let $(\Sigma, m)$ be a $(\Gamma, \Lambda)$-coupling.
We set $\Omega =\Sigma \times_{\Lambda}\Sigma$.
Suppose that
\begin{enumerate}
\item[(a)] the Dirac measure on the neutral element of $G$ is the only probability measure on $G$ invariant under conjugation by any element of $\pi(\Gamma)$; and
\item[(b)] we have an almost $(\Gamma \times \Gamma)$-equivariant Borel map from $\Omega$ into $(G, \pi, \pi)$.
\end{enumerate}
Then there exist a homomorphism $\rho \colon \Lambda \rightarrow G$ and an almost $(\Gamma \times \Lambda)$-equivariant Borel map $\Phi \colon \Sigma \rightarrow (G, \pi, \rho)$.
In addition, if $\ker \pi$ is finite and there is a Borel fundamental domain for the action of $\pi(\Gamma)$ on $G$ by left multiplication, then $\rho$ can be chosen so that $\ker \rho$ is finite.
\end{thm}

We say that a measure-preserving action of a discrete group $\Gamma$ on a probability space $(X, \mu)$ is {\it weakly mixing} if the diagonal action of $\Gamma$ on the product $(X\times X, \mu \times \mu)$ is ergodic.
This condition is known to imply that for any ergodic measure-preserving action of $\Gamma$ on a probability space $(Y, \nu)$, the diagonal action of $\Gamma$ on the product $(X\times Y, \mu \times \nu)$ is ergodic (see \cite[Proposition 2.2]{sch}).
We note that for any $(\Gamma, \Lambda)$-coupling $\Sigma$, the action $\Gamma \times \Gamma \c \Sigma \times_{\Lambda}\Sigma$ is ergodic if and only if the action $\Lambda \c \Sigma /\Gamma$ is weakly mixing.

\begin{thm}\label{thm-abel}
Let $\Gamma$ be a discrete group, $G$ an abelian standard Borel group and $\pi \colon \Gamma \to G$ a homomorphism.
Let $\Lambda$ be a discrete group, and let $(\Sigma, m)$ be a $(\Gamma, \Lambda)$-coupling.
We set $\Omega =\Sigma \times_{\Lambda}\Sigma$.
Suppose that
\begin{enumerate}
\item[(1)] the action $\Lambda \c \Sigma /\Gamma$ is weakly mixing; and
\item[(2)] we have an almost $(\Gamma \times \Gamma)$-equivariant Borel map from $\Omega$ into $(G, \pi, \pi)$.
\end{enumerate}
Then there exist a homomorphism $\rho \colon \Lambda \rightarrow G$ and an almost $(\Gamma \times \Lambda)$-equivariant Borel map $\Phi \colon \Sigma \to (G, \pi, \rho)$.
\end{thm}

\begin{proof}
For $(x, y)\in \Sigma \times \Sigma$, we denote by $[x, y]\in \Omega$ the equivalence class of $(x, y)$.
Let $\Psi \colon \Omega \to (G, \pi, \pi)$ be an almost $(\Gamma \times \Gamma)$-equivariant Borel map.

\begin{claim}\label{claim-h}
Define a Borel map $H\colon \Sigma^4\to G$ by
\[H(x, y, z, w)=\Psi([y, z])\Psi([x, z])^{-1}\Psi([x, w])\Psi([y, w])^{-1}\]
for $(x, y, z, w)\in \Sigma^4$.
Then $H$ is essentially constant.
\end{claim}

\begin{proof}
For any $\gamma \in \Gamma$, $\lambda \in \Lambda$ and $(x, y, z, w)\in \Sigma^4$, we have
\begin{align*}
H(x, y, z, w)&=H(\gamma x, y, z, w)=H(x, \gamma y, z, w)=H(x, y, \gamma z, w)=H(x, y, z, \gamma w)\\
&=H(\lambda x, \lambda y, \lambda z, \lambda w)
\end{align*}
because $G$ is abelian.
It follows that $H$ induces a Borel map $\bar{H}\colon (\Sigma /\Gamma)^4\to G$ which is invariant under the diagonal action of $\Lambda$ on $(\Sigma /\Gamma)^4$.
This action of $\Lambda$ is ergodic because the action $\Lambda \c \Sigma /\Gamma$ is weakly mixing.
The map $\bar{H}$ is therefore essentially constant, and so is $H$.
\end{proof}

Let $g_0\in G$ denote the essential value of the map $H$.
For $z\in \Sigma$, we define a Borel map $F_z\colon \Sigma^2\to G$ by
\[F_z(x, y)=\Psi([x, z])\Psi([y, z])^{-1}\]
for $(x, y)\in \Sigma^2$.
By Claim \ref{claim-h} and Fubini's theorem, for a.e.\ $x\in \Sigma$ and any $\lambda \in \Lambda$, we have
\[F_z(x, y)^{-1}F_w(x, y)=g_0=F_z(\lambda^{-1}x, y)^{-1}F_w(\lambda^{-1}x, y)\]
for a.e.\ $(y, z, w)\in \Sigma^3$.
For a.e.\ $x\in \Sigma$ and any $\lambda \in \Lambda$, the equality
\[F_z(\lambda^{-1}x, y)F_z(x, y)^{-1}=\Psi([\lambda^{-1}x, z])\Psi([x, z])^{-1}\]
for any $y\in \Sigma$ implies that the Borel map $\Sigma \ni z\mapsto \Psi([\lambda^{-1}x, z])\Psi([x, z])^{-1}\in G$ is essentially constant.
We define $\rho_x(\lambda)\in G$ to be the essential value of this map.
For a.e.\ $x\in \Sigma$ and any $\lambda_1, \lambda_2\in \Lambda$, choosing some $z\in \Sigma$, we have
\begin{align*}
\rho_x(\lambda_1\lambda_2^{-1})&=\Psi([\lambda_2\lambda_1^{-1}x, z])\Psi([x, z])^{-1}=\Psi([\lambda_1^{-1}x, \lambda_2^{-1}z])\Psi([x, z])^{-1}\\
&=\Psi([\lambda_1^{-1}x, \lambda_2^{-1}z])\Psi([x, \lambda_2^{-1}z])^{-1}\Psi([x, \lambda_2^{-1}z])\Psi([\lambda_2^{-1}x, \lambda_2^{-1}z])^{-1}\\
&=\rho_x(\lambda_1)\rho_x(\lambda_2)^{-1}.
\end{align*}
For a.e.\ $x\in \Sigma$, the map $\rho_x\colon \Lambda \to G$ is therefore a homomorphism.
There exists an element $x_0$ of $\Sigma$ such that $\rho_{x_0}\colon \Lambda \to G$ is a homomorphism and we have the equality $\rho_{x_0}(\lambda)=\Psi([\lambda^{-1}x_0, x])\Psi([x_0, x])^{-1}$ for any $\lambda \in \Lambda$ and a.e.\ $x\in \Sigma$.
We define a Borel map $\Phi \colon \Sigma \to G$ by $\Phi(x)=\Psi([x_0, x])^{-1}$ for $x\in \Sigma$.
The map $\Phi \colon \Sigma \to (G, \pi, \rho)$ is then almost $(\Gamma \times \Lambda)$-equivariant.
\end{proof}


\section{Elliptic subgroupoids}\label{sec-ell}

We set $\Gamma =\bs(p, q)$ with $2\leq |p|\leq |q|$.
Let $T$ be the Bass-Serre tree associated with $\Gamma$.
Suppose that we have a measure-preserving action of $\Gamma$ on a standard finite measure space $(X, \mu)$.
We set $\cal{G}=\Gamma \ltimes X$ and define a homomorphism $\rho \colon \cal{G}\to \Gamma$ by $\rho(\gamma, x)=\gamma$ for $(\gamma, x)\in \cal{G}$.
Let $A$ be a Borel subset of $X$ with positive measure.
We say that a subgroupoid $\cal{S}$ of $(\cal{G})_A$ is {\it elliptic} if there exists an $(\cal{S}, \rho)$-invariant Borel map from $A$ into $V(T)$.

For any $v\in V(T)$, the subgroupoid $\Gamma_v\ltimes X$ of $\cal{G}$ is of infinite type, amenable and quasi-normal in $\cal{G}$ because $\Gamma_v$ is quasi-normal in $\Gamma$.
Conversely, Theorem \ref{thm-ell} below says that these algebraic properties imply ellipticity.
Let us say that a discrete measured groupoid $\cal{H}$ on a standard finite measure space $(Y, \nu)$ is {\it nowhere amenable} if for any Borel subset $B$ of $Y$ with positive measure, $(\cal{H})_B$ is not amenable.

\begin{thm}\label{thm-ell}
We set $\Gamma =\bs(p, q)$ with $2\leq |p|\leq |q|$.
Suppose that we have a measure-preserving action of $\Gamma$ on a standard finite measure space $(X, \mu)$, and set $\cal{G}=\Gamma \ltimes X$.
Let $A$ be a Borel subset of $X$ with positive measure, and let $\cal{S}$ and $\cal{T}$ be subgroupoids of $(\cal{G})_A$ such that we have $\cal{S}<\cal{T}$; $\cal{S}$ is amenable and is quasi-normal in $\cal{T}$; and $\cal{T}$ is nowhere amenable.
Then $\cal{S}$ is elliptic.
\end{thm}

Before proving this theorem, we prepare the following:

\begin{notation}\label{not-tree}
For a locally compact Polish space $K$, we denote by $M(K)$ the space of probability measures on $K$ equipped with the weak* topology.
Let $V_f(T)$ denote the set of non-empty finite subsets of $V(T)$.
Let $S(T)$ denote the set of simplices of $T$.
We have the $\aut(T)$-equivariant map $C\colon V_f(T)\to S(T)$ associating to each element of $V_f(T)$ its barycenter (see \cite[Section 4.2]{kida-exama} for a precise definition).
We also have the $\aut(T)$-equivariant Borel map $M\colon M(V(T))\to V_f(T)$ associating to each $\nu \in M(V(T))$ the set of all elements of $V(T)$ attaining the maximal value of the function $\nu$ on $V(T)$.
We set
\[\delta T=\{\, (x, y, z)\in (\partial T)^3\mid x\neq y\neq z\neq x\,\},\]
where $\Gamma$ acts by the formula $\gamma (x, y, z)=(\gamma x, \gamma y, \gamma z)$ for $\gamma \in \Gamma$ and $(x, y, z)\in \delta T$.
We define a $\Gamma$-equivariant Borel map $G\colon \delta T\rightarrow V(T)$ so that for each $(x, y, z)\in \delta T$, $G(x, y, z)$ is the intersection of the three geodesics in $T$ joining two of $x$, $y$ and $z$.
We define $\partial_2 T$ as the quotient of $\partial T\times \partial T$ by the action of the symmetric group of two letters that exchanges the coordinates.
\end{notation}

\begin{proof}[Proof of Theorem \ref{thm-ell}]
This proof is similar to the proof of \cite[Lemmas 4.2, 4.4 and 4.5]{kida-exama} except for using quasi-normality in place of normality.
The argument in \cite{kida-exama} partially depends on the proof of \cite[Lemma 3.2]{adams} (see also \cite[Chapter 2]{hjorth-kechris}).
Let $\rho \colon \cal{G}\to \Gamma$ be the homomorphism defined by $\rho(\gamma, x)=\gamma$ for $(\gamma, x)\in \cal{G}$.
For a Borel subset $B$ of $A$ with positive measure, we define $I_B$ as the set of all $(\cal{S}, \rho)$-invariant Borel maps from $B$ into $V(T)$.
It is enough to deduce a contradiction under the assumption that there exists a Borel subset $B$ of $A$ with positive measure such that for any Borel subset $B_1$ of $B$ with positive measure, $I_{B_1}$ is empty.
Since $\cal{S}$ is amenable, there exists an $(\cal{S}, \rho)$-invariant Borel map from $A$ into $M(\partial T)$.

\begin{lem}\label{lem-two}
Let $B_1$ be a Borel subset of $B$ with positive measure, and let $\varphi \colon B_1\to M(\partial T)$ be an $(\cal{S}, \rho)$-invariant Borel map.
Then for a.e.\ $x\in B_1$, the measure $\varphi(x)$ is supported on at most two points of $\partial T$.
\end{lem}

\begin{proof}
If the lemma were not true, then there would exist a Borel subset $B_2$ of $B_1$ with positive measure such that for a.e.\ $x\in B_2$, the restriction of the measure $\varphi(x)^3$ on $(\partial T)^3$ to $\delta T$ is non-zero.
Composing the map assigning to each $x\in B_2$ the normalization of the restriction of $\varphi(x)^3$ to $\delta T$, we obtain an $(\cal{S}, \rho)$-invariant Borel map from $B_2$ into $M(\delta T)$.
Composing the maps $G$, $M$ and $C$, we obtain an $(\cal{S}, \rho)$-invariant Borel map from $B_2$ into $S(T)$.
We also obtain an $(\cal{S}, \rho)$-invariant Borel map from $B_2$ into $V(T)$ because $\Gamma$ acts on $T$ without inversions.
This contradicts our assumption that $I_{B_2}$ is empty.
\end{proof}

Lemma \ref{lem-two} implies that any $(\cal{S}, \rho)$-invariant Borel map $\varphi \colon B\to M(\partial T)$ induces an $(\cal{S}, \rho)$-invariant Borel map from $B$ into $\partial_2T$.
We define $J$ as the set of all $(\cal{S}, \rho)$-invariant Borel maps from $B$ into $\partial_2T$, which is non-empty.
For each $\varphi \in J$, we set
\[S_{\varphi}=\{\, x\in B\mid |{\rm supp}(\varphi(x))|=2\,\},\]
where ${\rm supp}(\nu)$ denotes the support of a measure $\nu$.
There exists an element $\varphi_0$ of $J$ with $\mu(S_{\varphi_0})=\sup_{\varphi \in J}\mu(S_{\varphi})$.

\begin{lem}\label{lem-fi-inv}
In the above notation, let $B_1$ be a Borel subset of $B$ with positive measure.
Let $\cal{S}_0$ be a subgroupoid of $(\cal{S})_{B_1}$ with $[(\cal{S})_{B_1}: \cal{S}_0]_x<\infty$ for a.e.\ $x\in B_1$.
Then for any $(\cal{S}_0, \rho)$-invariant Borel map $\varphi \colon B_1\to \partial_2 T$, we have $\varphi(x)\subset \varphi_0(x)$ for a.e.\ $x\in B_1$.
\end{lem}

\begin{proof}
We naturally identify $\partial_2 T$ with a subset of $\cal{F}(\partial T)$, the space of non-empty finite subsets of $\partial T$.
Let $\varphi \colon B_1\to \partial_2 T$ be an $(\cal{S}_0, \rho)$-invariant Borel map.
By Lemma \ref{lem-finite-index}, there exists an $(\cal{S}, \rho)$-invariant Borel map $\Phi \colon B_1\to \cal{F}(\partial T)$ such that $\varphi(x)\subset \Phi(x)$ for a.e.\ $x\in B_1$.
Since each element of $\cal{F}(\partial T)$ naturally associates a probability measure on $\partial T$, applying Lemma \ref{lem-two}, we have $\Phi(x)\in \partial_2 T$ for a.e.\ $x\in B_1$.
By the maximality of $\mu(S_{\varphi_0})$, we have $\Phi(x)\subset \varphi_0(x)$ for a.e.\ $x\in B_1$.
It thus turns out that $\varphi(x)\subset \varphi_0(x)$ for a.e.\ $x\in B_1$.
The lemma is proved.
\end{proof}

We claim that the map $\varphi_0$ is $(\cal{T}, \rho)$-invariant.
Pick $\psi \in \qn_{(\cal{T})_B}((\cal{S})_B)$.
It suffices to show that for a.e.\ $x\in D_{\psi}$, we have the equality $\rho(\psi(x))\varphi_0(x)=\varphi_0(r\circ \psi(x))$.
We define a Borel map $\chi \colon D_{\psi}\to \partial_2 T$ by $\chi(x)=\rho(\psi(x))^{-1}\varphi_0(r\circ \psi(x))$ for $x\in D_{\psi}$.
For a.e.\ $g\in (\cal{S})_{D_{\psi}}\cap U_{\psi}^{-1}((\cal{S})_{R_{\psi}})$ with $x=s(g)$ and $y=r(g)$, we have
\[\rho(g)\chi(x)=\rho(\psi(y))^{-1}\rho(\psi(y)g\psi(x)^{-1})\varphi_0(r\circ \psi(x))=\rho(\psi(y))^{-1}\varphi_0(r\circ \psi(y))=\chi(y),\]
where the second equality holds because $U_{\psi}(g)=\psi(y)g\psi(x)^{-1}$ belongs to $\cal{S}$.
Since $\psi$ is in $\qn_{(\cal{T})_B}((\cal{S})_B)$, we have $[(\cal{S})_{D_{\psi}}: (\cal{S})_{D_{\psi}}\cap U_{\psi}^{-1}((\cal{S})_{R_{\psi}})]_x<\infty$ for a.e.\ $x\in D_{\psi}$.
By Lemma \ref{lem-fi-inv}, the inclusion $\chi(x)\subset \varphi_0(x)$ holds for a.e.\ $x\in D_{\psi}$.

Let $\zeta \in \qn_{(\cal{T})_B}((\cal{S})_B)$ be the inverse of $\psi$.
We next define a Borel map $\omega \colon D_{\zeta}\to \partial_2 T$ by $\omega(y)=\rho(\zeta(y))^{-1}\varphi_0(r\circ \zeta(y))$ for $y\in D_{\zeta}$.
The argument in the last paragraph implies the inclusion $\omega(y)\subset \varphi_0(y)$ for a.e.\ $y\in D_{\zeta}$.
For a.e.\ $x\in D_{\psi}$, we have
\begin{align*}
\chi(x)&=\rho(\psi(x))^{-1}\varphi_0(r\circ \psi(x))\subset \varphi_0(x)=\rho(\psi(x))^{-1}\omega(r\circ \psi(x))\\
&\subset \rho(\psi(x))^{-1}\varphi_0(r\circ \psi(x))=\chi(x)
\end{align*}
and thus $\chi(x)=\varphi_0(x)$.
The claim is proved.

The action of $\Gamma$ on $\partial_2T$ is amenable in a measure-theoretic sense by \cite[Corollary 3.4]{kida-exama}.
The existence of the $(\cal{T}, \rho)$-invariant Borel map $\varphi_0\colon B\to \partial_2T$ therefore implies that $(\cal{T})_B$ is amenable by \cite[Proposition 2.5]{kida-exama}.
This is a contradiction because $\cal{T}$ is nowhere amenable.
\end{proof}

In the following theorem, we obtain a result similar to Theorem \ref{thm-ell} for a discrete group having an infinite amenable normal subgroup with the quotient hyperbolic.

\begin{thm}\label{thm-ell-hyp}
Let $\Delta$ be a discrete group, and let $N$ be an infinite, amenable and normal subgroup of $\Delta$ such that $\Delta /N$ is non-elementarily hyperbolic.
Suppose that we have a measure-preserving action of $\Delta$ on a standard finite measure space $(X, \mu)$.
We set $\cal{G}=\Delta \ltimes X$.
Let $A$ be a Borel subset of $X$ with positive measure, and let $\cal{S}$ and $\cal{T}$ be subgroupoids of $(\cal{G})_A$ such that we have $\cal{S}<\cal{T}$; $\cal{S}$ is amenable and is quasi-normal in $\cal{T}$; and $\cal{T}$ is nowhere amenable.

Then there exist a countable Borel partition $A=\bigsqcup_n A_n$ and a subgroup $L_n$ of $\Delta$ such that for each $n$, we have $N<L_n$, $[L_n: N]<\infty$ and $(\cal{S})_{A_n}<(L_n\ltimes X)_{A_n}$.
\end{thm}

\begin{proof}
The proof is essentially the same as that of Theorem \ref{thm-ell}.
We put $Q=\Delta /N$ and denote by $\partial Q$ the boundary of $Q$ as a hyperbolic metric space.
Let $Q$ act on itself by left multiplication.
This action of $Q$ extends to a continuous action on the compactification $Q\cup \partial Q$ of $Q$.
As before, for a locally compact Polish space $K$, we denote by $M(K)$ the space of probability measures on $K$ equipped with the weak* topology.
Let $Q_f$ denote the set of non-empty finite subsets of $Q$.
We have the $Q$-equivariant Borel map $M\colon M(Q)\to Q_f$ associating to each $\nu \in M(Q)$ the set of all elements of $Q$ attaining the maximal value of the function $\nu$ on $Q$.
We set
\[\delta Q =\{\, (x, y, z)\in (\partial Q)^3\mid x\neq y\neq z\neq x\,\},\]
where $Q$ acts by the formula $\gamma (x, y, z)=(\gamma x, \gamma y, \gamma z)$ for $\gamma \in Q$ and $(x, y, z)\in \delta Q$.
The set $\delta Q$ is non-empty because $Q$ is non-elementary.
In \cite[Definition 6.6]{adams-hyp}, Adams constructs a $Q$-equivariant Borel map $MS\colon \delta Q \to Q_f$, using hyperbolicity of $Q$.
We define $\partial_2Q$ as the quotient of $\partial Q \times \partial Q$ by the action of the symmetric group of two letters that exchanges the coordinates.

Let $\rho \colon \cal{G}\to Q$ be the composition of the projection from $\cal{G}$ onto $\Delta$ with the quotient map from $\Delta$ onto $Q$.
For a Borel subset $B$ of $A$ with positive measure, we define $I_B$ as the set of all $(\cal{S}, \rho)$-invariant Borel maps from $B$ into $Q_f$.
Assume that there is a Borel subset $B$ of $A$ with positive measure such that for any Borel subset $B_1$ of $B$ with positive measure, $I_{B_1}$ is empty.
Along the proof of Theorem \ref{thm-ell}, we can deduce a contradiction, using the fact that the action of $Q$ on $\partial_2Q$ is amenable in a measure-theoretic sense, proved in \cite[Theorem 5.1]{adams-hyp}.
It follows that $I_A$ is non-empty, that is, there exists an $(\cal{S}, \rho)$-invariant Borel map $\psi \colon A\to Q_f$.
Take a countable Borel partition $A=\bigsqcup_n A_n$ such that for each $n$, the map $\psi$ is constant on $A_n$.
Since the stabilizer of each element of $Q_f$ in $Q$ is finite, for each $n$, there exists a subgroup $L_n$ of $\Delta$ such that $N<L_n$, $[L_n: N]<\infty$ and $(\cal{S})_{A_n}<(L_n\ltimes X)_{A_n}$.
\end{proof}


\section{The modular cocycles}\label{sec-mod}

Given a finite-measure-preserving, discrete measured groupoid $\cal{G}$ and its quasi-normal subgroupoid $\cal{S}$, we introduce the modular cocycle of Radon-Nikodym type and the local-index cocycle in Sections \ref{subsec-rn} and \ref{subsec-lic}, respectively.
They are defined by measuring a difference between $\cal{S}$ and its conjugate by an element of $\qn_{\cal{G}}(\cal{S})$.
In Section \ref{subsec-comp}, these two cocycles are computed when $\cal{G}$ and $\cal{S}$ are associated with an action of a Baumslag-Solitar group and its restriction to an elliptic subgroup.

\subsection{The modular cocycle of Radon-Nikodym type}\label{subsec-rn}

In this subsection, unless otherwise stated, let $\cal{G}$ be a measure-preserving, discrete measured groupoid on a standard finite measure space $(X, \mu)$.
Let $\cal{S}$ be a quasi-normal subgroupoid of $\cal{G}$.
Fix $\phi \in \qn_{\cal{G}}(\cal{S})$ with $\mu(D_{\phi})>0$.
We first define $\D(\phi, x)\in \Rm$ for $x\in D_{\phi}$.
We will define $\D \colon \cal{G}\to \Rm$, called the modular cocycle of Radon-Nikodym type for $\cal{G}$ and $\cal{S}$, so that $\D(\phi(x))=\D(\phi, x)$ for a.e.\ $x\in D_{\phi}$.

We put $D=D_{\phi}$, $R=R_{\phi}$ and $U=U_{\phi}$.
We also put
\[\cal{S}_-=(\cal{S})_D\cap U^{-1}((\cal{S})_R),\quad \cal{S}_+=(\cal{S})_R\cap U((\cal{S})_D).\]
The restriction of $U$ is an isomorphism from $\cal{S}_-$ onto $\cal{S}_+$.
Unless there is a confusion, for each $x\in D$, we write $U(x)=r\circ \phi(x)\in R$.
Let $\pi \colon (X, \mu)\to (Z, \xi)$ be the ergodic decomposition for $\cal{S}$.
Let $\mu =\int_Z \mu_z d\xi(z)$ be the disintegration with respect to $\pi$.
For $z\in Z$, we put $X_z=\pi^{-1}(z)$.
Let
\[\pi_-\colon (D, \mu|_D)\to (Z_-, \xi_-)\quad \textrm{and}\quad \pi_+\colon (R, \mu|_R)\to (Z_+, \xi_+)\]
be the ergodic decompositions for $\cal{S}_-$ and $\cal{S}_+$, respectively, with $(\pi_-)_*(\mu|_D)=\xi_-$ and $(\pi_+)_*(\mu|_R)=\xi_+$.
For $z\in Z_-$, we put $D_z=\pi_-^{-1}(z)$, and for $w\in Z_+$, we put $R_w=\pi_+^{-1}(w)$.
Let $Y_-$ and $Y_+$ be the Borel subsets of $Z$ with $\cal{S}D=\pi^{-1}(Y_-)$ and $\cal{S}R=\pi^{-1}(Y_+)$.
We define $\pi_D\colon D\to Y_-$ and $\pi_R\colon R\to Y_+$ as the restrictions of $\pi$.
We set $\eta_-=(\pi_D)_*(\mu|_D)$ and $\eta_+=(\pi_R)_*(\mu|_R)$.
Let $\sigma_-\colon Z_-\to Y_-$ and $\sigma_+\colon Z_+\to Y_+$ be the canonical Borel maps such that the following diagrams commute:
\[\xymatrix{
& \ar[dl]_{\pi_-} (D, \mu|_D) \ar[dr]^{\pi_D} & \\
(Z_-, \xi_-) \ar[rr]^{\sigma_-} & & (Y_-, \eta_-) \\
}
\quad
\xymatrix{
& \ar[dl]_{\pi_+} (R, \mu|_R) \ar[dr]^{\pi_R} & \\
(Z_+, \xi_+) \ar[rr]^{\sigma_+} & & (Y_+, \eta_+) \\
}
\]
For $y\in Y_-$, we put $(Z_-)_y=\sigma_-^{-1}(y)$, and for $y\in Y_+$, we put $(Z_+)_y=\sigma_+^{-1}(y)$.

The map $\pi_D$ is the ergodic decomposition for $(\cal{S})_D$.
We have the disintegration
\[\mu|_D=\int_{Y_-}\mu_y(D\cap X_y)^{-1}\mu_y|_D\, d\eta_-(y)\]
with respect to $\pi_D$ because the equality $\eta_-=\chi\xi|_{Y_-}$ holds, where $\chi$ is the Borel function on $Y_-$ defined by $\chi(y)=\mu_y(D\cap X_y)$ for $y\in Y_-$.
Applying Lemma \ref{lem-erg-dec} to $(\cal{S})_D$ and $\cal{S}_-$, we may assume that for a.e.\ $y\in Y_-$, the set $(Z_-)_y$ is finite.
Applying Lemma \ref{lem-rest} to $\pi_D$, $\pi_-$ and $\sigma_-$, we see that for a.e.\ $z\in Z_-$, the restriction $\mu_{\sigma_-(z)}|_{D_z}$ is a constant multiple of the ergodic measure for $(\cal{S}_-)_z$.

In the same manner, we may assume that for a.e.\ $y\in Y_+$, the set $(Z_+)_y$ is finite, and we see that for a.e.\ $z\in Z_+$, the restriction $\mu_{\sigma_+(z)}|_{R_z}$ is a constant multiple of the ergodic measure for $(\cal{S}_+)_z$.
Since $U$ is an isomorphism from $\cal{S}_-$ onto $\cal{S}_+$, for a.e.\ $x\in D$, the two measures on the set $U(D_{\pi_-(x)})=R_{\pi_+(U(x))}$,
\[U_*(\mu_{\pi(x)}|_{D_{\pi_-(x)}})\quad \textrm{and}\quad \mu_{\pi(U(x))}|_{R_{\pi_+(U(x))}},\]
are a constant multiple of each other.
We define a number $\D(\phi, x)=\D(\cal{G}, \cal{S}, \phi, x)\in \Rm$ by the equality 
\[U_*(\mu_{\pi(x)}|_{D_{\pi_-(x)}})=\D(\phi, x)\mu_{\pi(U(x))}|_{R_{\pi_+(U(x))}}.\]

\begin{lem}\label{lem-d-res}
If $A$ is a Borel subset of $D$ with $\mu(A)>0$, then we have the equality $\D(\phi|_A, x)=\D(\phi, x)$ for a.e.\ $x\in A$.
\end{lem}

\begin{proof}
Let $W_-$ be the Borel subset of $Z_-$ with $\pi_-^{-1}(W_-)=\cal{S}_- A$.
The map $\rho_-\colon A\to W_-$ defined as the restriction of $\pi_-$ is the ergodic decomposition for $(\cal{S}_-)_A$.
Putting $B=U(A)$, we define $W_+$ and $\rho_+\colon B\to W_+$ similarly.
For a.e.\ $x\in A$, restricting the equality in the definition of $\D(\phi, x)$ to $B$, we obtain the equality
\[U_*(\mu_{\pi(x)}|_{D_{\pi_-(x)}\cap A})=\D(\phi, x)\mu_{\pi(U(x))}|_{R_{\pi_+(U(x))}\cap B}.\]
Putting $A_z=\rho_-^{-1}(z)$ for $z\in W_-$ and $B_w=\rho_+^{-1}(w)$ for $w\in W_+$, we have $D_{\pi_-(x)}\cap A=A_{\rho_-(x)}$ and $R_{\pi_+(U(x))}\cap B=B_{\rho_+(U(x))}$ for a.e.\ $x\in A$.
The lemma follows.
\end{proof}

We define a Borel function $\D =\D(\cal{G}, \cal{S})\colon \cal{G}\to \Rm$ as follows.
Pick a countable family $\Phi$ of elements of $\qn_{\cal{G}}(\cal{S})$ such that $\mu(D_{\phi})>0$ for any $\phi \in \Phi$ and the equality $\cal{G}=\bigsqcup_{\phi \in \Phi}\phi(D_{\phi})$ holds.
For $g\in \cal{G}$, choose $\phi \in \Phi$ with $g\in \phi(D_{\phi})$, and set $\D(g)=\D(\phi, s(g))$.
This definition of $\D$ does not depend on the choice of $\Phi$ by Lemma \ref{lem-d-res}.
We call $\D$ the {\it modular cocycle of Radon-Nikodym type} for $\cal{G}$ and $\cal{S}$.

\begin{lem}
The function $\D \colon \cal{G}\to \Rm$ defined above is a Borel cocycle, that is, it preserves products.
\end{lem}

\begin{proof}
Pick $\phi, \psi \in \qn_{\cal{G}}(\cal{S})$ and put $\eta =\psi \bullet \phi$, $D=D_{\eta}$, $R=R_{\eta}$ and $U=U_{\eta}$.
We assume $\mu(D)>0$ and the equality $D_{\psi}=R_{\phi}=U_{\phi}(D)$.
Put $L=U_{\phi}(D)$ and
\[\cal{S}_-=(\cal{S})_D\cap U^{-1}((\cal{S})_R),\quad \cal{S}_-'=(\cal{S})_D\cap U_{\phi}^{-1}((\cal{S})_L),\quad \cal{S}_-''=(\cal{S})_L\cap U_{\psi}^{-1}((\cal{S})_R).\]
For a.e.\ $x\in D$, since $[(\cal{S})_D: \cal{S}_-]_x$ is finite, the set $D\cap X_{\pi(x)}$ is decomposed into finitely many ergodic components for $\cal{S}_-$ by Lemma \ref{lem-erg-dec}.
Similarly, for a.e.\ $x\in D$, since $[(\cal{S})_D: \cal{S}_-']_x$ is finite, the set $D\cap X_{\pi(x)}$ is decomposed into finitely many ergodic components for $\cal{S}_-'$.
Let $Z_1$ be the Borel subset of $Z$ with $\cal{S}D=\pi^{-1}(Z_1)$.
For a.e.\ $z\in Z_1$, pick ergodic components $D_-$ and $D_-'$ for $\cal{S}_-$ and $\cal{S}_-'$, respectively, with $D_-, D_-'\subset X_z$ and $\mu_z(D_-\cap D_-')>0$.
For $\mu_z$-a.e.\ $x\in D_-'$, the equality
\[(U_{\phi})_*(\mu_{\pi(x)}|_{D_-'})=\D(\phi, x)\mu_{\pi(U_{\phi}(x))}|_{U_{\phi}(D_-')}\]
holds by the definition of $\D(\phi, x)$.

For a.e.\ $y\in L$, since $[(\cal{S})_L: \cal{S}_-'']_y$ is finite, the set $L\cap X_{\pi(y)}$ is decomposed into finitely many ergodic components for $\cal{S}_-''$.
Let $D_-''$ be an ergodic component for $\cal{S}_-''$ with $\mu_{z'}(U_{\phi}(D_-\cap D_-')\cap D_-'')>0$, where $z'$ is the point of $Z$ with $U_{\phi}(D_-')\subset X_{z'}$.
Such $z'$ exists because we have $U_{\phi}(\cal{S}_-')\subset (\cal{S})_L$.
For $\mu_{z'}$-a.e.\ $y\in D_-''$, the equality
\[(U_{\psi})_*(\mu_{\pi(y)}|_{D_-''})=\D(\psi, y)\mu_{\pi(U_{\psi}(y))}|_{U_{\psi}(D_-'')}\]
holds by the definition of $\D(\psi, y)$.
Combining the above two equalities, we obtain $\D(\eta, x)=\D(\psi, U_{\phi}(x))\D(\phi, x)$ for $\mu_z$-a.e.\ $x\in U_{\phi}^{-1}(U_{\phi}(D_-\cap D_-')\cap D_-'')$.
\end{proof}

The following two lemmas will be used to compute the cocycle $\D$.

\begin{lem}\label{lem-d-finite}
Let $\cal{S}$ and $\cal{T}$ be subgroupoids of $\cal{G}$ such that $\cal{S}<\cal{T}$; $[\cal{T}:\cal{S}]_x<\infty$ for a.e.\ $x\in X$; and $\cal{T}$ is quasi-normal in $\cal{G}$.
We set $\D_{\cal{S}}=\D(\cal{G}, \cal{S})$ and $\D_{\cal{T}}=\D(\cal{G}, \cal{T})$.
Then there exists a Borel map $\psi \colon X\to \Rm$ with the equality $\D_{\cal{S}}(g)=\psi(r(g))\D_{\cal{T}}(g)\psi(s(g))^{-1}$ for a.e.\ $g\in \cal{G}$.
\end{lem}

\begin{proof}
Note that by Lemma \ref{lem-qn-finite}, $\cal{S}$ is quasi-normal in $\cal{G}$.
Let $\pi \colon (X, \mu)\to (Z, \xi)$ and $\theta \colon (X, \mu)\to (W, \omega)$ be the ergodic decompositions for $\cal{S}$ and $\cal{T}$, respectively.
Let $\mu =\int_Z \mu_z d\xi(z)$ and $\mu =\int_W\nu_w d\omega(w)$ be the disintegrations with respect to $\pi$ and $\theta$, respectively.
We have the canonical Borel map $\tau \colon (Z, \xi)\to (W, \omega)$ with $\theta =\tau \circ \pi$.
Pick an element $\phi$ of the set $\qn_{\cal{G}}(\cal{T})=\qn_{\cal{G}}(\cal{S})$ with $\mu(D_{\phi})>0$.
We put $D=D_{\phi}$, $R=R_{\phi}$ and $U=U_{\phi}$, and define $\cal{S}_-$, $\cal{S}_+$, $\pi_-$ and $\pi_+$ as in the definition of $\D(\cal{G}, \cal{S})$.
Similarly, we set
\[\cal{T}_-=(\cal{T})_D\cap U^{-1}((\cal{T})_R),\quad \cal{T}_+=(\cal{T})_R\cap U((\cal{T})_D).\]
Let $\theta_-\colon (D, \mu|_D)\to (W_-, \omega_-)$ and $\theta_+\colon (R, \mu|_R)\to (W_+, \omega_+)$ denote the ergodic decompositions for $\cal{T}_-$ and $\cal{T}_+$, respectively.

Since we have $[\cal{T}:\cal{S}]_x<\infty$ for a.e.\ $x\in X$, the set $\tau^{-1}(w)$ is finite for a.e.\ $w\in W$ by Lemma \ref{lem-erg-dec}.
We define a Borel map $\psi \colon X\to \Rm$ by $\psi(x)=\nu_{\theta(x)}(X_{\pi(x)})$ for $x\in X$.
Applying Lemma \ref{lem-rest} to $\theta$, $\pi$ and $\tau$, we have the equality $\nu_{\theta(x)}|_{X_{\pi(x)}}=\psi(x)\mu_{\pi(x)}$ for a.e.\ $x\in X$.
Combining the two equalities
\begin{align*}
U_*(\mu_{\pi(x)}|_{D_{\pi_-(x)}})&=\D_{\cal{S}}(\phi, x)\mu_{\pi(U(x))}|_{R_{\pi_+(U(x))}},\\
U_*(\nu_{\theta(x)}|_{D_{\theta_-(x)}})&=\D_{\cal{T}}(\phi, x)\nu_{\theta(U(x))}|_{R_{\theta_+(U(x))}}
\end{align*}
defining $\D_{\cal{S}}(\phi, x)$ and $\D_{\cal{T}}(\phi, x)$, we obtain the equality in the lemma.
\end{proof}

\begin{lem}\label{lem-d-cohom}
Let $A$ be a Borel subset of $X$ with $\mu(A)>0$.
We put $\D=\D(\cal{G}, \cal{S})$ and $\D_A=\D((\cal{G})_A, (\cal{S})_A)$.
Then there exists a Borel map $\psi \colon A\to \Rm$ with the equality $\D_A(g)=\psi(r(g))\D(g)\psi(s(g))^{-1}$ for a.e.\ $g\in (\cal{G})_A$.
\end{lem}

\begin{proof}
We use the same notation as in the beginning of this subsection to define $\D(\phi, x)$ for $x\in D$.
Fix $\phi \in \qn_{\cal{G}}(\cal{S})$ with $\mu(D_{\phi})>0$.
Put $D=D_{\phi}$ and $R=R_{\phi}$.
We assume $D\subset A$ and $R\subset A$.
Put $\nu =\mu|_A$.
Let $\theta \colon (A, \nu)\to (W, \omega)$ be the ergodic decomposition for $(\cal{S})_A$ with $\theta_*\nu =\omega$.
Let $\nu=\int_W \nu_w d\omega(w)$ be the disintegration with respect to $\theta$.
To define the number $\D_A(\phi, x)=\D((\cal{G})_A, (\cal{S})_A, \phi, x)$, we use $\theta$ and $\{ \nu_w\}_w$ in place of $\pi$ and $\{ \mu_z\}_z$, and use the same $\cal{S}_-$, $\cal{S}_+$, $\pi_-$ and $\pi_+$ because we have $D\subset A$ and $R\subset A$.
For a.e.\ $x\in D$, we thus have the equality
\[U_*(\nu_{\theta(x)}|_{D_{\pi_-(x)}})=\D_A(\phi, x)\nu_{\theta(U(x))}|_{R_{\pi_+(U(x))}}.\]
Let $Z_0$ be the Borel subset of $Z$ with $\pi^{-1}(Z_0)=\cal{S}A$.
Let $\alpha \colon W\to Z_0$ be the Borel isomorphism with $\pi =\alpha \circ \theta$ on $A$.
We have the disintegration
\[\nu=\int_W\mu_{\alpha(w)}(A\cap X_{\alpha(w)})^{-1}\mu_{\alpha(w)}|_A\, d\omega(w)\]
with respect to $\theta$ because the equality $\alpha_*\omega =\chi \xi|_{Z_0}$ holds, where $\chi$ is the Borel function on $Z_0$ defined by $\chi(z)=\mu_z(A\cap X_z)$ for $z\in Z_0$.
For a.e.\ $w\in W$, we have $\mu_{\alpha(w)}|_A=\mu_{\alpha(w)}(A\cap X_{\alpha(w)})\nu_w$ by uniqueness of disintegration.
We define a Borel map $\psi \colon A\to \Rm$ by $\psi(x)=\mu_{\pi(x)}(A\cap X_{\pi(x)})$ for $x\in A$.
Combining the two equalities defining $\D(\phi, x)$ and $\D_A(\phi, x)$, we obtain the equality in the lemma.
\end{proof}


\subsection{The local-index cocycle}\label{subsec-lic}

We fix a standard finite measure space $(X, \mu)$ and a measure-preserving, discrete measured groupoid $\cal{G}$ on $(X, \mu)$.
Let $\cal{S}$ be a quasi-normal subgroupoid of $\cal{G}$.
We define a Borel function $\I=\I(\cal{G}, \cal{S})\colon \cal{G}\to \Rm$ as follows.
Pick a countable family $\Phi$ of elements of $\qn_{\cal{G}}(\cal{S})$ such that $\mu(D_{\phi})>0$ for any $\phi \in \Phi$ and the equality $\cal{G}=\bigsqcup_{\phi \in \Phi}\phi(D_{\phi})$ holds.
For $\phi \in \Phi$, we put
\[\cal{S}_-=(\cal{S})_{D_{\phi}}\cap U_{\phi}^{-1}((\cal{S})_{R_{\phi}}),\quad \cal{S}_+=(\cal{S})_{R_{\phi}}\cap U_{\phi}((\cal{S})_{D_{\phi}}).\]
For $g\in \cal{G}$, choosing $\phi \in \Phi$ with $g\in \phi(D_{\phi})$, we define the number $\I(g)\in \Rm$ by
\begin{align*}
\I(g)&=[[(\cal{S})_{R_{\phi}}:\cal{S}_+]]_{r(g)}[[(\cal{S})_{D_{\phi}}:\cal{S}_-]]_{s(g)}^{-1}\\
&=[[(\cal{S})_{R_{\phi}}:\cal{S}_+]]_{r(g)}[[U_{\phi}((\cal{S})_{D_{\phi}}):\cal{S}_+]]_{r(g)}^{-1}.
\end{align*}
By Lemma \ref{lem-li-res}, this definition of $\I$ does not depend on the choice of $\Phi$.
We call $\I$ the {\it local-index cocycle} for $\cal{G}$ and $\cal{S}$.

\begin{lem}
The function $\I \colon \cal{G}\to \Rm$ defined above is a Borel cocycle, that is, it preserves products.
\end{lem}

\begin{proof}
For each $\phi \in \qn_{\cal{G}}(\cal{S})$, we set $\cal{S}^{\phi}=U_{\phi}((\cal{S})_{D_{\phi}})$.
Pick two elements $\phi$, $\psi$ of $\qn_{\cal{G}}(\cal{S})$, and put $\eta =\psi \bullet \phi$, $D=D_{\eta}$ and $R=R_{\eta}$.
We assume $\mu(D)>0$ and the equalities $D=D_{\phi}$ and $R=R_{\psi}$.
For a.e.\ $x\in D$, putting $y=r\circ \eta(x)$, we obtain the equality
\[
\I(\psi, r\circ \phi(x))=\frac{[[(\cal{S})_R :(\cal{S})_R\cap \cal{S}^{\psi}]]_y}{[[\cal{S}^{\psi}:(\cal{S})_R\cap \cal{S}^{\psi}]]_y}=\frac{[[(\cal{S})_R:(\cal{S})_R\cap \cal{S}^{\psi}\cap \cal{S}^{\eta}]]_y}{[[\cal{S}^{\psi}:(\cal{S})_R\cap \cal{S}^{\psi}\cap \cal{S}^{\eta}]]_y},
\]
where we apply Lemma \ref{lem-li-prod} to the second equality.
Applying the same lemma, we also obtain the equality
\[
\I(\phi, x)=\frac{[[\cal{S}^{\psi}:\cal{S}^{\psi}\cap \cal{S}^{\eta}]]_y}{[[\cal{S}^{\eta}:\cal{S}^{\psi}\cap \cal{S}^{\eta}]]_y}=\frac{[[\cal{S}^{\psi}:(\cal{S})_R\cap \cal{S}^{\psi}\cap \cal{S}^{\eta}]]_y}{[[\cal{S}^{\eta}:(\cal{S})_R\cap \cal{S}^{\psi}\cap \cal{S}^{\eta}]]_y}
\]
for a.e.\ $x\in D$ with $y=r\circ \eta(x)$.
Combining these two equalities, we have
\begin{align*}
\I(\psi, r\circ \phi(x))\I(\phi, x)&=\frac{[[(\cal{S})_R:(\cal{S})_R\cap \cal{S}^{\psi}\cap \cal{S}^{\eta}]]_y}{[[\cal{S}^{\eta}:(\cal{S})_R\cap \cal{S}^{\psi}\cap \cal{S}^{\eta}]]_y}=\frac{[[(\cal{S})_R:(\cal{S})_R\cap \cal{S}^{\eta}]]_y}{[[\cal{S}^{\eta}:(\cal{S})_R\cap \cal{S}^{\eta}]]_y}\\
&=\I(\eta, x)
\end{align*}
for a.e.\ $x\in D$ with $y=r\circ \eta(x)$.
The function $\I$ is therefore a Borel cocycle.
\end{proof}

\begin{lem}\label{lem-lic-res}
The following assertions hold:
\begin{enumerate}
\item Let $A$ be a Borel subset of $X$ with $\mu(A)>0$.
We put $\I_A=\I((\cal{G})_A, (\cal{S})_A)$.
Then the equality $\I_A(g)=\I(g)$ holds for a.e.\ $g\in (\cal{G})_A$. 
\item Let $\cal{T}$ be a subgroupoid of $\cal{G}$ such that $\cal{S}<\cal{T}$; $[\cal{T}:\cal{S}]_x<\infty$ for a.e.\ $x\in X$; and $\cal{T}$ is quasi-normal in $\cal{G}$.
We set $\I_{\cal{T}}=\I(\cal{G}, \cal{T})$.
Then there exists a Borel map $\psi \colon X\to \Rm$ with the equality $\I_{\cal{T}}(g)=\psi(r(g))\I(g)\psi(s(g))^{-1}$ for a.e.\ $g\in \cal{G}$.
\end{enumerate}
\end{lem}

\begin{proof}
Assertion (i) follows from Lemma \ref{lem-li-res}.
We prove assertion (ii).
We pick an element $\phi$ of the set $\qn_{\cal{G}}(\cal{T})=\qn_{\cal{G}}(\cal{S})$ with $\mu(D_{\phi})>0$.
Put $D=D_{\phi}$ and $R=R_{\phi}$.
Define a Borel map $\psi \colon X\to \Rm$ by $\psi(x)=[[\cal{T}: \cal{S}]]_x$ for $x\in X$.
For a.e.\ $x\in D$, putting $y=r\circ \phi(x)$, we obtain the equality
\begin{align*}
\I_{\cal{T}}(\phi(x))&=\frac{[[(\cal{T})_R: (\cal{T})_R\cap \cal{T}^{\phi}]]_y}{[[\cal{T}^{\phi}: (\cal{T})_R\cap \cal{T}^{\phi}]]_y}=\frac{[[(\cal{T})_R: (\cal{S})_R\cap \cal{S}^{\phi}]]_y}{[[\cal{T}^{\phi}: (\cal{S})_R\cap \cal{S}^{\phi}]]_y}\\
&=\frac{[[(\cal{T})_R: (\cal{S})_R]]_y}{[[\cal{T}^{\phi}: \cal{S}^{\phi}]]_y}\frac{[[(\cal{S})_R: (\cal{S})_R\cap \cal{S}^{\phi}]]_y}{[[\cal{S}^{\phi}: (\cal{S})_R\cap \cal{S}^{\phi}]]_y}=\psi(y)\I(\phi(x))\psi(x)^{-1},
\end{align*}
where we apply Lemma \ref{lem-li-prod} to the second and third equalities, and apply Lemma \ref{lem-li-res} to the fourth equality.
Assertion (ii) follows.
\end{proof}


\subsection{Computation}\label{subsec-comp}

For a finite-measure-preserving action of a Baumslag-Solitar group $\Gamma$, we show relationship between the modular homomorphism $\m \colon \Gamma \to \Rm$ and the cocycles $\D$ and $\I$ for the groupoid associated with the action of $\Gamma$ and its elliptic subgroupoid.

\begin{lem}\label{lem-comp}
Let $\Gamma$ be a discrete group, and let $N$ be a quasi-normal subgroup of $\Gamma$.
Suppose that we have a measure-preserving action of $\Gamma$ on a standard finite measure space $(X, \mu)$.
We set $\cal{G}=\Gamma \ltimes X$ and $\cal{N}=N\ltimes X$, and set $\D=\D(\cal{G}, \cal{N})$ and $\I=\I(\cal{G}, \cal{N})$.
For any $\gamma \in \Gamma$, if both $N\cap \gamma N\gamma^{-1}$ and $N\cap \gamma^{-1}N\gamma$ are normal subgroups of $N$, then for a.e.\ $x\in X$, we have the equality
\[\D(\gamma, x)\I(\gamma, x)=[N: N\cap \gamma N\gamma^{-1}][N: N\cap \gamma^{-1}N\gamma]^{-1}.\]
\end{lem}

\begin{proof}
Fix $\gamma \in \Gamma$ and define $\phi \colon X\to \cal{G}$ by $\phi(x)=(\gamma, x)$ for $x\in X$.
We then have $\phi \in \qn_{\cal{G}}(\cal{N})$.
Let $\pi \colon (X, \mu)\to (Z, \xi)$ be the ergodic decomposition for $\cal{N}$ with $\pi_*\mu=\xi$.
Let $\mu =\int_Z \mu_z d\xi(z)$ be the disintegration with respect to $\pi$.
We set $N_-=N\cap \gamma^{-1}N\gamma$ and $N_+=N\cap \gamma N\gamma^{-1}$.
Similarly, we set $\cal{N}_-=N_-\ltimes X$ and $\cal{N}_+=N_+\ltimes X$.
Let $\pi_-\colon (X, \mu)\to (Z_-, \xi_-)$ and $\pi_+\colon (X, \mu)\to (Z_+, \xi_+)$ be the ergodic decompositions for $\cal{N}_-$ and $\cal{N}_+$, respectively.
We have the canonical Borel maps $\sigma_-\colon (Z_-, \xi_-)\to (Z, \xi)$ and $\sigma_+\colon (Z_+, \xi_+)\to (Z, \xi)$ with $\pi =\sigma_-\circ \pi_-=\sigma_+\circ \pi_+$.
The group $N$ transitively acts on $\sigma_-^{-1}(\pi(x))$ and on $\sigma_+^{-1}(\pi(y))$ for a.e.\ $x\in X$ and a.e.\ $y\in X$ because $N_-$ and $N_+$ are normal subgroups of $N$.
By Lemma \ref{lem-li-group}, the equalities
\[[[\cal{N}:\cal{N}_-]]_x=\frac{[N:N_-]}{|\sigma_-^{-1}(\pi(x))|},\quad [[\cal{N}:\cal{N}_+]]_y=\frac{[N:N_+]}{|\sigma_+^{-1}(\pi(y))|}\]
hold for a.e.\ $x\in X$ and a.e.\ $y\in X$.
On the other hand, we have
\[\mu_{\pi(x)}(\pi_-^{-1}(\pi_-(x)))=|\sigma_-^{-1}(\pi(x))|^{-1},\quad \mu_{\pi(y)}(\pi_+^{-1}(\pi_+(y)))=|\sigma_+^{-1}(\pi(y))|^{-1}\]
for a.e.\ $x\in X$ and a.e.\ $y\in X$.
The equality $\D(\gamma, x)\I(\gamma, x)=[N: N_+][N:N_-]^{-1}$ thus holds for a.e.\ $x\in X$.
\end{proof}

\begin{cor}\label{cor-bs-comp}
We set $\Gamma =\bs(p, q)$ with $2\leq |p|\leq |q|$ and denote by $\m \colon \Gamma \to \Rm$ the modular homomorphism.
Let $E$ be an infinite elliptic subgroup of $\Gamma$.
Suppose that we have a measure-preserving action of $\Gamma$ on a standard finite measure space $(X, \mu)$.
We set $\cal{G}=\Gamma \ltimes X$ and $\cal{E}=E\ltimes X$, and set $\D=\D(\cal{G}, \cal{E})$ and $\I=\I(\cal{G}, \cal{E})$.
Then we have the equality $\D(\gamma, x)\I(\gamma, x)=\m(\gamma)$ for any $\gamma \in \Gamma$ and a.e.\ $x\in X$.
\end{cor}


\section{The Mackey range of the modular cocycle}\label{sec-mackey}

Fundamental ideas of the Mackey range of a cocycle are discussed in \cite[Section 6]{mackey} and \cite[Section 7]{ramsay}, where it is called the ``range closure'' of the cocycle.
Let $\cal{G}$ be a discrete measured groupoid on a standard finite measure space $(X, \mu)$.
Let $H$ be a locally compact second countable group and $\tau \colon \cal{G}\to H$ a Borel cocycle.
We define the Mackey range of $\tau$ as follows.
Define a discrete measured equivalence relation $\cal{R}$ on $X\times H$ by
\[\cal{R}=\{\, ((r(g), \tau(g)h), (s(g), h))\in (X\times H)^2\mid g\in \cal{G},\ h\in H\,\}.\]
Let $Z$ denote the space of ergodic components for $\cal{R}$.
The action of $H$ on $X\times H$ defined by $h(x, h')=(x, h'h^{-1})$ for $x\in X$ and $h, h'\in H$ induces an action of $H$ on $Z$.
This action $H\c Z$ is called the {\it Mackey range} of $\tau$.
We can check the following two lemmas that are basic facts on Mackey ranges.

\begin{lem}\label{lem-mackey-res}
In the above notation, let $A$ be a Borel subset of $X$ with $\cal{G}A=X$.
We denote by $H\c Z_A$ the Mackey range of the restriction of $\tau$ to $(\cal{G})_A$.
Then the inclusion of $A\times H$ into $X\times H$ induces an isomorphism between the two actions $H\c Z_A$ and $H\c Z$.
Namely, it induces a measure space isomorphism $\eta \colon Z_A\to Z$ such that for any $h\in H$, we have $\eta(hz)=h\eta(z)$ for a.e.\ $z\in Z_A$. 
\end{lem}

\begin{lem}\label{lem-mackey-cohom}
In the above notation, the following assertions hold:
\begin{enumerate}
\item Let $f$ be a Borel automorphism of $\cal{G}$.
Define a Borel cocycle $\tau_f\colon \cal{G}\to H$ by $\tau_f=\tau\circ f$.
Then the automorphism of $X\times H$ sending $(x, h)$ to $(f^{-1}(x), h)$ for $x\in X$ and $h\in H$ induces an isomorphism between the Mackey ranges of $\tau$ and $\tau_f$.
\item Let $\varphi \colon X\to H$ be a Borel map.
Define a Borel cocycle $\tau_{\varphi}\colon \cal{G}\to H$ by $\tau_{\varphi}(g)=\varphi(r(g))\tau(g)\varphi(s(g))^{-1}$ for $g\in \cal{G}$.
Then the automorphism of $X\times H$ sending $(x, h)$ to $(x, \varphi(x)h)$ for $x\in X$ and $h\in H$ induces an isomorphism between the Mackey ranges of $\tau$ and $\tau_{\varphi}$.
\end{enumerate}
\end{lem}

In the rest of this section, we set $\Gamma =\bs(p, q)$ and $\Lambda =\bs(r, s)$ with $2\leq |p|\leq |q|$ and $2\leq |r|\leq |s|$, unless otherwise mentioned.
Let $\m \colon \Gamma \to \Rm$ and $\n \colon \Lambda \to \Rm$ denote the modular homomorphisms.

\begin{thm}\label{thm-r}
Let $(\Sigma, m)$ be a $(\Gamma, \Lambda)$-coupling.
Then there exists an almost $(\Gamma \times \Lambda)$-equivariant Borel map $\Phi \colon \Sigma \to (\R, \log \circ \m, \log \circ \n)$.
\end{thm}

\begin{proof}
We may assume that the action of $\Gamma \times \Lambda$ on $(\Sigma, m)$ is ergodic by \cite[Corollary 3.6]{fmw}.
Let $X, Y\subset \Sigma$ be fundamental domains for the actions of $\Lambda$ and $\Gamma$ on $\Sigma$, respectively.
Putting $Z=X\cap Y$, we may assume that $Z$ has positive measure.
Let $\alpha \colon \Gamma \times X\to \Lambda$ be the ME cocycle associated to $X$, defined by the condition $(\gamma, \alpha(\gamma, x))x\in X$ for $\gamma \in \Gamma$ and $x\in X$.
Since $X$ is identified with $\Sigma /(\{ e\} \times \Lambda)$, we have an ergodic measure-preserving action $\Gamma \c X$, where $X$ is equipped with the finite measure that is the restriction of $m$.
To distinguish this action from the original action on $\Sigma$, we use a dot for the new action, that is, we denote $(\gamma, \alpha(\gamma, x))x$ by $\gamma \cdot x$.
Similarly, we have an ergodic measure-preserving action $\Lambda \c Y$ and use a dot for this action.
We set $\cal{G}=\Gamma \ltimes X$ and $\cal{H}=\Lambda \ltimes Y$.
The map $f\colon (\cal{G})_Z\to (\cal{H})_Z$ defined by $f(\gamma, x)=(\alpha(\gamma, x), x)$ for $(\gamma, x)\in (\cal{G})_Z$ is then an isomorphism.
Let $T_{\Gamma}$ and $T_{\Lambda}$ denote the Bass-Serre trees associated to $\Gamma$ and $\Lambda$, respectively.
For a vertex $u$ of $T_{\Gamma}$, we set $\cal{G}_u=\Gamma_u\ltimes X$.
Similarly, for a vertex $v$ of $T_{\Lambda}$, we set $\cal{H}_v=\Lambda_v\ltimes Y$.

We fix a vertex $u$ of $T_{\Gamma}$.
By Theorem \ref{thm-ell}, $f((\cal{G}_u)_Z)$ is an elliptic subgroupoid of $(\cal{H})_Z$.
It follows that there exist a Borel subset $Z_1$ of $Z$ with positive measure and a vertex $v$ of $T_{\Lambda}$ with $f((\cal{G}_u)_{Z_1})<(\cal{H}_v)_{Z_1}$.
Repeating this argument for $f^{-1}$ and $\cal{H}_v$, we find a Borel subset $A$ of $Z_1$ with positive measure and a vertex $u'$ of $T_{\Gamma}$ with $(\cal{H}_v)_A<f((\cal{G}_{u'})_A)$.
Putting $\cal{E}=f^{-1}((\cal{H}_v)_A)$, we have
\[((\Gamma_u\cap \Gamma_{u'})\ltimes X)_A<\cal{E}<(\Gamma_{u'}\ltimes X)_A.\]

We put $\D_{\Gamma}=\D((\cal{G})_A, \cal{E})$ and $\I_{\Gamma}=\I((\cal{G})_A, \cal{E})$.
Let $\rho \colon \cal{G}\to \Gamma$ be the projection.
Since $\Gamma_u\cap \Gamma_{u'}$ is a finite index subgroup of $\Gamma_{u'}$, by Lemmas \ref{lem-d-finite}, \ref{lem-d-cohom} and \ref{lem-lic-res} and Corollary \ref{cor-bs-comp}, there exists a Borel map $\varphi_0\colon A\to \Rm$ such that
\[\D_{\Gamma}(g)\I_{\Gamma}(g)=\varphi_0(r_{\cal{G}}(g))\m \circ \rho(g)\varphi_0(s_{\cal{G}}(g))^{-1}\]
for a.e.\ $g\in (\cal{G})_A$, where $r_{\cal{G}}, s_{\cal{G}}\colon \cal{G}\to X$ denote the range and source maps of $\cal{G}$, respectively.
Put $\D_{\Lambda}=\D((\cal{H})_A, (\cal{H}_v)_A)$ and $\I_{\Lambda}=\I((\cal{H})_A, (\cal{H}_v)_A)$.
Let $\sigma \colon \cal{H}\to \Lambda$ be the projection.
A similar property also holds for these maps.
On the other hand, the equality $f(\cal{E})=(\cal{H}_v)_A$ implies that $\D_{\Lambda}\circ f=\D_{\Gamma}$ and $\I_{\Lambda}\circ f=\I_{\Gamma}$.
It follows that there exists a Borel map $\varphi \colon A\to \Rm$ such that
\[\varphi(r_{\cal{G}}(g))\m \circ \rho(g)\varphi(s_{\cal{G}}(g))^{-1}=\n \circ \sigma\circ f(g)\]
for a.e.\ $g\in (\cal{G})_A$.
We thus have the equality
\[\varphi(\gamma \cdot x)\m(\gamma)\varphi(x)^{-1}=\n \circ \alpha(\gamma, x)\tag{$\ast$}\label{eq-cohom}\]
for any $\gamma \in \Gamma$ and a.e.\ $x\in A$ with $\gamma \cdot x\in A$.
We define a Borel map $\Phi \colon \Sigma \to \R$ by
\[\Phi((\gamma, \lambda)x)=\log (\m(\gamma)\varphi(x)^{-1}\n(\lambda)^{-1})\]
for $x\in A$, $\gamma \in \Gamma$ and $\lambda \in \Lambda$.
Using equation (\ref{eq-cohom}), we can show that $\Phi$ is well-defined and is a desired map, along the argument in the proof of \cite[Theorem 4.4]{kida-ama}.
\end{proof}

Theorem \ref{thm-flow} is obtained by combining Lemma \ref{lem-mackey-res}, Lemma \ref{lem-mackey-cohom} and equation (\ref{eq-cohom}).

\begin{rem}\label{rem-type}
Let $\Gamma \c (X, \mu)$ be an ergodic measure-preserving action on a standard finite measure space, and call it $\alpha$.
We denote by $\pi \colon (X, \mu)\to (Z, \xi)$ the ergodic decomposition for the action of $\ker \m$ on $(X, \mu)$.
We have the canonical action of $\m(\Gamma)$ on $(Z, \xi)$ and the action of $\log \circ \m(\Gamma)$ on it through the isomorphism $\log \colon \Rm \to \R$.
The flow associated with $\alpha$, introduced in Section \ref{sec-int}, is defined as the action of $\R$ induced from the action of $\log \circ \m(\Gamma)$ on $(Z, \xi)$.
If $|p|=|q|$, then the flow associated with $\alpha$ is isomorphic to the action of $\mathbb{R}$ on itself by addition because $\m$ is trivial.
If $|p|<|q|$, then we have the following two cases:
\begin{itemize}
\item If there exists a positive integer $n$ such that up to null sets, $Z$ consists of $n$ points each of whose measure is $\xi(Z)/n$, then the flow associated with $\alpha$ is isomorphic to the action of $\R$ on $\R/(-\log |p/q|^n)\mathbb{Z}$ by addition.
\item If $Z$ contains no point whose measure is positive, then the flow associated with $\alpha$ is finite-measure-preserving and any orbit of it is of measure zero.
\end{itemize}
Following terminology for transformations of type ${\rm III}$  (see Section \ref{sec-exotic}), in the first case, we say that the action $\Gamma \c (X, \mu)$ is of {\it type $|p/q|^n$}.
In the second case, we say that the action $\Gamma \c (X, \mu)$ is of {\it type $0$}.
\end{rem}

\begin{cor}
Assume $|p|<|q|$ and $|r|<|s|$.
If $\Gamma$ and $\Lambda$ are ME, then $\ker \m$ and $\ker \n$ are ME.
\end{cor}

\begin{proof}
We may assume $|q/p|\leq |s/r|$.
Let $\Sigma$ be a $(\Gamma, \Lambda)$-coupling.
There exists an almost $(\Gamma \times \Lambda)$-equivariant Borel map $\Phi \colon \Sigma \to (\R, \log \circ \m, \log \circ \n)$ by Theorem \ref{thm-r}.
We set $U=\{\, u\in \mathbb{R}\mid 0\leq u<\log |q/p|\,\}$, which is a fundamental domain for the action of $\log \circ \m(\Gamma)$ on $\R$ and is contained in that for the action of $\log \circ \n(\Lambda)$ on $\R$.
The set $\Phi^{-1}(U)$ then has positive measure and is a $(\ker \m, \ker \n)$-coupling.
\end{proof}

For discrete groups $G$ and $H$, two measure-preserving actions $G \c (Z, \xi)$ and $H \c (W, \omega)$ on measure spaces are called {\it conjugate} if there exist an isomorphism $F\colon G\to H$ and a measurable isomorphism $f$ from a conull measurable subset of $Z$ onto a conull measurable subset of $W$ such that $f_*\xi$ and $\omega$ are equivalent and for any $g\in G$, we have $f(gz)=F(g)f(z)$ for a.e.\ $z\in Z$.

\begin{cor}\label{cor-conj}
Assume $|p/q|=|r/s|$.
Let $\Gamma \c (X, \mu)$ and $\Lambda \c (Y, \nu)$ be ergodic f.f.m.p.\ actions.
Let $\pi \colon (X, \mu)\to (Z, \xi)$ and $\theta \colon (Y, \nu)\to (W, \omega)$ be the ergodic decompositions for the actions $\ker \m \c (X, \mu)$ and $\ker \n \c (Y, \nu)$, respectively.
If the actions $\Gamma \c (X, \mu)$ and $\Lambda \c (Y, \nu)$ are WOE, then the canonical actions $\m(\Gamma)\c (Z, \xi)$ and $\n(\Lambda)\c (W, \omega)$ are conjugate.
\end{cor}

\begin{proof}
We have the actions $\log \circ \m(\Gamma)\c (Z, \xi)$ and $\log \circ \n(\Lambda)\c (W, \omega)$ through the isomorphism $\log \colon \Rm\to \R$.
Let $\R \c (Z_1, \xi_1)$ and $\R \c (W_1, \omega_1)$ denote the actions induced from these two actions, respectively.
By Theorem \ref{thm-flow}, these two actions of $\R$ are isomorphic.
The assumption $|p/q|=|r/s|$ implies the equality $\log \circ \m(\Gamma)=\log \circ \n(\Lambda)$.
The action of $\log \circ \m(\Gamma)$ on almost every ergodic component for its action on $(Z_1, \xi_1)$ is isomorphic to the original action $\log \circ \m(\Gamma)\c (Z, \xi)$.
A similar property holds for the actions of $\log \circ \n(\Lambda)$ on $(W_1, \omega_1)$ and on $(W, \omega)$.
The corollary therefore follows.
\end{proof}

\begin{rem}
We set $\Gamma =\bs(p, q)$ with $2\leq |p|\leq |q|$, and suppose that we have an ergodic measure-preserving action of $\m(\Gamma)$ on a standard finite measure space $(Z, \xi)$.
We construct an ergodic f.f.m.p.\ action $\Gamma \c (X, \mu)$ such that the canonical action of $\m(\Gamma)$ on the space of ergodic components for the action $\ker \m \c (X, \mu)$ is isomorphic to the given action on $(Z, \xi)$.
Let $\Gamma \c (Y, \nu)$ be an f.f.m.p.\ action such that its restriction to $\ker \m$ is ergodic.
We set $(X, \mu)=(Y, \nu)\times (Z, \xi)$.
Define an action of $\Gamma$ on $(X, \mu)$ by $\gamma(y, z)=(\gamma y, \m(\gamma)z)$ for $\gamma \in \Gamma$, $y\in Y$ and $z\in Z$.
This is shown to be a desired action.

By definition, any associated flow is induced from an ergodic measure-preserving action of $\log \circ \m(\Gamma)$ on a standard finite measure space.
Thanks to the construction discussed above, we can find an ergodic f.f.m.p.\ action of $\Gamma$ whose associated flow is isomorphic to such a given action of $\R$.
\end{rem}

\begin{proof}[Proof of Theorem \ref{thm-z}]
We set $\Gamma =\bs(p, q)$ with $2\leq |p|<|q|$, and suppose that an ergodic f.f.m.p.\ action $\Gamma \c (X, \mu)$ is WOE to a weakly mixing f.f.m.p.\ action $\Delta \c (Y, \nu)$ of a discrete group $\Delta$.
Let $\Sigma$ be the $(\Gamma, \Delta)$-coupling associated with the WOE between the actions $\Gamma \c (X, \mu)$ and $\Delta \c (Y, \nu)$.
Combining Theorems \ref{thm-abel} and \ref{thm-r}, we obtain a homomorphism $\delta \colon \Delta \to \R$ and an almost $(\Gamma \times \Delta)$-equivariant Borel map $\Phi \colon \Sigma \to (\R, \log \circ \m, \delta)$.

If $\delta$ were trivial, then $\Phi$ would induce a $\Gamma$-equivariant Borel map from $X=\Sigma /\Delta$ into $\R$.
The group $\log \circ \m(\Gamma)$ is infinite and discrete in $\R$.
This is a contradiction because we have the $\Gamma$-invariant finite measure $\mu$ on $X$.
The homomorphism $\delta$ is therefore non-trivial.

Put $M=(\log |q/p|)\mathbb{Z}$.
The map $\Phi$ induces a $\Delta$-equivariant Borel map $\bar{\Phi}$ from $Y=\Sigma /\Gamma$ into $\R /M$.
Define a Borel map $\Psi \colon Y\times Y\to \R/M$ by $\Psi(x, y)=\bar{\Phi}(x)-\bar{\Phi}(y)$ for $x, y\in Y$.
Let $\Delta$ act on $Y\times Y$ diagonally.
The map $\Psi$ is $\Delta$-invariant and is therefore essentially constant because the action $\Delta \c (Y, \nu)$ is weakly mixing.
It follows that $\bar{\Phi}$ is essentially constant and that $\delta(\Delta)$ is contained in $M$.
\end{proof}

\begin{proof}[Proof of Theorem \ref{thm-bs-hyp}]
We set $\Gamma =\bs(p, q)$ with $2\leq |p|<|q|$.
Let $\Delta$ be a discrete group such that there is an infinite amenable normal subgroup $N$ of $\Delta$ with the quotient $\Delta /N$ non-elementarily word-hyperbolic.
Assuming that $\Gamma$ and $\Delta$ are ME, we deduce a contradiction.
Let $\Sigma$ be a $(\Gamma, \Delta)$-coupling such that the action of $\Gamma \times \Delta$ on $\Sigma$ is ergodic.
As in the proof of Theorem \ref{thm-r}, we choose fundamental domains $X, Y\subset \Sigma$ for the actions of $\Delta$ and $\Gamma$ on $\Sigma$, respectively, such that the intersection $Z=X\cap Y$ has positive measure.
We set $\cal{G}=\Gamma \ltimes X$ and $\cal{H}=\Delta \ltimes Y$.
Let $\alpha \colon \Gamma \times X\to \Delta$ be the ME cocycle associated to $X$.
We then have the isomorphism $f\colon (\cal{G})_Z\to (\cal{H})_Z$ defined by $f(\gamma, x)=(\alpha(\gamma, x), x)$ for $(\gamma, x)\in (\cal{G})_Z$.
Let $T$ denote the Bass-Serre tree associated to $\Gamma$.
For a vertex $v$ of $T$, set $\cal{G}_v=\Gamma_v\ltimes X$.

Since $(\cal{G}_v)_Z$ is amenable and quasi-normal in $(\cal{G})_Z$, by Theorem \ref{thm-ell-hyp}, there exist a Borel subset $Z_1$ of $Z$ with positive measure and a subgroup $L$ of $\Delta$ with $N<L$, $[L: N]<\infty$ and $f((\cal{G}_v)_{Z_1})<(L\ltimes X)_{Z_1}$.
Since $(L\ltimes X)_{Z_1}$ is also amenable and quasi-normal in $(\cal{H})_{Z_1}$, by Theorem \ref{thm-ell}, there exist a Borel subset $A$ of $Z_1$ with positive measure and a vertex $v'$ of $T$ with $f^{-1}((L\ltimes X)_A)<(\cal{G}_{v'})_A$.
Along the proof of Theorem \ref{thm-r}, using Lemma \ref{lem-comp}, we find a $(\Gamma \times \Delta)$-equivariant Borel map $\Phi \colon \Sigma \to (\R, \log \circ \m, \log \circ I)$, where $I\colon \Delta \to \Rm$ is the trivial homomorphism.
The map $\Phi$ induces a $\Gamma$-equivariant Borel map from $\Sigma /\Delta$ into $\R$.
This is a contradiction because $\log \circ \m(\Gamma)$ is infinite and discrete in $\R$.
\end{proof}


\section{Reduction of a coupling of Baumslag-Solitar groups}\label{sec-red}

For a discrete group $G$ acting on a tree $T$ so that the stabilizer of any simplex of $T$ in $G$ is infinite and cyclic, Gilbert, Howie, Metaftsis and Raptis present a sufficient condition for the automorphism group of $G$ to be canonically represented into the automorphism group of $T$, in \cite[Theorem A]{ghmr}.
In Section \ref{subsec-red}, relying on their argument, for any coupling of two Baumslag-Solitar groups satisfying a certain ergodicity condition, we construct an equivariant Borel map from the coupling into the space of isomorphisms between the associated Bass-Serre trees.
Throughout this section, we fix the following:

\begin{notation}\label{not-bs}
Let $p$, $q$, $r$ and $s$ be integers with $2\leq |p|<|q|$ and $2\leq |r|<|s|$.
Let $d_0>0$ denote the greatest common divisor of $p$ and $q$, and let $c_0>0$ denote the greatest common divisor of $r$ and $s$.
Putting $p_0=p/d_0$, $q_0=q/d_0$, $r_0=r/c_0$ and $s_0=s/c_0$, we assume that $q$ is not a multiple of $p$ and that $s$ is not a multiple of $r$.
We thus have $1<|p_0|<|q_0|$ and $1<|r_0|<|s_0|$.

We set $\Gamma =\bs(p, q)$ and $\Lambda =\bs(r, s)$.
Let $T_{\Gamma}$ and $T_{\Lambda}$ be the Bass-Serre trees associated with $\Gamma$ and $\Lambda$, respectively.
We denote by $\imath \colon \Gamma \to \aut(T_{\Gamma})$ and $\jmath \colon \Lambda \to \aut(T_{\Lambda})$ the homomorphisms associated with the actions on the trees.
Let $a\in \Gamma$ (resp.\ $b\in \Lambda$) be a generator of the stabilizer of a vertex in $T_{\Gamma}$ (resp.\ $T_{\Lambda}$). 
We define $\isom(T_{\Lambda}, T_{\Gamma})$ as the set of orientation-preserving isomorphisms from $T_{\Lambda}$ onto $T_{\Gamma}$, which are equipped with the standard Borel structure induced by the pointwise convergence topology.
The set $\isom(T_{\Gamma}, T_{\Lambda})$ is also defined similarly.
\end{notation}

\subsection{Reduction to isomorphisms between trees}\label{subsec-red}

For non-zero integers $m$, $n$, we say that a measure-preserving action of $\mathbb{Z}$ on a measure space is {\it $(m, n)$-ergodic} if for any non-negative integers $k$, $l$, the action of the subgroup $m^kn^l\mathbb{Z}$ is ergodic.

\begin{thm}\label{thm-aut-t}
In Notation \ref{not-bs}, let $(\Sigma, m)$ be a $(\Gamma, \Lambda)$-coupling.
We suppose that
\begin{enumerate}
\item[(1)] the action of $\langle a^{d_0}\rangle$ on almost every ergodic component for the action $\langle a\rangle \c \Sigma /\Lambda$ is $(p_0, q_0)$-ergodic; and
\item[(2)] the action of $\langle b^{c_0}\rangle$ on almost every ergodic component for the action $\langle b\rangle \c \Sigma /\Gamma$ is $(r_0, s_0)$-ergodic.
\end{enumerate}
Then there exists an essentially unique, almost $(\Gamma \times \Lambda)$-equivariant Borel map $\Phi \colon \Sigma \to (\isom(T_{\Lambda}, T_{\Gamma}), \imath, \jmath)$.
In particular, we have $|p|=|r|$ and $|q|=|s|$.
\end{thm}

\begin{lem}\label{lem-dirac}
In Notation \ref{not-bs}, the Dirac measure on the neutral element of $\aut(T_{\Gamma})$ is the only probability measure on $\aut(T_{\Gamma})$ invariant under conjugation by any element of $\imath(\Gamma)$.
\end{lem}

\begin{proof}
Let $\mu$ be a probability measure on $\aut(T_{\Gamma})$ invariant under conjugation by any element of $\imath(\Gamma)$.
Let $\Gamma =\langle \, a, t\mid ta^pt^{-1}=a^q\,\rangle$ be the presentation of $\Gamma$.
Note that $t$ acts on $T_{\Gamma}$ as a hyperbolic isometry.
Let $x_{\pm}\in \partial T_{\Gamma}$ denote the two fixed points of $t$.
We define a Borel map $F\colon \aut(T_{\Gamma})\to \partial T_{\Gamma}$ by $F(\varphi)=\varphi(x_+)$ for $\varphi \in \aut(T_{\Gamma})$.
Since $tx_+=x_+$, we have $F_*\mu=(F\circ \ad t)_*\mu=(\imath(t)\circ F)_*\mu$.
The measure $F_*\mu$ is thus invariant under the action of $t$.
Since the action of $\langle t\rangle$ on $\partial T_{\Gamma}\setminus \{ x_{\pm}\}$ is free and admits a Borel fundamental domain, we have $F_*\mu(\{ x_{\pm}\})=1$.
It follows that $\mu$ is supported on the set $\{\, \varphi \in \aut(T_{\Gamma})\mid \varphi(x_+)\in \{ x_{\pm}\}\,\}$.
A verbatim argument implies that $\mu$ is supported on the stabilizer of the set $\{ x_{\pm}\}$ in $\aut(T_{\Gamma})$.

The element $ata^{-1}\in \Gamma$ also acts on $T_{\Gamma}$ as a hyperbolic isometry.
In a similar way, if $y_{\pm}\in \partial T_{\Gamma}$ denote the two fixed points of $ata^{-1}$, then $\mu$ is supported on the stabilizer of the set $\{ y_{\pm}\}$ in $\aut(T_{\Gamma})$.
Let $v$ denote the vertex of $T_{\Gamma}$ such that the stabilizer of $v$ in $\Gamma$ is equal to $\langle a\rangle$.
Any element $\varphi$ of $\aut(T_{\Gamma})$ with $\varphi(\{ x_{\pm}\})=\{ x_{\pm}\}$ and $\varphi(\{ y_{\pm}\})=\{ y_{\pm}\}$ fixes $v$ because $v$ is the intersection of the geodesic in $T_{\Gamma}$ between $x_+$ and $x_-$ and the one between $y_+$ and $y_-$.
It follows that $\mu$ is supported on the stabilizer of $v$ in $\aut(T_{\Gamma})$.
Since $\Gamma$ acts on $V(T_{\Gamma})$ transitively, we conclude that $\mu$ is supported on the neutral element of $\aut(T_{\Gamma})$.
\end{proof}

In the notation in Theorem \ref{thm-aut-t}, let $X, Y\subset \Sigma$ be fundamental domains for the actions of $\Lambda$ and $\Gamma$ on $\Sigma$, respectively.
Let $\alpha \colon \Gamma \times X\to \Lambda$ be the ME cocycle defined by the condition $(\gamma, \alpha(\gamma, x))x\in X$ for $\gamma \in \Gamma$ and $x\in X$.
To distinguish the action $\Gamma \c X$ and the action $\Gamma \c \Sigma$, we use a dot for the former action, that is, we set $\gamma \cdot x=(\gamma, \alpha(\gamma, x))x$ for $\gamma \in \Gamma$ and $x\in X$.
Similarly, we use a dot for the action $\Lambda \c Y$. 
By \cite[Lemma 2.27]{kida-survey}, we may assume that the intersection $Z=X\cap Y$ satisfies the equality $\Gamma \cdot Z=X$ when $Z$ is regarded as a subset of $X$, and satisfies the equality $\Lambda \cdot Z=Y$ when $Z$ is regarded as a subset of $Y$.
We set $\cal{G}=\Gamma \ltimes X$ and $\cal{H}=\Lambda \ltimes Y$.
Let $f\colon (\cal{G})_Z\rightarrow (\cal{H})_Z$ be the isomorphism defined by $f(\gamma, x)=(\alpha(\gamma, x), x)$ for $(\gamma, x)\in (\cal{G})_Z$.
For each simplex $u$ of $T_{\Gamma}$, we put $\cal{G}_u=\Gamma_u\ltimes X$.
Similarly, for each simplex $v$ of $T_{\Lambda}$, we put $\cal{H}_v=\Lambda_v\ltimes Y$.

\begin{lem}\label{lem-vertex-stab}
If we have $u_1, u_2\in V(T_{\Gamma})$ and a Borel subset $W$ of $Z$ with positive measure satisfying the inclusion $(\cal{G}_{u_1})_W\subset (\cal{G}_{u_2})_W$, then the equality $u_1=u_2$ holds.
\end{lem}

\begin{proof}
Assuming that $u_1\neq u_2$, we deduce a contradiction.
Let $e$ be the edge of $T_{\Gamma}$ containing $u_1$ and contained in the geodesic between $u_1$ and $u_2$.
The inclusion in the lemma implies the equality $(\cal{G}_e)_W=(\cal{G}_{u_1})_W$.
On the other hand, by condition (1) in Theorem \ref{thm-aut-t} and Lemma \ref{lem-index} (ii), (iii), we have
$[(\cal{G}_{u_1})_W: (\cal{G}_e)_W]_x=[\Gamma_{u_1}: \Gamma_e]>1$ for a.e.\ $x\in W$.
This is a contradiction.
\end{proof}

\begin{lem}\label{lem-pre}
For any $u\in V(T_{\Gamma})$, there exist a countable Borel partition $Z=\bigsqcup_n Z_n$ and $v_n\in V(T_{\Lambda})$ such that $f((\cal{G}_u)_{Z_n})=(\cal{H}_{v_n})_{Z_n}$ for each $n$.
\end{lem}

\begin{proof}
By Lemmas \ref{lem-qn-group} and \ref{lem-qn-res}, $(\cal{G}_u)_Z$ is quasi-normal in $(\cal{G})_Z$, and it is also of infinite type and amenable.
It follows that $f((\cal{G}_u)_Z)$ is of infinite type, amenable and quasi-normal in $(\cal{H})_Z$.
Let $\sigma \colon \cal{H}\to \Lambda$ be the projection.
By Theorem \ref{thm-ell}, $f((\cal{G}_u)_Z)$ is elliptic, that is, there exists an $(f((\cal{G}_u)_Z), \sigma)$-invariant Borel map $\varphi \colon Z\to V(T_{\Lambda})$.
Take a countable Borel partition $Z=\bigsqcup_n Z_n$ such that for each $n$, the map $\varphi$ is essentially constant on $Z_n$, and denote by $v_n\in V(T_{\Lambda})$ the essential value of $\varphi$ on $Z_n$.
We obtain the inclusion $f((\cal{G}_u)_{Z_n})\subset (\cal{H}_{v_n})_{Z_n}$ for any $n$.
Applying the same argument to $f^{-1}$ in place of $f$ and taking a finer Borel partition of $Z$, for each $n$, we find $u_n\in V(T_{\Gamma})$ with the inclusion $f^{-1}((\cal{H}_{v_n})_{Z_n})\subset (\cal{G}_{u_n})_{Z_n}$.
By Lemma \ref{lem-vertex-stab}, for any $n$, the equality $u_n=u$ holds, and the equality in the lemma also holds.
\end{proof}

\begin{lem}\label{lem-edge-stab}
If we have $e_1, e_2\in E(T_{\Lambda})$ and a Borel subset $W$ of $Z$ with positive measure such that $(\cal{H}_{e_1})_W=(\cal{H}_{e_2})_W$, then the equality $\Lambda_{e_1}=\Lambda_{e_2}$ holds.
\end{lem}

\begin{proof}
Since we have $(\cal{H}_{e_1})_W=(\cal{H}_{e_1}\cap \cal{H}_{e_2})_W=(\cal{H}_{e_2})_W$, condition (2) in Theorem \ref{thm-aut-t} and Lemma \ref{lem-index} (ii), (iii) imply that we have $\Lambda_{e_1}=\Lambda_{e_1}\cap \Lambda_{e_2}=\Lambda_{e_2}$.
\end{proof}

\begin{proof}[Proof of Theorem \ref{thm-aut-t}]
For a.e.\ $z\in Z$, we define a map $\varphi_z\colon V(T_{\Gamma})\to V(T_{\Lambda})$ by $\varphi_z(u)=v_n$ if $z\in Z_n$, where we use the notation in Lemma \ref{lem-pre}.
By symmetry, a conclusion similar to Lemma \ref{lem-vertex-stab} holds for subgroupoids of $\cal{H}$.
For a.e.\ $z\in Z$, the map $\varphi_z$ is thus well-defined and is independent of the choice of a countable Borel partition of $Z$.
Applying the same process to $f^{-1}$ in place of $f$, for a.e.\ $z\in Z$, we obtain a map $\psi_z\colon V(T_{\Lambda})\to V(T_{\Gamma})$.
By Lemma \ref{lem-vertex-stab}, for a.e.\ $z\in Z$, the composition $\psi_z\circ \varphi_z$ is the identity on $V(T_{\Gamma})$, and $\varphi_z\circ \psi_z$ is the identity on $V(T_{\Lambda})$.
In particular, $\varphi_z$ and $\psi_z$ are bijections.

\begin{claim}
For a.e.\ $z\in Z$, $\varphi_z$ defines a simplicial isomorphism from $T_{\Gamma}$ onto $T_{\Lambda}$.
\end{claim}

\begin{proof}
This proof heavily depends on the proof of \cite[Theorem A]{ghmr}.
We extend the map $\varphi_z\colon V(T_{\Gamma})\to V(T_{\Lambda})$ to a continuous map $\tilde{\varphi}_z\colon T_{\Gamma}\to T_{\Lambda}$ in an affine way.
Namely, for each edge $e$ of $T_{\Gamma}$ whose origin and terminal are $u_1$ and $u_2$, respectively, the restriction of $\tilde{\varphi}_z$ to $e$ is the affine homeomorphism from $e$ to the geodesic from $\varphi_z(u_1)$ to $\varphi_z(u_2)$.
Similarly, we extend $\psi_z$ to a continuous map $\tilde{\psi}_z\colon T_{\Lambda}\to T_{\Gamma}$.

We show that for a.e.\ $z\in Z$, $\tilde{\varphi}_z$ is locally injective, that is, for any $u\in V(T_{\Gamma})$ and any two distinct edges $e_1$, $e_2$ of $T_{\Gamma}$ whose boundaries contain $u$, the first edges of the paths $\tilde{\varphi}_z(e_1)$, $\tilde{\varphi}_z(e_2)$ are distinct.
This implies injectivity of $\tilde{\varphi}_z$, and thus implies the claim.

We pick $u\in V(T_{\Gamma})$ and suppose that we have a Borel subset $W$ of $Z$ with positive measure and two edges $e$, $e'$ of $T_{\Gamma}$ whose boundaries contain $u$, such that for a.e.\ $z\in W$, the first edges of the two paths $\tilde{\psi}_z\circ \tilde{\varphi}_z(e)$, $\tilde{\psi}_z\circ \tilde{\varphi}_z(e')$ in $T_{\Gamma}$ starting at $u$ are equal.
Denote the first edge by $d$.
It suffices to show the equality $e=e'$.
For a.e.\ $z\in Z$, any $\gamma \in \Gamma$ and any $v\in V(T_{\Gamma})$, we have $\gamma \psi_z\circ \varphi_z(v)=\gamma v=\psi_z\circ \varphi_z(\gamma v)$.
The equality $\gamma \tilde{\psi}_z\circ \tilde{\varphi}_z(s)=\tilde{\psi}_z\circ \tilde{\varphi}_z(\gamma s)$ thus holds for any $\gamma \in \Gamma$ and any $s\in T_{\Gamma}$.
We have the inclusions $\Gamma_e< \Gamma_d$ and $\Gamma_{e'}<\Gamma_d$, and thus $\Gamma_e=\Gamma_d=\Gamma_{e'}$ thanks to the assumption that $q$ is not a multiple of $p$.
There exist $k, k'\in \mathbb{Z}$ with $d=a_u^ke=a_u^{k'}e'$, where $a_u$ is a generator of $\Gamma_u$.
The equality $a_u^{k-k'}\tilde{\psi}_z\circ \tilde{\varphi}_z(e)=\tilde{\psi}_z\circ \tilde{\varphi}_z(e')$ implies the equality $a_u^{k-k'}d=d$.
We have $a_u^{k-k'}\in \Gamma_d=\Gamma_e$, and thus $e'=a_u^{k-k'}e=e$.
\end{proof}

\begin{claim}
For a.e.\ $z\in Z$, the simplicial isomorphism $\varphi_z\colon T_{\Gamma}\to T_{\Lambda}$ is orientation-preserving.
\end{claim}

\begin{proof}
Pick $u\in V(T_{\Gamma})$.
Let $e_1,\ldots, e_{|q|}\in E(T_{\Gamma})$ denote the edges whose boundaries contain $u$ as their origins.
Let $u_1,\ldots, u_{|q|}\in V(T_{\Gamma})$ be the vertices of $e_1,\ldots, e_{|q|}$ other than $u$, respectively.
Let $Z_1$ be a Borel subset of $Z$ with positive measure where the maps assigning $\varphi_z(u)$ and $\varphi_z(u_i)$ to $z$ are constant for any $i$.
We denote by $v$ and $v_i$ the values of the maps on $Z_1$, respectively, and denote by $f_i$ the edge in $E(T_{\Lambda})$ connecting $v$ with $v_i$.
The equality $\Gamma_{e_1}=\cdots =\Gamma_{e_{|q|}}$ implies the equality $(\cal{G}_{e_1})_{Z_1}=\cdots =(\cal{G}_{e_{|q|}})_{Z_1}$.
It follows that there exists a Borel subset $Z_2$ of $Z_1$ with positive measure such that $(\cal{H}_{f_1})_{Z_2}=\cdots =(\cal{H}_{f_{|q|}})_{Z_2}$.
By Lemma \ref{lem-edge-stab}, we have $\Lambda_{f_1}=\cdots =\Lambda_{f_{|q|}}$.
The edges $f_1,\ldots, f_{|q|}$ therefore have the same origin.
\end{proof}

The rest of the proof of Theorem \ref{thm-aut-t} is essentially the same as that of \cite[Theorem 4.4]{kida-ama}.
We define a Borel map $\varphi \colon Z\to \isom(T_{\Gamma}, T_{\Lambda})$ by $\varphi(z)=\varphi_z$ for $z\in Z$.
Following the proof of \cite[Lemma 4.7]{kida-ama}, we can show the equality
\[\jmath \circ \alpha(\gamma, x)\varphi(x)=\varphi(\gamma \cdot x)\imath(\gamma)\]
for any $\gamma \in \Gamma$ and a.e.\ $x\in Z$ with $\gamma \cdot x\in Z$.
The coupling $\Sigma$ is identified with $X\times \Lambda$ as a measure space.
The group $\Gamma \times \Lambda$ acts on $X\times \Lambda$ by
\[(\gamma, \lambda)(x, \lambda')=(\gamma \cdot x, \alpha(\gamma, x)\lambda'\lambda^{-1})\]
for $\gamma \in \Gamma$, $\lambda, \lambda'\in \Lambda$ and $x\in X$.
We define a Borel map $\Phi \colon \Sigma \to \isom(T_{\Lambda}, T_{\Gamma})$ by putting $\Phi((\gamma, \lambda)(x, e_{\Lambda}))=\imath(\gamma)\varphi(x)^{-1}\jmath(\lambda)^{-1}$ for $\gamma \in \Gamma$, $\lambda \in \Lambda$ and $x\in Z$, where $e_{\Lambda}$ denotes the neutral element of $\Lambda$.
This map $\Phi$ is shown to be well-defined and be a desired one.
We refer to the proof of \cite[Theorem 4.4]{kida-ama} for a precise argument.
Uniqueness of $\Phi$ follows from Lemma \ref{lem-dirac}.
\end{proof}

Through the translation between ME and WOE, we obtain the following:

\begin{cor}\label{cor-woe}
In Notation \ref{not-bs}, let $\Gamma \c (X, \mu)$ and $\Lambda \c (Y, \nu)$ be ergodic f.f.m.p.\ actions such that
\begin{itemize}
\item the action of $\langle a^{d_0}\rangle$ on almost every ergodic component for the action $\langle a\rangle \c (X, \mu)$ is $(p_0, q_0)$-ergodic; and
\item the action of $\langle b^{c_0}\rangle$ on almost every ergodic component for the action $\langle b\rangle \c (Y, \nu)$ is $(r_0, s_0)$-ergodic.
\end{itemize}
If the actions $\Gamma \c (X, \mu)$ and $\Lambda \c (Y, \nu)$ are WOE, then $|p|=|r|$ and $|q|=|s|$.
\end{cor}


 \subsection{A profinite completion of integers}
 
This subsection is preliminary to the next subsection.
We fix two integers $p$, $q$ with $2\leq p<q$.
Define $d_0$ as the greatest common divisor of $p$ and $q$, and set $p_0=p/d_0$ and $q_0=q/d_0$.
We assume that $q$ is not a multiple of $p$.
We thus have $1<p_0<q_0$.

Let $E$ denote the infinite cyclic group $\mathbb{Z}$.
For two non-negative integers $k$, $l$, we denote by $E_{k, l}$ the subgroup of $E$ with $[E: E_{k, l}]=d_0p_0^kq_0^l$.
We define $E_{\infty}$ as the projective limit
\[E_{\infty}=\varprojlim E/E_{k, l},\]
which is a compact unital ring.
The group $E$ is naturally identified with a dense subgroup of $E_{\infty}$.
For two non-negative integers $k$, $l$, we denote by $\bar{E}_{k, l}$ the closure of $E_{k, l}$ in $E_{\infty}$, which is equal to the kernel of the canonical homomorphism from $E_{\infty}$ onto $E/E_{k, l}$.

\begin{lem}\label{lem-limit}
In the above notation, the following assertions hold:
\begin{enumerate}
\item The additive group $\bar{E}_{0, 0}$ is torsion-free.
\item For any two non-negative integers $k$, $l$, the map $\sigma_{k, l}\colon \bar{E}_{0, 0}\to \bar{E}_{k, l}$ defined by $\sigma_{k, l}(x)=p_0^kq_0^lx$ for $x\in \bar{E}_{0, 0}$ is an isomorphism of additive groups.
\item For any continuous automorphism $\alpha$ of the additive group $E_{\infty}$, there exists a unit $r$ of $E_{\infty}$ with $\alpha(x)=rx$ for any $x\in E_{\infty}$.
\item Let $r$ be a unit of $E_{\infty}$.
If there exist non-negative integers $k$, $l$ with $rx=x$ for any $x\in \bar{E}_{k, l}$, then $d_0(r-1)=0$.
\item Conversely, if $r$ is a unit of $E_{\infty}$ with $d_0(r-1)=0$, then we have $rx=x$ for any $x\in \bar{E}_{0, 0}$.
\end{enumerate}
\end{lem}

\begin{proof}
To prove assertion (i), we pick an element $x$ of $\bar{E}_{0, 0}$ and a positive integer $m$ with $mx=0$.
Choose a sequence $\{ n_i\}_i$ in $E_{0, 0}$ converging to $x$.
The sequence $\{ mn_i\}_i$ then converges to $0$.
Let $k$ be a positive integer.
For any sufficiently large $i$, the integer $mn_i$ is divided by $d_0(p_0q_0)^k$.
It follows that for any sufficiently large $i$, the integer $n_i$ is also divided by $d_0(p_0q_0)^k$.
The sequence $\{ n_i\}_i$ therefore converges to $0$, and we obtain $x=0$.
Assertion (i) is proved.

Assertion (ii) follows from assertion (i).
In assertion (iii), putting $r=\alpha(1)$, we have the equality $\alpha(n)=nr$ for any $n\in \mathbb{Z}$.
The equality $\alpha(x)=rx$ for any $x\in E_{\infty}$ follows from the continuity of $\alpha$.

Let $r$ be a unit of $E_{\infty}$ satisfying the assumption in assertion (iv).
We then have $d_0p_0^kq_0^l(r-1)=0$ because $d_0p_0^kq_0^l$ belongs to $\bar{E}_{k, l}$.
On the other hand, $d_0(r-1)$ belongs to $\bar{E}_{0, 0}$.
We have $d_0(r-1)=0$ by assertion (i).
Assertion (iv) is proved.

In assertion (v), pick an element $x$ of $\bar{E}_{0, 0}$, which is approximated by a sequence $\{ d_0n_i\}_i$ of elements of $E_{0, 0}$ with $n_i\in E$ for any $i$.
Let $y\in E_{\infty}$ be an accumulation point of the sequence $\{ n_i\}_i$.
We then have $d_0y=x$ and $(r-1)x=(r-1)d_0y=0$.
Assertion (v) is proved.
\end{proof}


\subsection{Further reduction}

Under an assumption stronger than that in Theorem \ref{thm-aut-t}, we obtain the following stronger conclusion.

\begin{thm}\label{thm-mer}
In Notation \ref{not-bs}, let $(\Sigma, m)$ be a $(\Gamma, \Lambda)$-coupling.
We assume the following two conditions:
\begin{enumerate}
\item[(I)] The action of $\langle a^{d_0}\rangle$ on $\Sigma/\Lambda$ is ergodic.
\item[(II)] The action of $\langle b^{c_0}\rangle$ on $\Sigma/\Gamma$ is ergodic.
\end{enumerate}
Then there exists an integer $\varepsilon \in \{ \pm 1\}$ with $(p, q)=(\varepsilon r, \varepsilon s)$, that is, $\Gamma$ and $\Lambda$ are isomorphic.
\end{thm}

Theorem \ref{thm-woe} follows from this theorem.
Before the proof of Theorem \ref{thm-mer}, let us make a few comments.
By Lemma \ref{lem-erg-comp}, conditions (I) and (II) imply conditions (1) and (2) in Theorem \ref{thm-aut-t}, respectively.
It follows that there exists an almost $(\Gamma \times \Lambda)$-equivariant Borel map $\Phi \colon \Sigma \to (\isom(T_{\Lambda}, T_{\Gamma}), \imath, \jmath)$ and that the equalities $|p|=|r|$ and $|q|=|s|$ hold.
In particular, the equality $c_0=d_0$ holds.

\begin{proof}[Proof of Theorem \ref{thm-mer}]
Because of the above argument, to prove the theorem, it is enough to deduce a contradiction under the assumption that $p>0$, $q>0$, $r=p$ and $s=-q$.
We then have $r_0=p_0$ and $s_0=-q_0$.
We fix the notation as follows.
Fix $u_0\in V(T_{\Gamma})$ and $v_0\in V(T_{\Lambda})$.
Let $E$ denote the stabilizer of $u_0$ in $\Gamma$, and let $F$ denote the stabilizer of $v_0$ in $\Lambda$.
For two non-negative integers $k$, $l$, we denote by $E_{k, l}$ the subgroup of $E$ with $[E: E_{k, l}]=d_0p_0^kq_0^l$.
Similarly, we denote by $F_{k, l}$ the subgroup of $F$ with $[F: F_{k, l}]=d_0p_0^kq_0^l$.
Define the projective limits
\[E_{\infty}=\varprojlim E_{k, l},\quad F_{\infty}=\varprojlim F_{k, l}.\]
We define $\mathfrak{I}$ as the space of continuous isomorphisms from $E_{\infty}$ onto $F_{\infty}$ as additive groups, and equip it with the standard Borel structure associated with the compact-open topology.

By Theorem \ref{thm-aut-t}, there exists an almost $(\Gamma \times \Lambda)$-equivariant Borel map $\Phi \colon \Sigma \to (\isom(T_{\Lambda}, T_{\Gamma}), \imath, \jmath)$.
Let $\Sigma_0$ be the inverse image under $\Phi$ of the Borel subset $\{\, \varphi \in \isom(T_{\Lambda}, T_{\Gamma})\mid \varphi(v_0)=u_0\,\}$.
The subset $\Sigma_0$ has positive measure because we have the equality $\Gamma \Sigma_0=\Sigma$ up to null sets.
As in \cite[Lemma 5.2]{kida-ama}, it is shown that $\Sigma_0$ is an $(E, F)$-coupling.
Let $X, Y\subset \Sigma_0$ be fundamental domains for the actions $F\c \Sigma_0$ and $E\c \Sigma_0$, respectively, such that $Z=X\cap Y$ has positive measure.
They are also fundamental domains for the actions $\Lambda \c \Sigma$ and $\Gamma \c \Sigma$, respectively.
We have the natural actions $\Gamma \c X$ and $\Lambda \c Y$, and set $\cal{G}=\Gamma \ltimes X$ and $\cal{H}=\Lambda \ltimes Y$.
Let $\alpha \colon \Gamma \times X\to \Lambda$ be the ME cocycle associated with $X$.
The Borel map $f\colon (\cal{G})_Z\to (\cal{H})_Z$ defined by $f(\gamma, x)=(\alpha(\gamma, x), x)$ for $(\gamma, x)\in (\cal{G})_Z$ is then an isomorphism.
We set $\cal{E}=E\ltimes X$ and $\cal{F}=F\ltimes Y$.
The equality $f((\cal{E})_Z)=(\cal{F})_Z$ holds by the construction of $X$ and $Y$.
For two non-negative integers $k$, $l$, we set
\[\cal{E}_{k, l}=E_{k, l}\ltimes X,\quad \cal{F}_{k, l}=F_{k, l}\ltimes Y.\]

\begin{claim}\label{claim-part}
For any two non-negative integers $k$, $l$ with $(k, l)\neq (0, 0)$, there exists a countable Borel partition $Z=\bigsqcup_nZ_n^{k, l}$ with the equality $f((\cal{E}_{k, l})_{Z_n^{k, l}})=(\cal{F}_{k, l})_{Z_n^{k, l}}$ for any $n$.
\end{claim}

\begin{proof}
Let $k$ and $l$ be non-negative integers with $(k, l)\neq (0, 0)$.
There exists a vertex $u\in V(T_{\Gamma})$ with $E\cap \Gamma_u=E_{k, l}$, where $\Gamma_u$ is the stabilizer of $u$ in $\Gamma$.
By Lemma \ref{lem-pre}, there exist a countable Borel partition $Z=\bigsqcup_n Z_n$ and $v_n\in V(T_{\Lambda})$ such that $f((\cal{G}_u)_{Z_n})=(\cal{H}_{v_n})_{Z_n}$ for any $n$, where we use the same notation as in the lemma.
It follows that $f$ sends $(\cal{E}_{k, l})_{Z_n}$ onto $(\cal{F}\cap \cal{H}_{v_n})_{Z_n}$ for any $n$.
Condition (I) in Theorem \ref{thm-mer} and Lemma \ref{lem-erg-comp} imply that $\cal{E}_{k, l}$ is ergodic.
By Lemma \ref{lem-index} (ii) and (iii), the index of $(\cal{E}_{k, l})_{Z_n}$ in $(\cal{E})_{Z_n}$ at a.e.\ $x\in Z_n$ is equal to $d_0p_0^kq_0^l$.
It follows that the index of $(\cal{F}\cap \cal{H}_{v_n})_{Z_n}$ in $(\cal{F})_{Z_n}$ at a.e.\ $x\in Z_n$ is also equal to the same number.
There exist non-negative integers $k'$, $l'$ with $(k', l')\neq (0, 0)$ and $\cal{F}\cap \cal{H}_{v_n}=\cal{F}_{k', l'}$.
Condition (II) in Theorem \ref{thm-mer} and Lemma \ref{lem-erg-comp} imply that $\cal{F}_{k', l'}$ is ergodic.
By Lemma \ref{lem-index} (ii) and (iii), we have $k'=k$ and $l'=l$.
The claim is proved.
\end{proof}

\noindent {\bf Construction of $\psi_0\colon Z\to \mathfrak{I}$.}
Let $k$ and $l$ be non-negative integers with $(k, l)\neq (0, 0)$.
We have a countable Borel partition $Z=\bigsqcup_n Z_n^{k, l}$ satisfying the conclusion in Claim \ref{claim-part}.
For each $n$, we have the isomorphism
\[E/E_{k, l}\to \cal{E}/\cal{E}_{k, l}\to (\cal{E})_{Z_n^{k, l}}/(\cal{E}_{k, l})_{Z_n^{k, l}}\stackrel{f}{\to}(\cal{F})_{Z_n^{k, l}}/(\cal{F}_{k, l})_{Z_n^{k, l}}\to \cal{F}/\cal{F}_{k, l}\to F/F_{k, l},\]
where we use Lemma \ref{lem-qu} in the isomorphisms other than the third one.
We define a Borel map $\psi_{k, l}\colon Z\times E/E_{k, l}\to F/F_{k, l}$ so that for any $x\in Z_n^{k, l}$, $\psi_{k, l}(x, \cdot)$ is equal to the above isomorphism.
This map is defined independently of the choice of the countable Borel partition of $Z$.

For any integers $k'$ and $l'$ with $k'\geq k$ and $l'\geq l$ and for a.e.\ $x\in Z$, we have the following commutative diagram, where the vertical arrows denote the canonical homomorphisms:
\[\xymatrix{
E/E_{k', l'} \ar[rr]^{\psi_{k', l'}(x, \cdot)} \ar[d] & & F/F_{k', l'} \ar[d]\\
E/E_{k, l} \ar[rr]^{\psi_{k, l}(x, \cdot)} & & F/F_{k, l}
}\]
We therefore obtain a Borel map $\psi_0\colon Z\to \mathfrak{I}$ such that for a.e.\ $x\in Z$ and for any non-negative integers $k$, $l$ with $(k, l)\neq (0, 0)$, the isomorphism $\psi_0(x)$ from $E_{\infty}$ onto $F_{\infty}$ induces the isomorphism $\psi_{k, l}(x, \cdot)$ from $E/E_{k, l}$ onto $F/F_{k, l}$.

\begin{claim}\label{claim-const}
The map $\psi_0\colon Z\to \mathfrak{I}$ is essentially constant.
\end{claim}

\begin{proof}
We fix $\gamma \in E$ and a Borel subset $A$ of $Z$ with positive measure such that the inclusion $\gamma \cdot A\subset Z$ holds, and $\alpha(\gamma, \cdot)$ is constant on $A$ with the value $\lambda \in F$.
Define an automorphism $U_{\gamma}$ of $\cal{G}$ by $U_{\gamma}(g)=(\gamma, r(g))g(\gamma, s(g))^{-1}$ for $g\in \cal{G}$.
Similarly, define an automorphism $V_{\lambda}$ of $\cal{H}$ by $V_{\lambda}(h)=(\lambda, r(h))h(\lambda, s(h))^{-1}$ for $h\in \cal{H}$.
The automorphism $U_{\gamma}$ preserves $\cal{E}$ and $\cal{E}_{k, l}$ for any non-negative integers $k$, $l$ with $(k, l)\neq (0, 0)$.
It follows that $U_{\gamma}$ induces an isomorphism
\[E/E_{k, l}\to (\cal{E})_A/(\cal{E}_{k, l})_A\to (\cal{E})_{\gamma \cdot A}/(\cal{E}_{k, l})_{\gamma \cdot A}\to E/E_{k, l},\]
which is the identity because $E$ is commutative.
Similarly, the automorphism $V_{\lambda}$ induces an isomorphism
\[F/F_{k, l}\to (\cal{F})_A/(\cal{F}_{k, l})_A\to (\cal{F})_{\lambda \cdot A}/(\cal{F}_{k, l})_{\lambda \cdot A}\to F/F_{k, l},\]
which is the identity.
Note that we have $\gamma \cdot A=\lambda \cdot A$ because $f$ induces the identity on $Z$.
The equality $f\circ U_{\gamma}=V_{\lambda}\circ f$ holds on $(\cal{G})_A$ by the definition of $f$.
We thus have $\psi_0(\gamma \cdot x)=\psi_0(x)$ for a.e.\ $x\in A$.
This equality holds for a.e.\ $x\in Z$ by the choice of $A$.
The claim follows because $\cal{E}$ is ergodic.
\end{proof}

\noindent {\bf Construction of $C$.}
The groups $\Gamma$ and $\Lambda$ do not act on $\mathfrak{I}$ naturally.
Let us now introduce a standard Borel space $C$ on which $\Gamma \times \Lambda$ acts and which contains a quotient of $\mathfrak{I}$.
We define $\mathfrak{C}_0$ as the set of continuous isomorphisms from a closed finite index subgroup of $E_{\infty}$ onto a closed finite index subgroup of $F_{\infty}$ as additive groups.
We say that two elements of $\mathfrak{C}_0$ are equivalent if there exists a finite index subgroup of $E_{\infty}$ on which they are equal.
Define $\mathfrak{C}$ as the set of equivalence classes in $\mathfrak{C}_0$.
For $\varphi \in \mathfrak{C}_0$, let $[\varphi]\in \mathfrak{C}$ denote the equivalence class of $\varphi$.

Lemma \ref{lem-limit} (ii) implies that we have the isomorphisms $\sigma_{1, 0}\colon \bar{E}_{0, 0}\to \bar{E}_{1, 0}$ and $\sigma_{0, 1}\colon \bar{E}_{0, 0}\to \bar{E}_{0, 1}$ defined as the multiplications by $p_0$ and $q_0$, respectively.
We define $\sigma \colon \bar{E}_{1, 0}\to \bar{E}_{0, 1}$ as the composition $\sigma_{0, 1}\circ (\sigma_{1, 0})^{-1}$.
In the presentation $\Gamma =\langle\, a, t\mid ta^pt^{-1}=a^q\,\rangle$, let $\Gamma$ act on $\mathfrak{C}$ from right as follows.
Let $a$ act on $\mathfrak{C}$ by the identity and $t$ act on $\mathfrak{C}$ by $[\varphi]t=[\varphi \circ \sigma]$ for $\varphi \in \mathfrak{C}_0$, where the composition $\varphi \circ \sigma$ is defined on a closed finite index subgroup of $\bar{E}_{1, 0}$.

Similarly, we have the isomorphisms $\tau_{1, 0}\colon \bar{F}_{0, 0}\to \bar{F}_{1, 0}$ and $\tau_{0, 1}\colon \bar{F}_{0, 0}\to \bar{F}_{0, 1}$ defined as the multiplications by $p_0$ and $q_0$, respectively.
We define $\tau \colon \bar{F}_{1, 0}\to \bar{F}_{0, 1}$ as the composition $\tau_{0, 1}\circ (\tau_{1, 0})^{-1}$.
We also have the automorphism $I$ of $F_{\infty}$ defined by $I(x)=-x$ for $x\in F_{\infty}$.
In the presentation $\Lambda =\langle\, b, u\mid ub^pu^{-1}=b^{-q}\,\rangle$, let $\Lambda$ act on $\mathfrak{C}$ from left as follows.
Let $b$ act on $\mathfrak{C}$ by the identity and $u$ act on $\mathfrak{C}$ by $u[\varphi]=[\tau \circ I \circ \varphi]$ for $\varphi \in \mathfrak{C}_0$.
This action of $\Lambda$ on $\mathfrak{C}$ commutes the action of $\Gamma$.

Let $U$ denote the group of units of $E_{\infty}$, which is a closed subset of $E_{\infty}$.
For $r\in U$, let $m_r$ denote the automorphism of the additive group $E_{\infty}$ defined by $m_r(x)=rx$ for $x\in E_{\infty}$.
The group $U$ acts on $\mathfrak{I}$ by the formula $r\varphi =\varphi \circ m_{r^{-1}}$ for $r\in U$ and $\varphi \in \mathfrak{I}$.
We define $U_0$ as the subgroup of $U$ consisting of all $r\in U$ with $d_0(r-1)=0$, which is closed in $U$.
The natural map from $\mathfrak{I}$ into $\mathfrak{C}$ induces an injective map from the quotient $\mathfrak{I}/U_0$ into $\mathfrak{C}$ by Lemma \ref{lem-limit} (iii)--(v).
We define $C_0$ as the image of this map and equip it with the standard Borel structure induced by the map.

For $n\in \mathbb{Z}$, we set $C_0\sigma^n=\{\, [\varphi \circ \sigma^n]\in \mathfrak{C}\mid [\varphi] \in C_0\,\}$.
Similarly, we define the subset $\tau^nC_0$ of $\mathfrak{C}$.
For any $\varphi \in \mathfrak{C}_0$ and any $n\in \mathbb{Z}$, the equality $[\varphi \circ \sigma^n]=[\tau^n\circ \varphi]$ holds.
It follows that for any $n\in \mathbb{Z}$, the equality $C_0\sigma^n=\tau^nC_0$ holds.
Since $[I\circ \varphi]$ belongs to $C_0$ for any $[\varphi]\in C_0$, we have
\[C_0\Gamma =\bigcup_{n\in \mathbb{Z}}C_0\sigma^n=\bigcup_{n\in \mathbb{Z}}\tau^nC_0=\Lambda C_0.\]
Let $C$ denote this subset of $\mathfrak{C}$, on which $\Gamma \times \Lambda$ acts.

We claim that the sets $C_0\sigma^n$ through $n\in \mathbb{Z}$ are mutually disjoint.
Let $\xi$ and $\eta$ denote the Haar measures on $E_{\infty}$ and $F_{\infty}$, respectively, with the total measure 1.
For any $[\varphi]\in C_0$, we have $\varphi_*\xi=\eta$ on a closed finite index subgroup of $F_{\infty}$.
On the other hand, we have $(\sigma^n)_*\xi=(q_0/p_0)^n\xi$ on a closed finite index subgroup of $E_{\infty}$.
The claim follows.
We equip $C$ with the standard Borel structure so that the actions of $\Gamma$ and $\Lambda$ on it are Borel.

\medskip

We define a Borel map $\psi \colon Z\to C_0$ as the composition of the essentially constant map $\psi_0\colon Z\to \mathfrak{I}$ with the natural map from $\mathfrak{I}$ into $\mathfrak{C}$.
Recall that we have the cocycle $\alpha \colon \Gamma \times X\to \Lambda$ associated with $X$.
We use a dot for the action of $\Gamma$ on $X$, that is, we set $\gamma \cdot x=(\gamma, \alpha(\gamma, x))x$ for $\gamma \in \Gamma$ and $x\in X$.

\begin{claim}\label{claim-psi-eq}
For any $\gamma \in \Gamma$ and a.e.\ $x\in Z$ with $\gamma \cdot x\in Z$, we have the equality $\alpha(\gamma, x)\psi(x)=\psi(\gamma \cdot x)\gamma$.
\end{claim}

\begin{proof}
We fix $\gamma \in \Gamma$ and a Borel subset $A$ of $Z$ with positive measure such that the inclusion $\gamma \cdot A\subset Z$ holds, and $\alpha(\gamma, \cdot)$ is constant on $A$ with the value $\lambda \in \Lambda$.
Define the automorphism $U_{\gamma}$ of $\cal{G}$ and the automorphism $V_{\lambda}$ of $\cal{H}$ by the same formulas as in the proof of Claim \ref{claim-const}.
The equality $f\circ U_{\gamma}=V_{\lambda}\circ f$ holds on $(\cal{G})_A$ by the definition of $f$.

Choose integers $n$, $K$ and $L$ such that $K>|n|$, $L>|n|$ and $\gamma E_{k, l}\gamma^{-1}=E_{k-n, l+n}$ for any $k, l\in \mathbb{Z}$ with $k\geq K$ and $l\geq L$.
For any such $k, l\in \mathbb{Z}$, the automorphism $U_{\gamma}$ induces an isomorphism
\begin{align*}
E_{K, L}/E_{k, l}\to &(\cal{E}_{K, L})_A/(\cal{E}_{k, l})_A\\
&\to (\cal{E}_{K-n, L+n})_{\gamma \cdot A}/(\cal{E}_{k-n, l+n})_{\gamma \cdot A}\to E_{K-n, L+n}/E_{k-n, l+n}.
\end{align*}
It moreover induces an isomorphism from $\bar{E}_{K, L}$ onto $\bar{E}_{K-n, L+n}$, which is equal to the restriction of $\sigma^n$.
For a.e.\ $x\in A$, the isomorphism $f\circ U_{\gamma}$ from $(\cal{G})_A$ onto $(\cal{H})_{\lambda \cdot A}$ therefore induces the element $[\psi_0(\gamma \cdot x)\circ \sigma^n]$ of $C$, which is equal to $\psi(\gamma \cdot x)\gamma$.

Similarly, the automorphism $V_{\lambda}$ induces an element of $\mathfrak{C}_0$, which is a restriction of $(\tau\circ I)^m$ for some $m\in \mathbb{Z}$.
For a.e.\ $x\in A$, the isomorphism $V_{\lambda}\circ f$ induces the element $[(\tau \circ I)^m\circ \psi_0(x)]$ of $C$, which is equal to $\lambda \psi_0(x)$.
\end{proof}

We define a Borel map $\Psi \colon \Sigma \to C$ by $\Psi((\gamma, \lambda)x)=\lambda \psi(x)\gamma^{-1}$ for $\gamma \in \Gamma$, $\lambda \in \Lambda$ and $x\in Z$.
As noted in the end of the proof of Theorem \ref{thm-aut-t}, the equality in Claim \ref{claim-psi-eq} is used to show that $\Psi$ is well-defined and is almost $(\Gamma \times \Lambda)$-equivariant.
By condition (I) in Theorem \ref{thm-mer}, the action $E\c \Sigma /\Lambda$ is ergodic.
Since the action of $E$ on $\Lambda \backslash C$ is trivial, the map from $\Sigma /\Lambda$ into $\Lambda \backslash C$ induced by $\Psi$ is essentially constant.
It follows that the image of $\Psi$ is essentially contained in $\Lambda [\alpha_0]$, where $\alpha_0$ is an element of $\mathfrak{C}_0$ such that $[\alpha_0]$ is the essential value of the map $\psi \colon Z\to C_0$.
A similar argument shows that the image of $\Psi$ is essentially contained in $[\alpha_0]\Gamma$.
We thus have the equality $\Lambda [\alpha_0]=[\alpha_0]\Gamma$.
It implies that there exists $n\in \mathbb{Z}$ with $[\tau \circ I\circ \alpha_0]=[\alpha_0\circ \sigma^n]$.
Since we have $[\alpha_0\circ \sigma^n]=[\tau^n\circ \alpha_0]$, there exists a closed finite index subgroup of $F_{\infty}$ on which $\tau^{1-n}\circ I$ is the identity.
One can find a non-zero element $x$ of $\bar{F}_{0, 0}$ such that $\tau^{1-n}\circ I(x)$ is defined and equal to $x$.
If $1-n\geq 0$, then there exists an element $y$ of $\bar{F}_{0, 0}$ with $x=p_0^{1-n}y=-q_0^{1-n}y$.
We thus have $(p_0^{1-n}+q_0^{1-n})y=0$ and obtain a contradiction by Lemma \ref{lem-limit} (i).
If $1-n<0$, then there exists an element $z$ of $\bar{F}_{0, 0}$ with $x=q_0^{n-1}z=-p_0^{n-1}z$.
We also obtain a contradiction.
\end{proof}

\begin{rem}
In Theorem \ref{thm-mer}, the conclusion is not necessarily true if we assume conditions (1) and (2) in Theorem \ref{thm-aut-t} in place of conditions (I) and (II).
A simple counterexample comes from commensurability between $\bs(p, q)$ and $\bs(p, -q)$ with $2\leq |p|\leq |q|$.
If $p$ and $q$ are integers with $2\leq |p|\leq |q|$, then for each $\varepsilon \in \{ \pm 1\}$, the subgroup $\Gamma_{\varepsilon}$ of $\bs(p, \varepsilon q)=\langle\, a, t\mid ta^pt^{-1}=a^{\varepsilon q}\, \rangle$ generated by $a$, $tat^{-1}$ and $t^2$ is of index 2, and $\Gamma_1$ and $\Gamma_{-1}$ are isomorphic.
Let $F\colon \Gamma_{-1}\to \Gamma_1$ be this isomorphism.

Let $\Gamma_1\times \Gamma_{-1}$ act on $\Gamma_1$ by $(\gamma, \delta)g=\gamma gF(\delta)^{-1}$ for $\gamma, g\in \Gamma_1$ and $\delta \in \Gamma_{-1}$.
Put $\Gamma =\bs(p, q)$ and $\Lambda =\bs(p, -q)$.
We define $\Sigma$ as the $(\Gamma, \Lambda)$-coupling defined by the action of $\Gamma \times \Lambda$ induced from the action of $\Gamma_1\times \Gamma_{-1}$ on $\Gamma_1$.
This coupling fulfills conditions (1) and (2) in Theorem \ref{thm-aut-t}, while it does not satisfy the conclusion of Theorem \ref{thm-mer} because $\Gamma$ and $\Lambda$ are not isomorphic.
\end{rem}


\subsection{A coupling with an unknown group}

The argument of this subsection heavily relies on Monod-Shalom \cite{ms}.
We say that a measure-preserving action of a discrete group $G$ on a measure space $(X, \mu)$ is {\it mildly mixing} if for any measurable subset $A$ of $X$ and any sequence $\{ g_i\}_i$ of $G$ tending to infinity with $\{ \mu(g_iA\triangle A)\}_i$ tending to $0$, we have either $\mu(A)=0$ or $\mu(X\setminus A)=0$.

\begin{thm}\label{thm-mm}
We set $\Gamma =\bs(p, q)$ with $2\leq |p|<|q|$ and denote by $d_0>0$ the greatest common divisor of $p$ and $q$.
Assume that $q$ is not a multiple of $p$.
Let $\Delta$ be a discrete group.
Let $(\Sigma, m)$ be a $(\Gamma, \Delta)$-coupling such that the action $\langle a^{d_0}\rangle \c \Sigma /\Delta$ is ergodic; and the action $\Delta \c \Sigma /\Gamma$ is mildly mixing.
We define $T$ as the Bass-Serre tree associated with $\Gamma$.

Then there exists an orientation-preserving simplicial action of $\Delta$ on $T$ such that for any $v\in V(T)$, the stabilizer of $v$ in $\Delta$ is infinite and amenable; and the actions of $\Delta$ on $V(T)$ and on $E(T)$ are both transitive.
\end{thm}

The conclusion of this theorem implies that $\Delta$ is an HNN extension of an infinite amenable group $\Delta_0$ relative to an isomorphism between a subgroup of $\Delta_0$ of index $|p|$ and a subgroup of $\Delta_0$ of index $|q|$.

\begin{proof}[Proof of Theorem \ref{thm-mm}]
We define $\Omega$ as the quotient space of $\Sigma \times \Sigma$ by the diagonal action of $\Delta$, which is a $(\Gamma, \Gamma)$-coupling.
By \cite[Lemma 6.5]{ms}, both the actions $\langle a^{d_0}\rangle \c \Omega /(\{ e\} \times \Gamma)$ and $\langle a^{d_0}\rangle \c \Omega /(\Gamma \times \{ e\})$ are ergodic.
In the application of the cited lemma, we use the assumption that the action $\Delta \c \Sigma/\Gamma$ is mildly mixing.
Thanks to Lemma \ref{lem-erg-comp}, Theorem \ref{thm-aut-t} can be applied to $\Omega$.
Combining Theorem \ref{thm-furman} and Lemma \ref{lem-dirac}, we obtain a homomorphism $\rho \colon \Delta \to \aut(T)$ and an almost $(\Gamma \times \Delta)$-equivariant Borel map $\Phi \colon \Sigma \to (\aut(T), \imath, \rho)$.
For each simplex $s$ of $T$, let $\stab(s)$ denote the stabilizer of $s$ in $\aut(T)$, and put
\[\Sigma_s=\Phi^{-1}(\stab(s)),\quad \Gamma_s=\imath^{-1}(\imath(\Gamma)\cap \stab(s)),\quad \Delta_s=\rho^{-1}(\rho(\Delta)\cap \stab(s)).\]
As in \cite[Lemma 5.2]{kida-ama}, we can show that $\Sigma_s$ is a $(\Gamma_s, \Delta_s)$-coupling.
As in \cite[Lemma 5.3]{kida-ama}, using ergodicity of the action $\Gamma_s \c \Sigma/\Delta$ for each simplex $s$ of $T$, we can show that the actions of $\Delta$ on $V(T)$ and on $E(T)$ are both transitive.
\end{proof}


\section{Non-trivial orbit equivalence}\label{sec-oe}

This section is devoted to the proof of the following: 

\begin{thm}\label{thm-woe-non-conj}
Let $p$ and $q$ be integers with $2\leq |p|\leq |q|$, and set $\Gamma =\bs(p, q)$.
Then there exist two f.f.m.p.\ actions $\Gamma \c (X, \mu)$ and $\Gamma \c (Y, \nu)$ such that they are WOE, but are not conjugate; and the actions of any non-trivial elliptic subgroup of $\Gamma$ on $(X, \mu)$ and on $(Y, \nu)$ are both ergodic.
\end{thm}

The two actions of $\Gamma$ in this theorem are not virtually conjugate either in the sense of \cite[Definition 1.3]{kida-oer} because $\Gamma$ is torsion-free by \cite[Theorem IV.2.4]{ls} and because the actions of any finite index subgroup of $\Gamma$ on $(X, \mu)$ and on $(Y, \nu)$ are ergodic.
It follows that under the assumption in Theorem \ref{thm-mer}, we cannot conclude that the two actions $\Gamma \c \Sigma /\Lambda$ and $\Lambda \c \Sigma /\Gamma$ in it are virtually conjugate.

Until Theorem \ref{thm-sigma}, we set $\Gamma =\bs(p, q)$ with $2\leq |p|\leq |q|$.
Let $d_0>0$ denote the greatest common divisor of $p$ and $q$, and set $p_0=p/d_0$ and $q_0=q/d_0$.
We fix a presentation $\Gamma =\langle \, a, t\mid ta^pt^{-1}=a^q\,\rangle$.
Let $T$ be the Bass-Serre tree associated with $\Gamma$.
Let $\imath \colon \Gamma \to \aut(T)$ be the homomorphism associated with the action of $\Gamma$ on $T$.


\subsection{A standard coupling}

We define a continuous homomorphism $\tau \colon \aut(T)\to \mathbb{Z}$ as follows.
Let $E_+$ and $E_-$ denote the sets of oriented edges of $T$ introduced in Section \ref{sec-bs}.
We set $\sigma(e)=1$ for each $e\in E_+$, and set $\sigma(f)=-1$ for each $f\in E_-$.
Fix a vertex $v_0$ of $T$.
For $\varphi \in \aut(T)$, let $e_1,\ldots, e_n$ be the shortest sequence of oriented edges of $T$ such that the origin of $e_1$ is $v_0$; the terminal of $e_n$ is $\varphi(v_0)$; and for any $i=1,\ldots, n-1$, the terminal of $e_i$ and the origin of $e_{i+1}$ are equal.
We then define $\tau(\varphi)$ as the integer $\sum_{i=1}^n\sigma(e_i)$.
The map $\tau$ is independent of the choice of $v_0$ and is indeed a continuous homomorphism.
Note that $\tau(\imath(a))=0$ and $\tau(\imath(t))=1$.

We set $L=\R$.
Let $\aut(T)$ act on $L$ by $\varphi x=(q/p)^{\tau(\varphi)}x$ for $\varphi \in \aut(T)$ and $x\in L$.
We then have the semi-direct product $L\rtimes \aut(T)$.
Let $K$ denote the closure of the cyclic group $\langle \imath(a)\rangle$ in $\aut(T)$, which is compact and abelian.
We define $G$ as the subgroup of $L\rtimes \aut(T)$ generated by $L$, $K$ and $\imath(t)$.
The subgroup $G$ is closed in $L\rtimes \aut(T)$.

For each $\theta \in L\setminus \{ 0\}$, we define a homomorphism $\pi_{\theta}\colon \Gamma \to G$ by $\pi_{\theta}(a)=(\theta, \imath(a))$ and $\pi_{\theta}(t)=(0, \imath(t))$.
This is injective, and the image $\pi_{\theta}(\Gamma)$ is a cocompact lattice in $G$ such that $[0, |\theta|)\times K$ is its fundamental domain. 
Let $\alpha_{\theta}$ be the automorphism of $L\rtimes \aut(T)$ that is the identity on $\aut(T)$ and is the multiplication by $\theta$ on $L$.
The equality $\pi_{\theta}=\alpha_{\theta}\circ \pi_1$ on $\Gamma$ then holds.

It follows that $(G, \pi_{\theta}, \pi_1)$ is a $(\Gamma, \Gamma)$-coupling.
In the rest of this subsection, we examine the action of $\Gamma$ on $G/\pi_1(\Gamma)$ through $\pi_{\theta}$ and that on $\pi_{\theta}(\Gamma)\backslash G$ through $\pi_1$.
The argument in the previous paragraph implies that the latter action is isomorphic to the action of $\Gamma$ on $G/\pi_1(\Gamma)$ through $\pi_{\theta^{-1}}$.
Since $L$ and $K$ commute in $G$, we identify the subgroup of $G$ generated by them with the direct product $L\times K$.
We set $E=\langle a\rangle$.
The space $G/\pi_1(\Gamma)$ is then identified with $(L\times K)/\pi_1(E)$.
We denote the latter space by $M$, which is a compact abelian group.

\begin{lem}\label{lem-dense}
Let $\theta$ be an irrational number.
Then the action of any non-trivial subgroup of $E$ on $M$ through $\pi_{\theta}$ is ergodic.
\end{lem}

\begin{proof}
Let $n$ be a non-zero integer.
Let $\Lambda$ denote the subgroup of $L\times K$ generated by $\pi_{\theta}(\langle a^n\rangle )$ and $\pi_1(E)$.
It suffices to show that $\Lambda$ is dense in $L\times K$.
The group $\Lambda$ is generated by $(n(\theta -1), e)$ and $(1, \imath(a))$.
Density of $\Lambda$ in $L\times K$ is thus equivalent to density of the subgroup generated by $([1], \imath(a))$ in the group $\mathbb{R}/n(\theta -1)\mathbb{Z}\times K$, where for $r\in \mathbb{R}$, $[r]$ denotes the equivalence class of $r$ in $\mathbb{R}/n(\theta -1)\mathbb{Z}$.
The latter condition follows from the assumption that $\theta$ is irrational.
\end{proof}

\begin{prop}\label{prop-conj}
For each $i=1, 2$, let $B_i$ be a compact, second countable and abelian group with the normalized Haar measure $m_i$, and let $\Lambda_i$ be a dense countable subgroup of $B_i$.
Let $\Lambda_i$ act on $(B_i, m_i)$ by left multiplication.
We suppose that the two actions $\Lambda_1\c (B_1, m_1)$ and $\Lambda_2\c (B_2, m_2)$ are conjugate, that is, we have an isomorphism $F\colon \Lambda_1\to \Lambda_2$ and a Borel isomorphism $f$ from a conull Borel subset of $B_1$ onto a conull Borel subset of $B_2$ with $f_*m_1=m_2$ and $f(\lambda b)=F(\lambda)f(b)$ for any $\lambda \in \Lambda_1$ and a.e.\ $b\in B_1$.

Then there exist a continuous isomorphism $\bar{F}\colon B_1\to B_2$ and an element $b_0\in B_2$ such that $\bar{F}$ extends $F$ and the equality $f(b)=\bar{F}(b)b_0$ holds for a.e.\ $b\in B_1$.
\end{prop}

\begin{proof}
We define a Borel map $h\colon B_1\times B_1\to B_2$ by $h(b, c)=f(b)f(c)f(bc)^{-1}$ for $b, c\in B_1$.
Since $B_1$ and $B_2$ are abelian, we have $h(\lambda b, c)=h(b, \lambda c)=h(b, c)$ for any $\lambda \in \Lambda_1$ and a.e.\ $(b, c)\in B_1\times B_1$.
Ergodicity of the action of $\Lambda_1\times \Lambda_1$ on $B_1\times B_1$ by left multiplication implies that $h$ is essentially constant.
Let $b_0\in B_2$ denote the essential image of $h$.

We define a Borel map $\bar{f}\colon B_1\to B_2$ by $\bar{f}(b)=f(b)b_0^{-1}$ for $b\in B_1$.
The map $\bar{f}$ then preserves products a.e.
It follows from \cite[Theorems B.2 and B.3]{zim-book} that there exists a continuous homomorphism $\bar{F}\colon B_1\to B_2$ with $\bar{F}=\bar{f}$ a.e.
For any $\lambda \in \Lambda_1$ and a.e.\ $b\in B_1$, we have
\[F(\lambda)f(b)=f(\lambda b)=\bar{F}(\lambda b)b_0=\bar{F}(\lambda)\bar{F}(b)b_0=\bar{F}(\lambda)f(b)\] 
and thus $F(\lambda)=\bar{F}(\lambda)$.
The map $\bar{F}$ extends $F$ and is surjective.
An argument for $f^{-1}$ in place of $f$ implies that $\bar{F}$ is an isomorphism.
\end{proof}

Let $\eta \colon L\times K\to M$ denote the canonical projection.

\begin{lem}\label{lem-path}
For any $k\in K$, $\eta(L\times \{ k\})$ is a path-connected component of $M$.
\end{lem}

\begin{proof}
Let $x$ be an element of $L$.
We define $I$ as the closed interval $[x-1/3, x+1/3]$ in $L$.
The restriction of $\eta$ to $I\times K$ is injective and is thus a homeomorphism onto the image.
For any $k\in K$, the image $\eta(I\times K)$ contains an open neighborhood of $\eta(x, k)$ in $M$ because $\eta$ is an open map.
Since $K$ is totally disconnected, for any $k\in K$ and any continuous path $l\colon [0, 1]\to M$ with $l(t_0)=\eta(x, k)$ for some $t_0\in [0, 1]$, there exists a neighborhood of $t_0$ in $[0, 1]$ such that $l$ sends any element of it to $\eta(I\times \{ k\})$.
It follows that for any continuous map from $[0, 1]$ into $M$, there exists $k\in K$ such that the image of the map is contained in $\eta(L\times \{ k\})$.
\end{proof}

\begin{lem}\label{lem-rat}
Let $F$ be a continuous automorphism of $M$.
Let $\alpha$ be the continuous automorphism of $L$ with $F(\eta(x, e))=\eta(\alpha(x), e)$ for any $x\in L$.
Then the element of $L$, $\alpha(1)$, is a rational number.
\end{lem}

By Lemma \ref{lem-path}, for any continuous automorphism $F$ of $M$, there exists a unique continuous automorphism $\alpha$ of $L$ satisfying the equality $F(\eta(x, e))=\eta(\alpha(x), e)$ for any $x\in L$.

\begin{proof}[Proof of Lemma \ref{lem-rat}]
For any positive integer $n$, we have the equality $\eta(d_0p_0^nq_0^n, e)=\eta(0, \imath(a)^{-d_0p_0^nq_0^n})$.
The right hand side approaches $0$ as $n$ goes to infinity because $\imath(a)^{-d_0p_0^nq_0^n}$ approaches the identity in $\aut(T)$.
By continuity of $F$, the element $F(\eta(d_0p_0^nq_0^n, e))=\eta(d_0p_0^nq_0^n\alpha(1), e)$ also approaches $0$.
Let $\eta_1\colon \mathbb{R}\to \mathbb{R}/\mathbb{Z}$ denote the canonical projection.
It follows that $\eta_1(d_0p_0^nq_0^n\alpha(1))$ approaches $0$ as $n$ goes to infinity.
If $|p_0|=|q_0|=1$, then $d_0\alpha(1)$ is an integer, and thus $\alpha(1)$ is rational.
Otherwise, considering the $|p_0q_0|$-adic expansion of $d_0\alpha(1)$, we find integers $r$, $n$ with $n\geq 0$ and $d_0\alpha(1)=r/|p_0q_0|^n$.
The number $\alpha(1)$ is thus rational.
\end{proof}

\begin{lem}\label{lem-conj-e}
Let $\theta$ and $\theta'$ be irrational numbers.
If the action of $E$ on $M$ through $\pi_{\theta}$ and that through $\pi_{\theta'}$ are conjugate, then $(\theta'-1)/(\theta-1)$ is a rational number. 
\end{lem}

\begin{proof}
The restriction of $\pi_{\theta}$ to $E$ is injective, and $\pi_{\theta}(E)$ is dense in $M$ by Lemma \ref{lem-dense}.
The same property holds for $\theta'$.
Proposition \ref{prop-conj} implies that there exists a continuous automorphism $F$ of $M$ such that either $F\circ \pi_{\theta}=\pi_{\theta'}$ on $E$ or $F\circ \pi_{\theta}=\pi_{\theta'}\circ I$ on $E$, where $I$ is the automorphism of $E$ defined by $I(x)=x^{-1}$ for $x\in E$.
Let $\alpha$ be the continuous automorphism of $L$ with the equality $F(\eta(x, e))=\eta(\alpha(x), e)$ for any $x\in L$.

If the equality $F\circ \pi_{\theta}=\pi_{\theta'}$ on $E$ holds, then we have $F(\eta(\theta, \imath(a)))=\eta(\theta', \imath(a))$ and thus $F(\eta(\theta-1, e))=\eta(\theta'-1, e)$.
It follows that $\alpha(\theta-1)=\theta'-1$.
Lemma \ref{lem-rat} implies that $(\theta'-1)/(\theta-1)$ is rational.
A similar argument can be applied to the other case.
\end{proof}

\begin{thm}\label{thm-not-conj}
Let $\theta$ be an irrational number.
Then the following assertions hold:
\begin{enumerate}
\item The action of $E$ on $(L\times K)/\pi_1(E)$ through $\pi_{\theta}$ and the action of $E$ on $\pi_{\theta}(E)\backslash (L\times K)$ through $\pi_1$ are not conjugate.
\item The action of $\Gamma$ on $G/\pi_1(\Gamma)$ through $\pi_{\theta}$ and the action of $\Gamma$ on $\pi_{\theta}(\Gamma)\backslash G$ through $\pi_1$ are not conjugate.
\end{enumerate}
\end{thm}

\begin{proof}
As noted right before Lemma \ref{lem-dense}, the action of $E$ on $\pi_{\theta}(E)\backslash (L\times K)$ through $\pi_1$ is conjugate with the action of $E$ on $(L\times K)/\pi_1(E)$ through $\pi_{\theta^{-1}}$.
Since we have the equality $(\theta^{-1}-1)/(\theta-1)=-\theta^{-1}$, assertion (i) follows from Lemma \ref{lem-conj-e}.

Any automorphism of $\Gamma$ preserves elliptic subgroups of $\Gamma$ because ellipticity of elements of $\Gamma$ is algebraic as noticed in the third paragraph in Section \ref{sec-bs}.
Assertion (ii) therefore follows from assertion (i).
\end{proof}

Finally, we notice that the action of $\Gamma$ on $G/\pi_1(\Gamma)$ through $\pi_{\theta}$ is not essentially free.
Let $\m \colon \Gamma \to \Rm$ denote the modular homomorphism.
For any $\theta \in L\setminus \{ 0\}$ and $\gamma \in \Gamma$, let $l_{\theta}(\gamma)\in L$ denote the first coordinate of $\pi_{\theta}(\gamma)$ in $L\rtimes \aut(T)$.
Since $\imath(\ker \m)$ acts on $L$ trivially, the map $l_{\theta}\colon \ker \m \to L$ is a homomorphism.
Let $N$ denote the kernel of this homomorphism.
The group $N$ is independent of $\theta$ because we have the equality $l_{\theta}(\gamma)=\theta l_1(\gamma)$ for any $\gamma \in \Gamma$.
Define the derived subgroups
\[\Gamma^{(1)}=[\Gamma, \Gamma],\quad \Gamma^{(2)}=[\Gamma^{(1)}, \Gamma^{(1)}].\]
We then have the inclusions $\Gamma^{(1)}<\ker \m$ and $\Gamma^{(2)}<N$.

\begin{lem}\label{lem-trivial}
For any $\theta \in L\setminus \{ 0\}$, the action of $N$ on $G/\pi_1(\Gamma)$ through $\pi_{\theta}$ is trivial.
\end{lem}

\begin{proof}
Pick $\gamma \in N$.
Choose a positive integer $n$ such that $\gamma$ commutes $a^n$.
Let $K_1$ denote the closure of $\langle \imath(a^n)\rangle$ in $K$, which is a subgroup of $K$ of index at most $n$.

The subset $[0, 1)\times K$ of $G$ is a fundamental domain for the action of $\Gamma$ on $G$ by right multiplication through $\pi_1$.
For any $x\in [0, 1)$ and $k\in K$, choosing $m\in \mathbb{Z}$ and $k_1\in K_1$ with $k=\imath(a^m)k_1$, we have the equality
\begin{align*}
\pi_{\theta}(\gamma)(x, k)&=(x, \imath(\gamma)k)=(x, k_1\imath(\gamma)\imath(a^m))=(x, k_1\imath(a^m)\imath(a^{-m}\gamma a^m))\\
&=(x, k\imath(a^{-m}\gamma a^m))=(x, k)\pi_1(a^{-m}\gamma a^m),
\end{align*}
where we use commutativity of $K$, and the first and last equalities hold because $\gamma$ and $a^{-m}\gamma a^m$ lie in $N$.
The lemma therefore follows.
\end{proof}


\subsection{A modified coupling}

Modifying the coupling $(G, \pi_{\theta}, \pi_1)$, we obtain a WOE between f.f.m.p.\ actions of $\Gamma$ satisfying the conclusion in Theorem \ref{thm-woe-non-conj}.
We pick an essentially free and measure-preserving action of $\Gamma$ on a standard probability space $(Z, \xi)$.
Let $(\Sigma, m)$ denote the product space $(Z\times G, \xi \times m_G)$, where $m_G$ is the Haar measure on $G$.
We define an action of $\Gamma \times \Gamma$ on $\Sigma$ by
\[(\gamma_1, \gamma_2)(z, g)=(\gamma_1z, \pi_{\theta}(\gamma_1)g\pi_1(\gamma_2)^{-1})\]
for $\gamma_1, \gamma_2\in \Gamma$, $z\in Z$ and $g\in G$.
This action makes $(\Sigma, m)$ into a $(\Gamma, \Gamma)$-coupling.
Let $\mathsf{L}, \mathsf{R}\colon \Gamma \to \Gamma \times \Gamma$ be the homomorphisms defined by
\[\mathsf{L}(\gamma)=(\gamma, e),\quad \mathsf{R}(\gamma)=(e, \gamma)\]
for $\gamma \in \Gamma$.
We define two subsets $X_0$, $Y_0$ of $G$ and two subsets $X$, $Y$ of $\Sigma$ by
\[X_0=[0, 1)\times K,\quad Y_0=[0, |\theta|)\times K,\quad X=Z\times X_0,\quad Y=Z\times Y_0.\] 
The subsets $X_0$ and $Y_0$ are fundamental domains for the actions $\mathsf{R}(\Gamma)\c G$ and $\mathsf{L}(\Gamma)\c G$, respectively.
The subsets $X$ and $Y$ are fundamental domains for the actions $\mathsf{R}(\Gamma)\c \Sigma$ and $\mathsf{L}(\Gamma)\c \Sigma$, respectively.

We have the natural actions $\Gamma \c X$ and $\Gamma \c Y$.
The former action is isomorphic to the diagonal action of $\Gamma$ on $Z\times X_0$.
It follows that the actions $\Gamma \c X$ and $\Gamma \c Y$ are essentially free because so is the action $\Gamma \c Z$.
Let $\beta \colon \Gamma \times Y_0\to \Gamma$ be the ME cocycle for the coupling $(G, \pi_{\theta}, \pi_1)$.
As for the action $\Gamma \c Y$, we have the equality
\[\gamma \cdot (z, y)=(\beta(\gamma, y)z, \gamma \cdot y)\]
for any $\gamma \in \Gamma$, any $z\in Z$ and a.e.\ $y\in Y_0$.
We now prove that the actions $\Gamma \c X$ and $\Gamma \c Y$ satisfy the desired property in Theorem \ref{thm-woe-non-conj}.

\begin{lem}\label{lem-z-conj}
Suppose that $\theta$ is irrational and that any non-neutral element of $\Gamma^{(2)}$ acts on $(Z, \xi)$ ergodically.
Then the following assertions hold:
\begin{enumerate}
\item For any non-neutral element $\gamma_0$ of $\Gamma^{(2)}$, the ergodic decomposition for the action $\langle \gamma_0 \rangle \c X$ is equal to the projection from $X=Z\times X_0$ onto $X_0$.
Moreover, the ergodic decomposition for the action $\langle \gamma_0 \rangle \c Y$ is equal to the projection from $Y=Z\times Y_0$ onto $Y_0$. 
\item The actions $\Gamma \c X$ and $\Gamma \c Y$ are not conjugate.
\end{enumerate}
\end{lem}

\begin{proof}
The former assertion in assertion (i) holds because $\Gamma$ acts on $X=Z\times X_0$ diagonally and because $\Gamma^{(2)}$ acts on $X_0$ trivially by Lemma \ref{lem-trivial}.
We prove the latter assertion.
Pick a non-neutral element $\gamma_0$ of $\Gamma^{(2)}$.
As in the second paragraph of the proof of Lemma \ref{lem-trivial}, for any $x\in [0, |\theta|)$ and $k\in K$, we can find $m\in \mathbb{Z}$ such that $\beta(\gamma_0^n, (x, k))=a^m\gamma_0^na^{-m}$ for any $n\in \mathbb{Z}$.
It follows that the restriction of the action $\langle \gamma_0 \rangle \c Y$ to $Z\times \{ (x, k)\}$ is isomorphic to the action $\langle \gamma_0 \rangle \c Z$, which is ergodic by assumption.
The latter assertion in assertion (i) is proved.

We prove assertion (ii).
Assuming that the two actions $\Gamma \c X$ and $\Gamma \c Y$ are conjugate, we deduce a contradiction.
Let $F$ be an automorphism of $\Gamma$ and $f$ a Borel isomorphism from a conull Borel subset of $X$ onto a conull Borel subset of $Y$ such that $f$ preserves the classes of the measures on $X$ and $Y$ and satisfies the equality $f(\gamma \cdot x)=F(\gamma)\cdot f(x)$ for any $\gamma \in \Gamma$ and a.e.\ $x\in X$.
Pick a non-neutral element $\gamma_0$ of $\Gamma^{(2)}$.
Since $F$ preserves $\Gamma^{(2)}$, the element $F(\gamma_0)$ is also non-neutral and belongs to $\Gamma^{(2)}$.
Assertion (i) implies that we have a Borel isomorphism $h$ from a conull Borel subset of $X_0$ onto a conull Borel subset of $Y_0$ which preserves the classes of the measures on $X_0$ and $Y_0$ and satisfies the equality $f(Z\times \{ x\})=Z\times \{ h(x)\}$ up to null sets, for a.e.\ $x\in X_0$.
For any $\gamma \in \Gamma$ and a.e.\ $x\in X_0$, we then have the equality $h(\gamma \cdot x)=F(\gamma)\cdot h(x)$.
It follows that the actions $\Gamma \c X_0$ and $\Gamma \c Y_0$ are conjugate.
This contradicts Theorem \ref{thm-not-conj}.
\end{proof}

\begin{lem}\label{lem-z-erg}
Suppose that $\theta$ is irrational.
Let $H$ be a non-trivial subgroup of $E$.
Then the following assertions hold:
\begin{enumerate}
\item We identify $E$ with $\mathbb{Z}$ through the isomorphism sending $a$ to $1$.
Then for a.e.\ $y\in Y_0$, the function $\beta(\cdot, y)$ from $H$ into $E$ is non-decreasing if $\theta$ is positive, and is non-increasing if $\theta$ is negative.
Moreover, for a.e.\ $y\in Y_0$, the function $\beta(\cdot, y)$ is bounded neither below nor above. 
\item If the action of $E$ on $(Z, \xi)$ is mixing, then the actions of $H$ on $X$ and on $Y$ are both ergodic.
\end{enumerate}
\end{lem}

\begin{proof}
For any $x\in [0, |\theta|)$, any $k\in K$ and any $n\in \mathbb{Z}$, the integer $m$ satisfying the equation $\beta(a^n, (x, k))=a^m$ is determined by the condition $x-n+\theta m\in [0, |\theta|)$.
Assertion (i) follows.

In assertion (ii), ergodicity of the action of $H$ on $X$ follows from Lemma \ref{lem-dense} and the assumption that the action of $E$ on $(Z, \xi)$ is mixing.
We show that the action of $H$ on $Y$ is ergodic in the case where $\theta$ is positive.
The proof of the other case is similar and is thus omitted.
Let $\nu_0$ denote the measure on $Y_0$.
We may assume that $\nu_0$ is a probability measure, and define the probability measure $\nu =\xi \times \nu_0$ on $Y$.
Let $b$ be the generator of $H$ that is a positive power of $a$.
It is enough to show that for any Borel subsets $A_1, A_2\subset Z$ and $B_1, B_2\subset Y_0$, we have
\[\frac{1}{n}\sum_{k=1}^n\nu(b^k(A_1\times B_1)\cap (A_2\times B_2))\to \nu(A_1\times B_1)\nu(A_2\times B_2)\]
as $n$ goes to infinity (see \cite[I.2.6 (iii)]{bm}).
For non-negative integers $k$, $m$, we set
\[C_k^m=\{\, y\in B_1\mid \beta(b^k, y)=a^m\,\}.\]
For any positive integer $k$, we have the decomposition $B_1=\bigsqcup_{m=0}^{\infty}C_k^m$ and thus
\[b^k(A_1\times B_1)=\bigsqcup_{m=0}^{\infty}(a^mA_1)\times (b^kC_k^m).\]
Pick a positive real number $\varepsilon$.
Since the action of $E$ on $(Z, \xi)$ is mixing, there exists a positive integer $M_1$ such that for any integer $m$ with $m\geq M_1$, we have
\[|\xi(a^m A_1\cap A_2)-\xi(A_1)\xi(A_2)|<\varepsilon.\]
By assertion (i), there exists a positive integer $M_2$ such that for any integer $k$ with $k\geq M_2$, we have $\sum_{m=0}^{M_1-1}\nu_0(C_k^m)<\varepsilon$.
For any integer $k$ with $k\geq M_2$, we have
\begin{align*}
&\, \left|\,\nu(b^k(A_1\times B_1)\cap (A_2\times B_2))-\xi(A_1)\xi(A_2)\nu_0(b^kB_1\cap B_2)\,\right|\\
= & \,\left|\, \sum_{m=0}^{\infty}\xi(a^mA_1\cap A_2)\nu_0(b^kC_k^m\cap B_2)-\xi(A_1)\xi(A_2)\nu_0(b^kB_1\cap B_2)\,\right|\\
\leq & \,\varepsilon +\left|\, \sum_{m=M_1}^{\infty}\xi(a^mA_1\cap A_2)\nu_0(b^kC_k^m\cap B_2)-\xi(A_1)\xi(A_2)\nu_0(b^kB_1\cap B_2)\,\right|\\
\leq & \,2\varepsilon +\xi(A_1)\xi(A_2)\left|\, \sum_{m=M_1}^{\infty}\nu_0(b^kC_k^m\cap B_2)-\nu_0(b^kB_1\cap B_2)\,\right| \leq 3\varepsilon.
\end{align*}
For any sufficiently large integer $n$ with $n\geq M_2$ and $n\geq M_2/\varepsilon$, we have
\begin{align*}
&\, \left|\,\frac{1}{n}\sum_{k=1}^n\nu(b^k(A_1\times B_1)\cap (A_2\times B_2))-\nu(A_1\times B_1)\nu(A_2\times B_2)\,\right|\\
\leq & \, \left|\,\frac{1}{n}\sum_{k=1}^{M_2-1}\nu(b^k(A_1\times B_1)\cap (A_2\times B_2))\,\right|\\
& +\left|\,\frac{1}{n}\sum_{k=M_2}^n\nu(b^k(A_1\times B_1)\cap (A_2\times B_2))-\xi(A_1)\xi(A_2)\nu_0(B_1)\nu_0(B_2)\,\right|\\
\leq & \, \varepsilon +3\varepsilon + \xi(A_1)\xi(A_2)\left|\,\frac{1}{n}\sum_{k=M_2}^n \nu_0(b^kB_1\cap B_2)-\nu_0(B_1)\nu_0(B_2)\,\right|\\
\leq & \, 4\varepsilon + \xi(A_1)\xi(A_2)\frac{1}{n}\sum_{k=1}^{M_2-1}\nu_0(b^kB_1\cap B_2)\\
& +\xi(A_1)\xi(A_2)\left|\, \frac{1}{n}\sum_{k=1}^n\nu_0(b^kB_1\cap B_2)-\nu_0(B_1)\nu_0(B_2)\,\right|\\
\leq &\, 4\varepsilon +\varepsilon +\varepsilon =6\varepsilon,
\end{align*}
where we use ergodicity of the action of $H$ on $Y_0$ in the last inequality, which follows from Lemma \ref{lem-dense}.
\end{proof}

Combining Lemmas \ref{lem-z-conj} and \ref{lem-z-erg}, we obtain the following:

\begin{thm}\label{thm-sigma}
If $\theta$ is irrational and the action $\Gamma \c (Z, \xi)$ is mixing, then the $(\Gamma, \Gamma)$-coupling $(\Sigma, m)$ satisfies the following three properties:
\begin{itemize}
\item The action $\Gamma \times \Gamma \c (\Sigma, m)$ is essentially free.
\item The two actions $\mathsf{L}(\Gamma)\c \Sigma/\mathsf{R}(\Gamma)$ and $\mathsf{R}(\Gamma)\c \Sigma/\mathsf{L}(\Gamma)$ are not conjugate. 
\item For any non-trivial elliptic subgroup $H$ of $\Gamma$, the actions $\mathsf{L}(H)\c \Sigma /\mathsf{R}(\Gamma)$ and $\mathsf{R}(H)\c \Sigma/\mathsf{L}(\Gamma)$ are both ergodic.
\end{itemize}
\end{thm}

Theorem \ref{thm-woe-non-conj} is a direct consequence of this theorem.

\begin{rem}\label{rem-ber}
In contrast with Theorem \ref{thm-woe-non-conj}, we obtain rigidity of Bernoulli shifts of Baumslag-Solitar groups as an application of Popa's cocycle superrigidity theorem \cite{popa-gap}.
Let $\Gamma$ be a discrete group.
An infinite subgroup $H$ of $\Gamma$ is called {\it wq-normal} in $\Gamma$ if for any subgroup $H_1$ with $H<H_1<\Gamma$ and $H_1\neq \Gamma$, there exists an element $\gamma \in \Gamma \setminus H_1$ with $\gamma H_1\gamma^{-1}\cap H_1$ infinite.
This definition is different from that in \cite{popa-gap}.
Equivalence between these two definitions is discussed in \cite[Definition 2.3]{popa-coh}.

\begin{lem}\label{lem-wq}
We set $\Gamma =\bs(p, q)=\langle\, a, t\mid ta^pt^{-1}=a^q\,\rangle$ with $2\leq |p|\leq |q|$.
Then $\langle a^q\rangle$ is wq-normal in $\Gamma$, and the centralizer of $\langle a^q\rangle$ in $\Gamma$ is non-amenable.
\end{lem}

\begin{proof}
Let $H_1$ be a proper subgroup of $\Gamma$ containing $\langle a^q\rangle$. Either $a$ or $t$ does not belong to $H_1$.
Both $aH_1a^{-1}\cap H_1$ and $tH_1t^{-1}\cap H_1$ are infinite because we have $a\langle a^q\rangle a^{-1}=\langle a^q\rangle$ and $\langle a^{q^2}\rangle <t\langle a^q\rangle t^{-1}\cap \langle a^q\rangle$.
The former assertion of the lemma is proved.
If $|p|=|q|$, then the centralizer of $\langle a^q\rangle$ in $\Gamma$ contains $a$, $tat^{-1}$ and $t^2$ and is of finite index in $\Gamma$.
If $|p|<|q|$, then the centralizer of $\langle a^q\rangle$ in $\Gamma$ contains $a$ and $tat^{-1}$ and is non-amenable.
The latter assertion of the lemma follows.
\end{proof}

For a discrete group $\Gamma$ and a standard probability space $(X_0, \mu_0)$ such that there is no point of $X_0$ whose measure is equal to 1, the {\it Bernoulli shift} $\Gamma \c (X_0, \mu_0)^{\Gamma}$ is defined by the formula $\gamma (x_{\delta})_{\delta \in \Gamma}=(x_{\gamma^{-1}\delta})_{\delta \in \Gamma}$ for $\gamma \in \Gamma$ and $(x_{\delta})_{\delta \in \Gamma}\in X_0^{\Gamma}$.
Lemma \ref{lem-wq} and \cite[Corollary 1.2]{popa-gap} imply the following rigidity.

\begin{thm}\label{thm-ber}
Set $\Gamma =\bs(p, q)$ with $2\leq |p|\leq |q|$.
Let $(X_0, \mu_0)$ be a standard probability space such that there is no point of $X_0$ whose measure is equal to $1$.
If the Bernoulli shift $\Gamma \c (X_0, \mu_0)^{\Gamma}$ is WOE to an ergodic f.f.m.p.\ action $\Lambda \c (Y, \nu)$ of a discrete group $\Lambda$, then there exist a finite index subgroup $\Lambda_0$ of $\Lambda$ and a $\Lambda_0$-invariant Borel subset $Y_0$ of $Y$ satisfying the following three conditions:
\begin{itemize}
\item The equality $\nu(Y_0)/\nu(Y)=[\Lambda :\Lambda_0]^{-1}$ holds.
\item The action $\Lambda \c (Y, \nu)$ is induced from the action $\Lambda_0\c (Y_0, \nu|_{Y_0})$.
\item The actions $\Gamma \c (X_0, \mu_0)^{\Gamma}$ and $\Lambda_0\c (Y_0, \nu|_{Y_0})$ are conjugate.
\end{itemize}
\end{thm}

In \cite[Corollary 1.7]{bowen}, Bowen shows that the entropy of the base space $(X_0, \mu_0)$ is a complete conjugacy-invariant of Bernoulli shifts of a sofic Ornstein group.
The group $\bs(p, q)$ is sofic (see \cite[Example 4.6]{pestov}) and is also Ornstein because it has an infinite amenable subgroup which is Ornstein (see \cite[p.218]{bowen}).
Theorem \ref{thm-ber} implies that the entropy of $(X_0, \mu_0)$ is also a complete WOE-invariant for Bernoulli shifts of $\bs(p, q)$.
\end{rem}



\appendix

\section{Ergodic components for elliptic subgroups}\label{app}

We set $\Gamma =\bs(p, q)=\langle \, a, t\mid ta^pt^{-1}=a^q\,\rangle$ with $2\leq |p|<|q|$.
Let $d_0>0$ denote the greatest common divisor of $p$ and $q$, and put $p_0=p/d_0$ and $q_0=q/d_0$.
We assume that $q$ is not a multiple of $p$.
We thus have $1<|p_0|<|q_0|$.
In this appendix, given an ergodic measure-preserving action of $\Gamma$ on a standard probability space, we discuss ergodicity of its restriction to an elliptic subgroup of $\Gamma$.

\begin{lem}\label{lem-atom}
Let $(X, \mu)$ be a standard probability space.
Suppose that we have an ergodic measure-preserving action of $\Gamma$ on $(X, \mu)$.
Let $\theta \colon (X, \mu)\to (Z, \xi)$ be the ergodic decomposition for the action of $\langle a^{d_0}\rangle$ on $(X, \mu)$.
We define $Z_0$ as the subset of all points of $Z$ with positive measure.
Then either $\xi(Z_0)=0$ or $\xi(Z\setminus Z_0)=0$.
\end{lem}

\begin{proof}
We may assume that $\xi(Z_0)=\mu(\theta^{-1}(Z_0))$ is positive.
It is enough to show that $\theta^{-1}(Z_0)$ is $\Gamma$-invariant.
Pick $\gamma \in \Gamma$.
Pick an ergodic component $X_0$ for the action $\langle a^{d_0} \rangle \c (X, \mu)$ which is contained in $\theta^{-1}(Z_0)$.
Choose a positive integer $n$ with $\gamma a^{nd_0}\gamma^{-1}\in \langle a^{d_0} \rangle$.
Let $X_1$ be an ergodic component for the action of $\langle a^{nd_0}\rangle$ on $X_0$.
Since $\gamma X_1$ is an ergodic component for the action of $\langle \gamma a^{nd_0}\gamma^{-1}\rangle$ on $X$ and has positive measure with respect to $\mu$, we have $\gamma X_1\subset \theta^{-1}(Z_0)$.
It follows that $\gamma X_0\subset \theta^{-1}(Z_0)$ and thus $\gamma \theta^{-1}(Z_0)\subset \theta^{-1}(Z_0)$.
\end{proof}

Thanks to Lemma \ref{lem-atom}, ergodic measure-preserving actions of $\Gamma$ on a standard probability space $(X, \mu)$ have the following dichotomy:
Either almost every ergodic component for the action of $\langle a \rangle$ on $(X, \mu)$ has positive measure with respect to $\mu$, or almost every ergodic component for the action of $\langle a \rangle$ on $(X, \mu)$ has zero measure with respect to $\mu$.
The following lemma focuses on actions with the former property.

\begin{lem}\label{lem-erg-comp}
In the notation in Lemma \ref{lem-atom}, if $\xi(Z\setminus Z_0)=0$, then for any non-negative integers $k$, $l$ and any $z\in Z_0$, $\theta^{-1}(z)$ is an ergodic component for the action of $\langle a^{d_0p_0^kq_0^l}\rangle$ on $(X, \mu)$.
\end{lem}

To prove this lemma, we use the following elementary observation on the ergodic decomposition for a power of an ergodic transformation.

\begin{lem}\label{lem-erg-comp-number}
Let $r$ and $s$ be positive integers.
Suppose that we have an ergodic measure-preserving action of the group $\mathbb{Z}$ on a standard finite measure space $(Y, \nu)$.
For non-negative integers $k$, $l$, we define $\lambda_{k, l}\colon (Y, \nu)\to (W_{k, l}, \omega_{k, l})$ as the ergodic decomposition for the action of the subgroup $r^ks^l\mathbb{Z}$ on $(Y, \nu)$, and assume that any point of $W_{k, l}$ has positive measure.
Then for any non-negative integers $k$, $l$,
\begin{enumerate}
\item both $|W_{k+1, l}|/|W_{k, l}|$ and $|W_{k, l+1}|/|W_{k, l}|$ are integers; and
\item $|W_{k+1, l}|/|W_{k, l}|$ is a divisor of $r$, and $|W_{k, l+1}|/|W_{k, l}|$ is a divisor of $s$.
\end{enumerate}
\end{lem}

\begin{proof}[Proof of Lemma \ref{lem-erg-comp}]
For non-negative integers $k$, $l$, we set $E_{k, l}=\langle a^{d_0p_0^kq_0^l}\rangle$.
Let $X_0$ be an ergodic component for the action of $E_{0, 0}$ on $X$ with $\mu(X_0)>0$.
Choose a positive integer $n$ with $\mu(t^nX_0\cap X_0)>0$.
We put $A=X_0\cap t^{-n}X_0$.

Assuming that there exists a positive integer $k$ such that the action of $E_{k, 0}$ on $X_0$ is not ergodic, we deduce a contradiction.
Let $k$ be the minimal integer such that $k\geq n$ and the action of $E_{k, 0}$ on $X_0$ is not ergodic.
Let $X_1$ be an ergodic component for the action of $E_{k, 0}$ on $X_0$ with $\mu(X_1\cap A)>0$.
The number $\mu(X_0)/\mu(X_1)$ is then a divisor of $p_0^k$ bigger than $1$.
Since $t^nA\subset X_0$, we have $t^n(X_1\cap A)\subset X_0$.
For any $x\in X_1\cap A$ and any integer $r$, we have $t^na^{d_0p_0^kr}x=a^{d_0p_0^{k-n}q_0^nr}t^nx$.
It follows that $t^nX_1\subset \langle a^{d_0}\rangle X_0=X_0$.
The equality $t^na^{d_0p_0^k}t^{-n}=a^{d_0p_0^{k-n}q_0^n}$ implies that $t^nX_1$ is an ergodic component for the action of $E_{k-n, n}$ on $X_0$.
The number $\mu(X_0)/\mu(t^nX_1)$ is therefore a divisor of $p_0^{k-n}q_0^n$.
If there were a divisor of $\mu(X_0)/\mu(t^nX_1)$ which is a divisor of $p_0$ bigger than $1$, then we would have $k-n\geq 1$, and the action of $E_{k-n, 0}$ on $X_0$ would not be ergodic by Lemma \ref{lem-erg-comp-number}.
It follows that $k-1\geq n$ and the action of $E_{k-1, 0}$ on $X_0$ is not ergodic.
This contradicts the minimality of $k$.
The number $\mu(X_0)/\mu(t^nX_1)$ is therefore a divisor of $q_0^n$.
On the other hand, we have the equality $\mu(X_0)/\mu(t^nX_1)=\mu(X_0)/\mu(X_1)$.
This is a contradiction.

We have shown that for any positive integer $k$, the action of $E_{k, 0}$ on $X_0$ is ergodic.
Along a similar argument, we can show that for any positive integer $l$, the action of $E_{0, l}$ on $X_0$ is ergodic.
The lemma now follows from Lemma \ref{lem-erg-comp-number}. 
\end{proof}


\section{Exotic amenable normal subgroupoids}\label{sec-exotic}

We set $\Gamma =\bs(p, q)$ with $2\leq |p|<|q|$ and the presentation $\Gamma =\langle\, a, t\mid ta^pt^{-1}=a^q\,\rangle$, and set $E=\langle a\rangle$.
Let $(X, \mu)$ be a standard finite measure space.
For a certain measure-preserving action of $\Gamma$ on $(X, \mu)$, we show that the subgroupoid $E\ltimes X$ is normal in $\Gamma \ltimes X$.
We also discuss the associated quotient groupoid.
The existence of such a subgroupoid is contrast to the property that the trivial group is the only amenable normal subgroup of $\Gamma$.
In fact, if there were an infinite amenable normal subgroup $N$ of $\Gamma$, then $N$ would be elliptic by Theorem \ref{thm-ell}, but it would be non-amenable by the presentation of $\Gamma$.

For non-zero integers $d$, $m$ and $n$, we say that a measure-preserving action of the group $\mathbb{Z}$ on $(X, \mu)$ is {\it $(d; m, n)$-periodic} if for any non-negative integers $k$, $l$, there exists a $\mathbb{Z}$-equivariant Borel map from $X$ into $\mathbb{Z}/(dm^kn^l\mathbb{Z})$.

\begin{prop}\label{prop-normal}
Let $d_0>0$ denote the greatest common divisor of $p$ and $q$, and put $p_0=p/d_0$ and $q_0=q/d_0$.
Let $\Gamma \c (X, \mu)$ be a measure-preserving action such that the action of $E$ is $(d_0; p_0, q_0)$-periodic.
We set $\cal{G}=\Gamma \ltimes X$ and $\cal{E}=E\ltimes X$.
Then $\cal{E}$ is normal in $\cal{G}$.
\end{prop}

\begin{proof}
For non-negative integers $k$, $l$, we set $E_{k, l}=\langle a^{d_0p_0^kq_0^l}\rangle$.
Pick $\gamma \in \Gamma$.
There exist non-negative integers $i$, $j$, $k$ and $l$ with $E_{i, j}=E\cap\gamma^{-1}E\gamma$ and $E_{k, l}=\gamma E_{i, j}\gamma^{-1}$.
Let $A$ be an $E_{i, j}$-invariant Borel subset of $X$ with
\[X=A_0\sqcup A_1\sqcup \cdots \sqcup A_{d_0|p_0|^i|q_0|^j-1}\]
up to null sets, where we set $A_m=a^mA$ for each $m\in \mathbb{Z}$.
Similarly, let $B$ be an $E_{k, l}$-invariant Borel subset of $X$ with
\[X=B_0\sqcup B_1\sqcup \cdots \sqcup B_{d_0|p_0|^k|q_0|^l-1}\]
up to null sets, where we set $B_m=a^mB$ for each $m\in \mathbb{Z}$.
We define $X=\bigsqcup_{n\in N}C_n$ as the partition of $X$ generated by the partition of $X$ into $\{ A_m\}_m$ and that into $\{ \gamma^{-1}B_{m'}\}_{m'}$.
Namely, it is the partition of $X$ so that for each $n\in N$, $C_n$ is of the form $A_m\cap \gamma^{-1}B_{m'}$.
The index set $N$ is finite.

Define an automorphism $U$ of $\cal{G}$ by $U(\delta, x)=(\gamma \delta \gamma^{-1}, \gamma x)$ for $(\delta, x)\in \cal{G}$.
For any $m\in \mathbb{Z}$, we have the equalities $(E_{i, j}\ltimes X)_{A_m}=(\cal{E})_{A_m}$ and $(E_{k, l}\ltimes X)_{B_m}=(\cal{E})_{B_m}$.
It follows that for any $n\in N$, we have the equality
\[U((\cal{E})_{C_n})=U((E_{i, j}\ltimes X)_{C_n})=(E_{k, l}\ltimes X)_{\gamma C_n}=(\cal{E})_{\gamma C_n}.\]
The map from $C_n$ to $\cal{G}$ sending each element $x$ of $C_n$ to $(\gamma, x)$ thus lies in ${\rm N}_{\cal{G}}(\cal{E})$.
\end{proof}

Let $\cal{G}$ be an ergodic discrete measured groupoid on $(X, \mu)$ with $r, s\colon \cal{G}\to X$ the range and source maps of $\cal{G}$, respectively.
Let $\delta \colon \cal{G}\to \Rm$ be the {\it Radon-Nikodym cocycle} defined by the equality
\[\mu(r\circ \phi(A))=\int_A\delta(\phi(x))\, d\mu(x)\]
for any $\phi \in [[\cal{G}]]$ and any Borel subset $A$ of $D_{\phi}$.
The cocycle $\delta$ factors through the quotient equivalence relation of $\cal{G}$ defined as $\{\, (r(g), s(g))\in X\times X\mid g\in \cal{G}\,\}$.
The Mackey range of the cocycle $\log \circ \delta \colon \cal{G}\to \R$ is an ergodic non-singular action of $\mathbb{R}$ on a standard Borel space with a $\sigma$-finite measure, and is called the {\it flow associated with} $\cal{G}$.
We say that
\begin{itemize}
\item $\cal{G}$ is of {\it type} ${\rm III}$ if the flow associated with $\cal{G}$ is not isomorphic to the action of $\mathbb{R}$ on itself by addition; 
\item $\cal{G}$ is of {\it type} ${\rm III}_{\lambda}$, where $\lambda$ is a real number with $0<\lambda <1$, if the flow associated with $\cal{G}$ is isomorphic to the action of $\mathbb{R}$ on $\mathbb{R}/(-\log \lambda)\mathbb{Z}$ by addition;
\item $\cal{G}$ is of {\it type} ${\rm III}_0$ if the flow associated with $\cal{G}$ is recurrent and any orbit of it is of measure zero; and
\item $\cal{G}$ is of {\it type} ${\rm III}_1$ if the flow associated with $\cal{G}$ is isomorphic to the trivial action of $\mathbb{R}$ on a single point.
\end{itemize}
Note that $\cal{G}$ is of type ${\rm III}$ if and only if $\cal{G}$ is of type ${\rm III}_{\lambda}$ for some $\lambda$ with $0\leq \lambda \leq 1$.
We refer to \cite{ho} for fundamental results on orbit equivalence relations of type ${\rm III}$.

As introduced in Remark \ref{rem-type}, every ergodic measure-preserving action of $\Gamma$ on $(X, \mu)$ is classified into those of type $\lambda$ with $\lambda \in \{ |p/q|^n\}_{n\in \mathbb{Z}_{>0}}\cup \{ 0\}$.
Under the assumption in Proposition \ref{prop-normal}, we have the following relationship between the type of the action $\Gamma \c (X, \mu)$ and the type of the quotient $\cal{G}/\cal{E}$.

\begin{prop}\label{prop-3}
In the notation in Proposition \ref{prop-normal}, suppose that the action $\Gamma \c (X, \mu)$ is ergodic.
If the action $\Gamma \c (X, \mu)$ is of type $\lambda$ in the sense of Remark \ref{rem-type}, then the quotient $\cal{G}/\cal{E}$ is of type ${\rm III}_{\lambda}$.
In particular, $\cal{G}/\cal{E}$ is of type ${\rm III}$.
\end{prop}

\begin{proof}
Put $\cal{Q}=\cal{G}/\cal{E}$ and let $\theta \colon \cal{G}\to \cal{Q}$ denote the quotient homomorphism.
Let $\delta \colon \cal{Q}\to \Rm$ denote the Radon-Nikodym cocycle for $\cal{Q}$.
We define $\D \colon \cal{G}\to \Rm$ as $\D(\cal{G}, \cal{E})$, the modular cocycle of Radon-Nikodym type defined in Section \ref{subsec-rn}.
We also define $\I \colon \cal{G}\to \Rm$ as $\I(\cal{G}, \cal{E})$, the local-index cocycle defined in Section \ref{subsec-lic}.
Let $\m \colon \Gamma \to \Rm$ be the modular homomorphism.
The last equality obtained in the proof of Proposition \ref{prop-normal} implies that $\I(\gamma, x)=1$ for any $\gamma \in \Gamma$ and a.e.\ $x\in X$.
By Corollary \ref{cor-bs-comp}, we have $\m(\gamma)=\D(\gamma, x)$ for any $\gamma \in \Gamma$ and a.e.\ $x\in X$.
To prove the proposition, it thus suffices to show the equality $\delta \circ \theta(g)=\D(g)$ for a.e.\ $g\in \cal{G}$.

Let $(Z, \xi)$ be the unit space of $\cal{Q}$.
Recall that the induced map $\theta \colon (X, \mu)\to (Z, \xi)$ is the ergodic decomposition for $\cal{E}$.
We fix $\gamma \in \Gamma$, and set $E_-=E\cap \gamma^{-1}E\gamma$ and $E_+=E\cap \gamma E\gamma^{-1}$.
Let
\[\theta_-\colon (X, \mu)\to (Z_-, \xi_-)\quad \textrm{and}\quad \theta_+\colon (X, \mu)\to (Z_+, \xi_+)\]
be the ergodic decompositions for the actions of $E_-$ and $E_+$, respectively.
We have the Borel maps $\sigma_-\colon Z_-\to Z$ and $\sigma_+\colon Z_+\to Z$ with $\theta =\sigma_-\circ \theta_-=\sigma_+\circ \theta_+$.
For a.e.\ $x\in X$, we have $\theta(\gamma, ax)=\theta(\gamma, ax)\theta(a, x)=\theta(\gamma, x)$.
The map $\theta(\gamma, \cdot)$ from $X$ into $\cal{Q}$ therefore induces the Borel map $\phi \colon Z\to \cal{Q}$ with $\phi(\theta(x))=\theta(\gamma, x)$ for a.e.\ $x\in X$.

Let $Y$ be a Borel subset of $Z$ such that the restriction of $r\circ \phi$ to $Y$ is injective, where $r$ denotes the range map of $\cal{Q}$.
Pick a Borel subset $Y_-$ of $Z_-$ such that $\sigma_-$ induces an isomorphism from $Y_-$ onto $Y$.
The element $\gamma$ induces an isomorphism from $(Z_-, \xi_-)$ onto $(Z_+, \xi_+)$.
Define a Borel subset $Y_+$ of $Z_+$ as $Y_+=\gamma Y_-$, and set $A=\theta_-^{-1}(Y_-)$.
The equality $\gamma A=\theta_+^{-1}(Y_+)$ holds, and $\sigma_+$ induces an isomorphism from $Y_+$ onto $r\circ \phi(Y)$ because the restriction of $r\circ \phi$ to $Y$ is injective.
Since the action $E\c (X, \mu)$ is $(d_0; p_0, q_0)$-periodic, we have $\mu(A)=\xi(Y)[E: E_-]^{-1}$ and $\mu(\gamma A)=\xi(r \circ \phi(Y))[E: E_+]^{-1}$.
The equality $\mu(A)=\mu(\gamma A)$ implies
\[\xi(r\circ \phi(Y))=\frac{[E: E_+]}{[E: E_-]}\xi(Y)=\m(\gamma)\xi(Y).\]
It therefore follows that $\delta \circ \theta(\gamma, x)=\m(\gamma)$ for a.e.\ $x\in X$.
\end{proof}


\end{document}